\documentclass{article}
\usepackage{syntonly,indentfirst,mathrsfs,latexsym,float}
\usepackage{bm,amsmath,amsbsy,amsthm,amssymb,amsfonts}
\usepackage{array}
\usepackage{multirow}
\usepackage{mathrsfs}
\usepackage{type1cm}
\usepackage{wasysym}
\usepackage{cite,fancyhdr,dsfont,enumerate,graphicx,color,stmaryrd}
\usepackage[colorlinks,linkcolor=black,citecolor=blue]{hyperref}
\usepackage[capitalise, nosort]{cleveref}

\usepackage{indentfirst}
\usepackage{stmaryrd}
\usepackage{subfigure}
\usepackage{graphicx}
\usepackage[utf8]{inputenc}
\usepackage{pgf,tikz}
\usepackage{threeparttable}
\usetikzlibrary{arrows}
\numberwithin{equation}{section}
\newcommand{\tabcaption}{\def\@captype{table}\caption}
\newtheorem{lem}{Lemma}[section]
\newtheorem{thm}{Theorem}[section]

\newtheorem{rem}{Remark}[section]

\newtheorem{exmp}{Example}[section]

\title{An eXtended HDG method for  Darcy-Stokes-Brinkman interface problems \thanks
  {
    This work was supported in part by   National Natural Science Foundation of
    China (11771312).
  }
}

\author{
	Yihui Han\thanks{South China Research Center for Applied Mathematics and Interdisciplinary Studies, South China Normal University, Guangzhou, 510630, China, Email: yhhan@m.scnu.edu.cn},
	Xiao-Ping Wang\thanks{Department of Mathematics, The Hong Kong University of Science and Technology, Clear Water Bay, Kowloon, Hong Kong, China, Email: mawang@ust.hk},
    Xiaoping Xie \thanks{Corresponding author. School of Mathematics, Sichuan University, Chengdu 610064, China, Email: xpxie@scu.edu.cn}
 }

\date{}

\begin{document}
	\maketitle
\begin{abstract}
This paper proposes an interface/boundary-unfitted  eXtended hybridizable discontinuous Galerkin (X-HDG) method for Darcy-Stokes-Brinkman interface problems in two and three dimensions. The method uses piecewise linear polynomials for the velocity approximation and piecewise constants for both the velocity gradient and pressure approximations in the interior of elements inside the subdomains separated by the interface, uses piecewise constants for the numerical traces of velocity on  the inter-element boundaries inside the subdomains, and uses  piecewise constants or linear polynomials for the numerical traces of velocity on the interface. Optimal error estimates  are derived for the interface-unfitted X-HDG scheme.
 Numerical experiments are provided to verify the theoretical results and the robustness of the proposed method.
		
	$\bf{Key}$ $\bf{Words}$: eXtended HDG method, Darcy-Stokes-Brinkman interface   problem, interface/boundary-unfitted mesh,   error estimate.
\end{abstract}

\section{Introduction} 
Let  $\Omega \subset \mathbb{R}^d (d = 2, 3) $   be  a bounded domain 
divided  into two subdomains, $\Omega_i (i = 1,2)$,   by a piecewise smooth interface $\Gamma$ (cf. Figure \ref{domain}). We consider the following Darcy-Stokes-Brinkman interface problem:  find the velocity $\bm{u}$ and the pressure $p$ such that
\begin{align}
\label{pb1}
\left \{
\begin{array}{rl}
- \nabla\cdot(\nu\nabla \bm{u})+\nabla p + \alpha\bm{u}= \bm{f}&   \text{in} \quad\Omega_1\cup\Omega_2,\\ 
\nabla\cdot \bm{u} = 0&   \text{in} \quad\Omega_1\cup\Omega_2,\\ 
\bm{u}=\bm{g}_D&  \text{on} \quad \partial\Omega,\\
\llbracket \bm{u} \rrbracket= \bm{0}, \ \llbracket(\nu \nabla \bm{u}-p\bm{I})\bm{n}\rrbracket = \bm{g}_N^{\Gamma} & \text{on}\quad \Gamma.
\end{array}
\right.
\end{align}
Here   the viscosity coefficient $\nu$ and the  zeroth-order term coefficient $\alpha$ are piecewise constants with
\begin{align}
\nu|_{\Omega_i} = \nu_i >0,\quad  
\alpha|_{\Omega_i} = \alpha_i\geq 0,\quad i=1,2. 
\end{align}   
The jump of a function $w$ across the interface $\Gamma$ is defined by $ \llbracket  w  \rrbracket := (w|_{\Omega_{1}})|_\Gamma-(w|_{\Omega_{2}})|_\Gamma$, $\bm{I}$ the $d\times d$  identity matrix, and $\bm{n}$ denotes the unit normal vector along $\Gamma$,   pointing to $\Omega_{2}$. $\bm{f}$  denotes  the   body force,   $\bm{g}_N^{\Gamma}$ the   interface traction, and  $\bm{g}_D$  the source term satisfying
\begin{align}\label{13-gD}
\int_{\partial \Omega}\bm{g}_D\cdot\bm{n} = 0,
\end{align}
where $\bm{n}$ is the outward unit normal vector along $\partial \Omega$.  The Darcy-Stokes-Brinkman  model \eqref{pb1}  is usually used to describe  porous media flow coupled with viscous fluid  flow in a single form of equation (cf. \cite{Nield;2006,  xie2008uniformly,konno2011h,Johnny2012A,gatica2014analysis}).  


\begin{figure}[htp]	
	\centering
	\includegraphics[height = 5 cm,width= 5.5 cm]{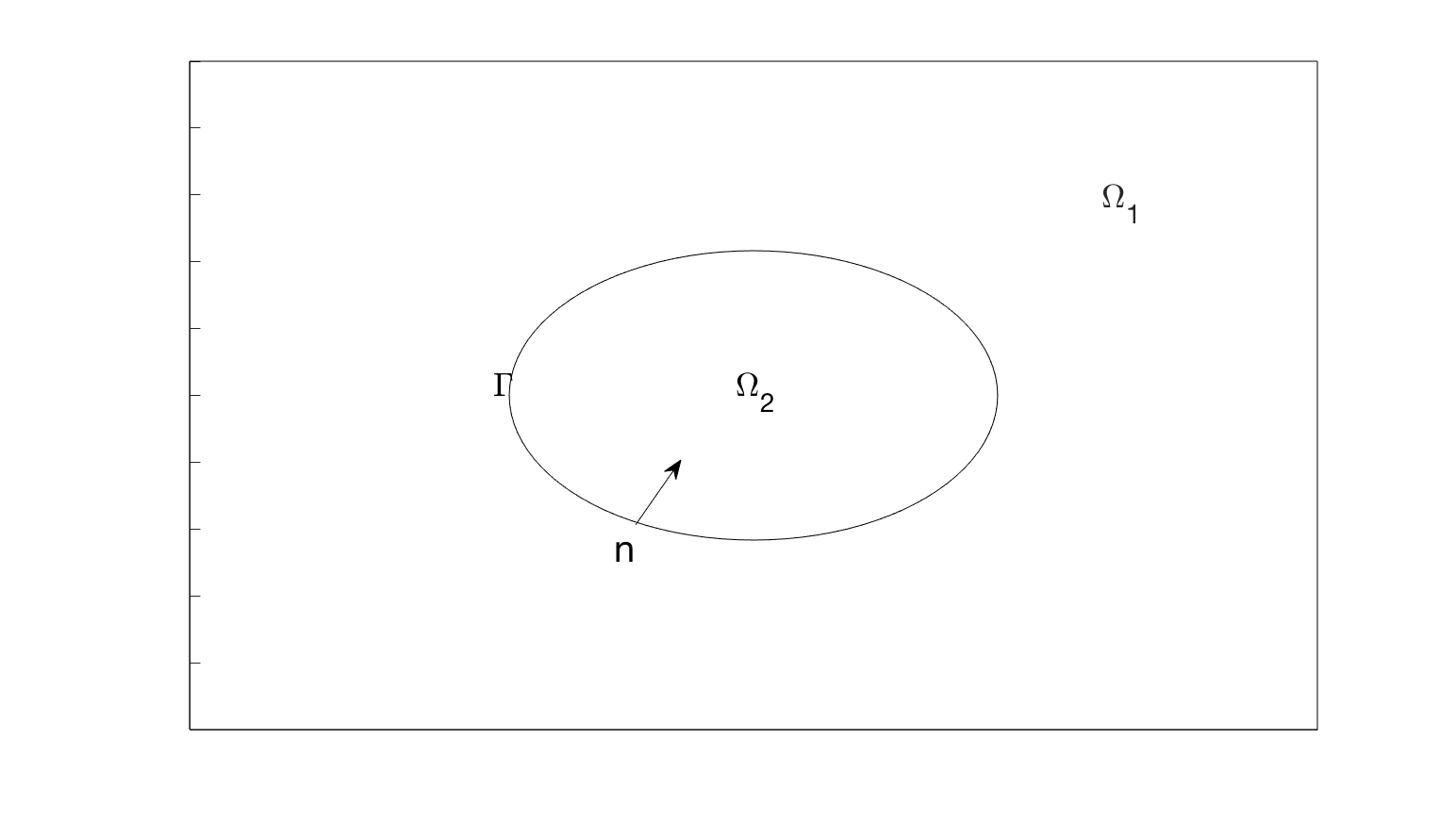} 
	\includegraphics[height = 5 cm,width= 5.5 cm]{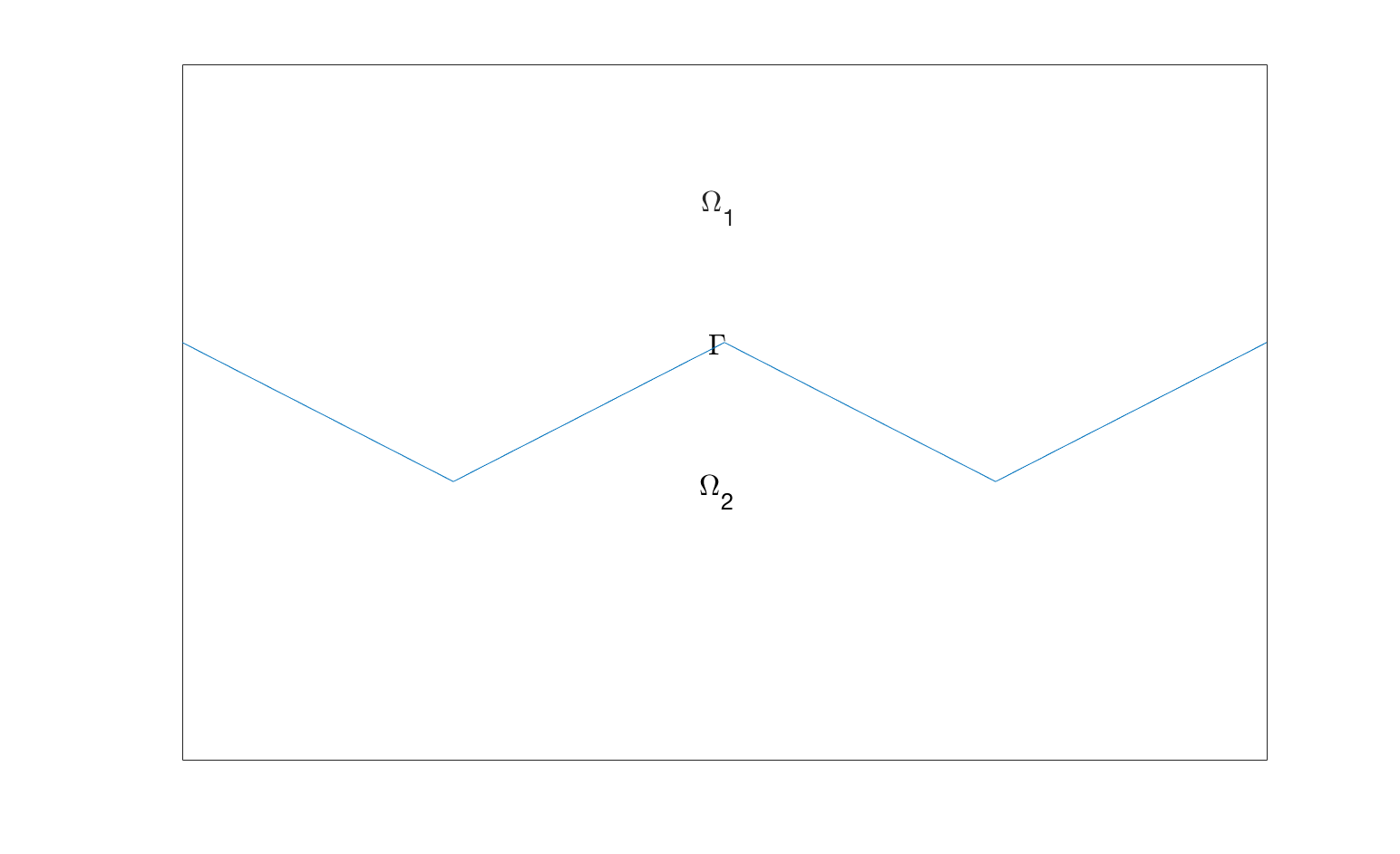} 
	\caption{The geometry of domain with circle interface or fold line interface}\label{domain}
\end{figure}

For an elliptic interface problem, the low regularity of the  solution  due to  the coefficient discontinuity   may result in reduced accuracy of finite element discretization  \cite{babuvska1970finite,xu2013estimate}.  One strategy   for this situation is to use interface(or body)-fitted meshes (cf. Figure  \ref{Fitted and unfitted mesh}) so as  to dominate the approximation error caused by the non-smoothness of solution  \cite{barrett1987fittedandunfitted, Brambles96fin, Chen1998Finite, Huang2002Some, Plum2003Optimal, Cai2017Discontinuous}. However, the generation of interface-fitted meshes is usually expensive, especially when the interface is of complicated geometry or moving with time or iteration. 

Another strategy avoiding the loss of numerical accuracy is to use certain modification of the finite element approximation around the interface. The resultant finite element methods do not need interface-fitted meshes (cf. Figure  \ref{Fitted and unfitted mesh}). One representative of such interface-unfitted methods is the  eXtended/Generalized Finite Element Method (XFEM or GFEM),  where additional basis functions   characterizing the singularity of the solution around the interface are enriched into the corresponding approximation space. We refer to \cite{Thomas2010The} for an overview work and  \cite{babuvska1994special,belytschko2009review,babuska2010Optimal,Burman2012Fictitious,hansbo2002unfitted,wangtao2018nitsche,yangcc2018interface} for some developments of   XFEM/GFEM. In particular, we  refer to \cite{Cattaneo2015Stabilized,chernyshenko2020an,hansbo2014cut,hansbo2017a,KirchhartXFEm, wang2015new} for several XFEMs using  additional cut basis functions for Stokes or Darcy interface problems. It should be pointed out that the immersed finite element method (IFEM) is another type of interface-unfitted method,  where   special finite element basis  functions  are constructed to satisfy the interface jump conditions (cf. \cite{Adjerid2015An,adjerid2019an,
lizhilin1998immersed,lizhilin2006immersed,lin2015partially,zhang2004immersed,He2012Theconvergence} and the references therein). 

\begin{figure}[H]
		\begin{minipage}[t]{0.35\textwidth}
			\includegraphics[height = 5.3 cm,width= 6.4 cm, clip, trim = 3.5cm 2cm 3cm 1cm]{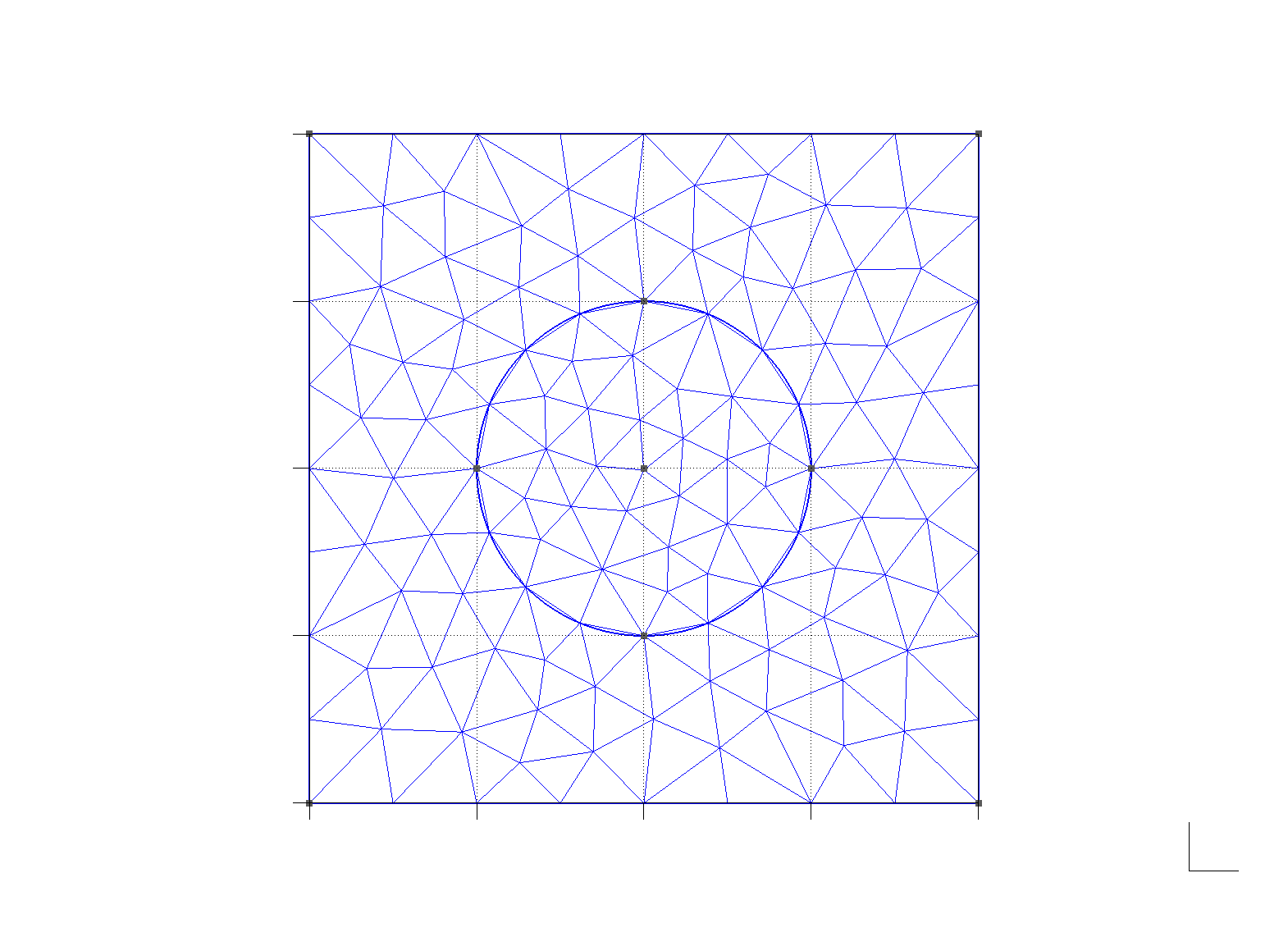} 
		\end{minipage}
		\begin{minipage}[t]{0.32\textwidth}
			\includegraphics[height = 5 cm,width=5 cm]{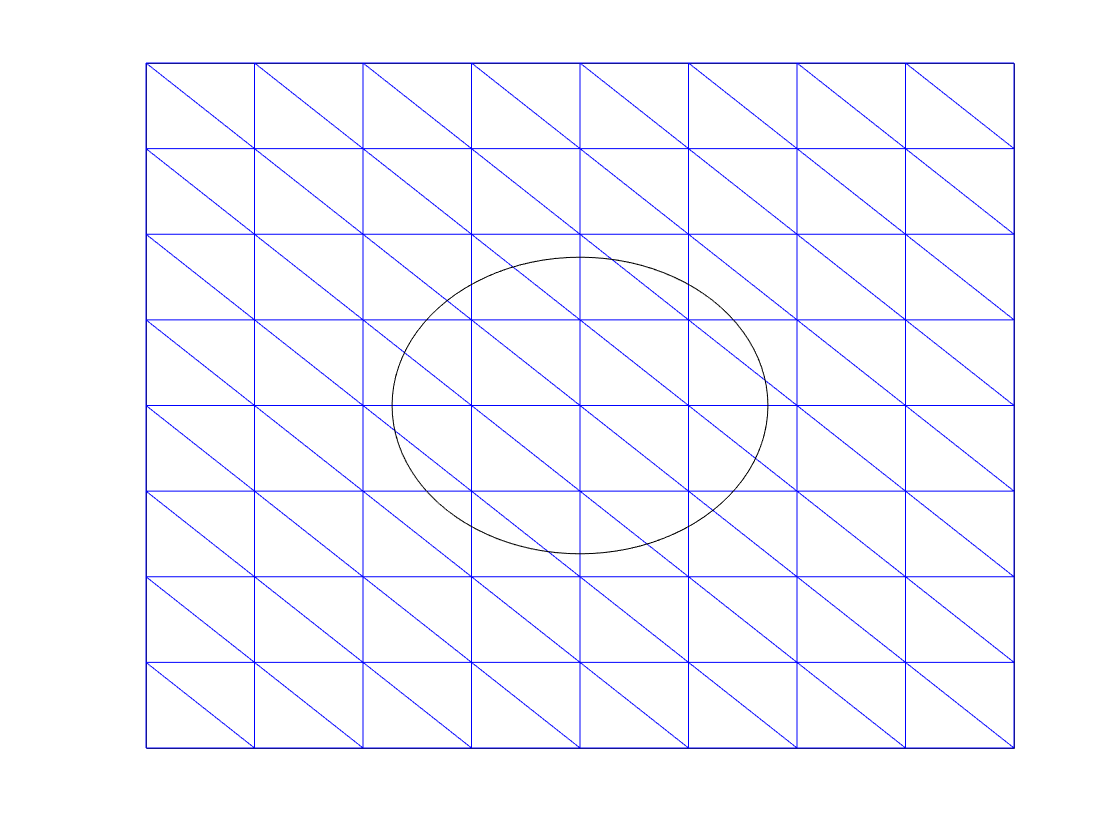}
		\end{minipage}
		\begin{minipage}[t]{0.32\textwidth}
			\includegraphics[height = 5 cm,width=5 cm]{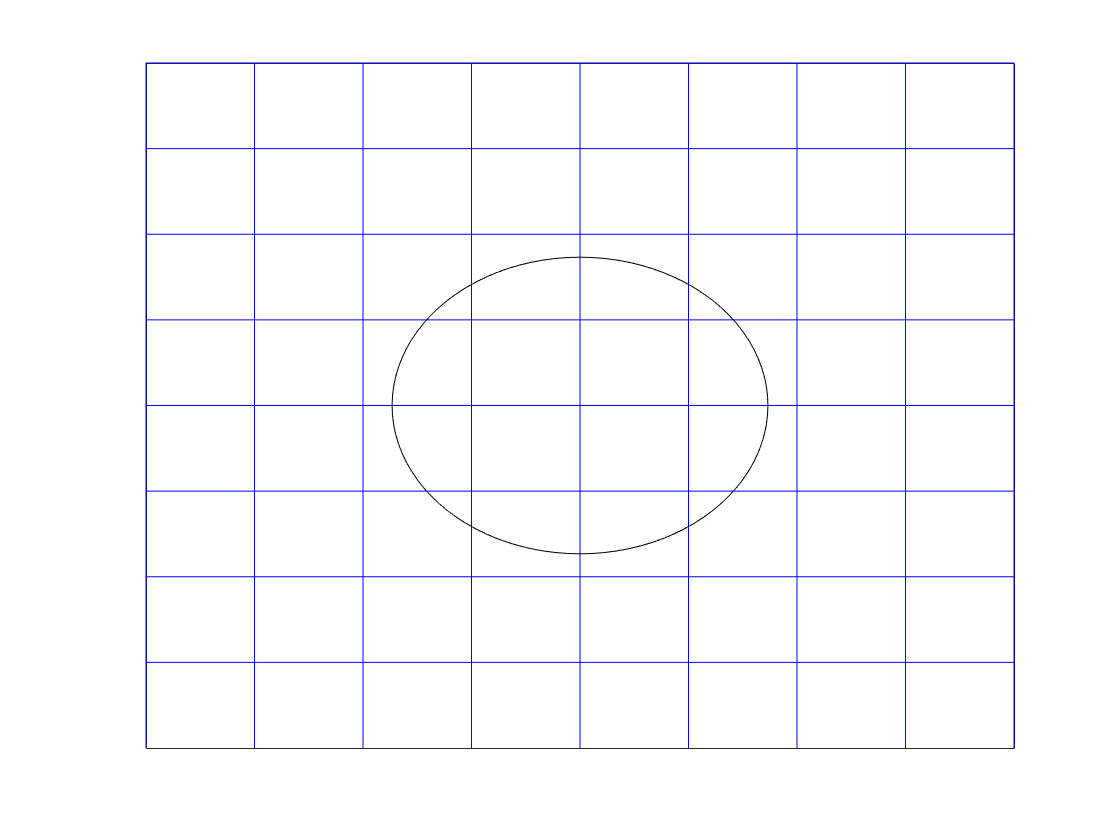}
		\end{minipage}
    \caption{Fitted(left) and Unfitted mesh(middle and right)}
	\label{Fitted and unfitted mesh}
\end{figure}

The hybridizable discontinuous Galerkin (HDG) framework \cite{Cockburn2009Unified} provides a unifying strategy for hybridization of finite element methods  for second order elliptic problems. 
By the local elimination of the unknowns defined in the interior of elements, the HDG method    leads to a system where the unknowns are only the globally coupled degrees of freedom describing the introduced Lagrange multiplier.  We refer to \cite{Cockburn2010HDGstokes,Cockburn2011HDGstokes,cockburn2014divergence, Li-X-Z2016analysis,Li-X2016analysis,Li-X2016SIAM, Fu2018Parameter,Araya2018Analysis,Gatica2017Analysis,Huynh2013A,Dong2018An,G2016eXtendedvoid,G2016eXtended1,Wang2013Hybridizable}
for some   developments and applications of the HDG method.    We also mention that 
    arbitrary order interface-unfitted eXtended HDG methods with optimal convergence were analyzed   in \cite{HanChenWangXie2019Extended,hanwangxie2020interfaceboundaryunfitted}  for   elliptic and elasticity interface problems, respectively. 

In this paper we aim to propose a low order eXtended HDG (X-HDG) method for the Darcy-Stokes-Brinkman interface problem \eqref{pb1}. The main features of the method are as follows:
\begin{itemize}
	\item  The method is a low order scheme, which uses piecewise linear polynomials for the velocity approximation and piecewise constants for both the velocity gradient and pressure   approximations in the interior of elements inside the subdomains separated by the interface, and uses piecewise constants for the numerical traces of velocity on the inter-element boundaries inside the subdomains.  
	
	\item  To   deal with the interface conditions,  the  interface   is approximated by a fold line/plane,  on which the numerical traces of velocity  adopt    piecewise constants or  piecewise linear polynomials.  
	
	\item The method is parametric-friendly in the sense that optimal error estimates are obtained without requiring   ``sufficiently large” stabilization parameters in the scheme.
	
	\item The method uses interface-unfitted polygonal/polyhedral meshes, and applies to curved domains with boundary-unfitted meshes.
	
	\end{itemize}

The rest of the paper is organized as follows. Section 2 introduces  the X-HDG scheme for the interface problem with a polygonal/polyhedral domain. Section 3 is devoted to the error estimation. Section 4 applies  the X-HDG method to a curved domain problem. Numerical examples are provided in Section 5 to verify the theoretical results. Finally, Section 6 gives some concluding remarks.

\section{X-HDG scheme for interface problem }

\subsection{Notation and XFE spaces}
For any bounded domain $D\subset \mathbb{R}^s$ $(s= d, d-1)$ and nonnegative integer $m$, let $H^m(D)$ and $H_0^m(D)$ be the usual $m$-th order Sobolev spaces on $D$, with norm $\lVert \cdot\rVert_{m,D}$ and semi-norm   $\lvert\cdot\rvert_{m,D}$.  In particular, $L^2(D):=H^0(D)$ is the space of square integrable functions, with  inner product   $(\cdot,\cdot)_{D}$.
When $D \subset R^{d-1}$, we use $\langle\cdot,\cdot\rangle_D$ to replace $(\cdot,\cdot)_D$.
For $m=1,2$, we set 
\begin{align*}
H^m(\Omega_{1}\cup\Omega_{2}):= \{v\in L^2(\Omega),v|_{\Omega_{1}} \in H^m(\Omega_{1}),\  \text{and}\  v|_{\Omega_{2}} \in H^m(\Omega_{2})\}, \\
\lVert\cdot\rVert_m := \lVert\cdot\rVert_{m,\Omega_1\cup\Omega_2} = \sum\limits_{i=1}^{2} \lVert\cdot\rVert_{m,\Omega_i}, \quad
\lvert\cdot\rvert_m :=\lvert\cdot\rvert_{m,\Omega_1\cup\Omega_2}= \sum\limits_{i=1}^{2} \lvert\cdot\rvert_{m,\Omega_i}. 
\end{align*} 
For any integer $k\geqslant0$,    $P_k(D)$   denotes the set of all polynomials on D with degree at most $k$.

Assume  that $\Omega$ is a polygonal/polyhedral domain. Let $\mathcal{T}_h=\cup\{K\}$, consisting of arbitrary open polygons/polyhedrons, be a shape-regular partition of the domain $\Omega$  
in the   sense that the following two assumptions hold (cf. \cite{chen-Xie2016WG}): 

\noindent  {\bf(M1)}. There exists a positive constant $\theta_*$ such that the following holds: for each element $K\in \mathcal{T}_h$, there exists a point $M_K\in K$ such that $K$ is star-shaped with respect to every point in the circle (or sphere) of center $M_K$ and radius $\theta_*h_K$.
	
\noindent  {\bf(M2)}. There exists a positive constant $l_*$ such that for every element $K\in \mathcal{T}_h$, the distance between any two vertexes is no less than $  l_*h_K$.

We define the set of all elements intersected by the interface $\Gamma$ as 
\begin{align*}
\mathcal{T}_h^{\Gamma} :=\{K\in \mathcal{T}_h: K\cap\Gamma \neq \emptyset\}.
\end{align*} 
For any $K\in \mathcal{T}_h^{\Gamma}$, called an interface element, let  $\Gamma_K := K\cap \Gamma$ be the part of $ \Gamma$ in $K$, and   $\Gamma_{K,h}$ be the straight line/plane segment connecting the intersection between $\Gamma_K$ and $\partial K$ (Figure \ref{interfaceelem}). 

To ensure that $\Gamma$ is reasonably resolved by $\mathcal{T}_h$, we make the following standard assumptions on $\mathcal{T}_h$ and   $\Gamma$( cf. Figure \ref{violate assump} for two cases  violating the  assumptions): 

\noindent  {\bf(A1)}.  For $K\in \mathcal{T}_h^{\Gamma}$ and any edge/face $F\subset \partial K$ which intersects $\Gamma$,  $F_\Gamma:=\Gamma\cap F$ is simply connected with either   $F_\Gamma =F$ or $meas(F_\Gamma)=0$. 

\noindent 	{\bf(A2)}. For $K\in \mathcal{T}_h^{\Gamma}$, $\Gamma_{K}$ is sufficiently smooth such that for any two different points $\bm{x},\bm{y}\in \Gamma_{K}$, the unit normal vectors $\bm{n}(\bm{x})$ and $\bm{n}(\bm{y})$, pointing to $\Omega_{2}$, at $\bm{x}$ and $\bm{y}$ satisfy
\begin{align}\label{gamma}
\lvert \bm{n}(\bm{x})-\bm{n}(\bm{y})\rvert \leq \gamma h_K,
\end{align}
with $\gamma\geq 0$(cf.\cite{Chen1998Finite,xu2013estimate}). Note that   $\gamma = 0$ when $\Gamma_{K}=\Gamma_{K,h}$, i.e. $\Gamma_{K}$ is a straight line/plane segment.

\begin{figure}[H]
	\centering
		\includegraphics[height = 4.5 cm,width=4.5 cm]{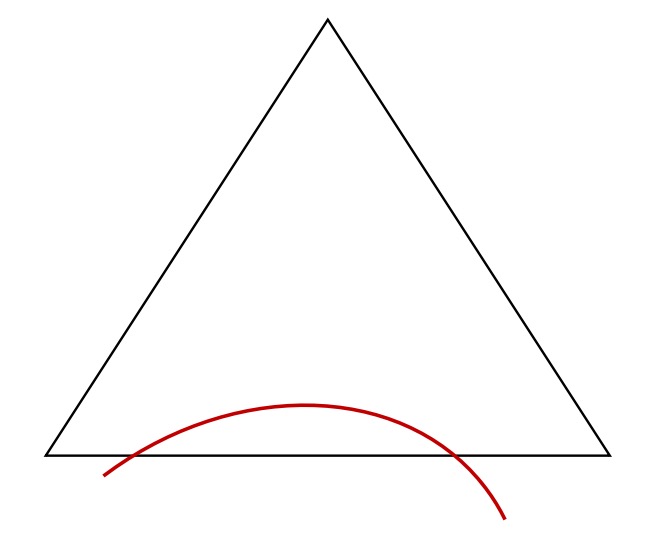} \qquad 
		\includegraphics[height = 4.5 cm,width=4.5 cm]{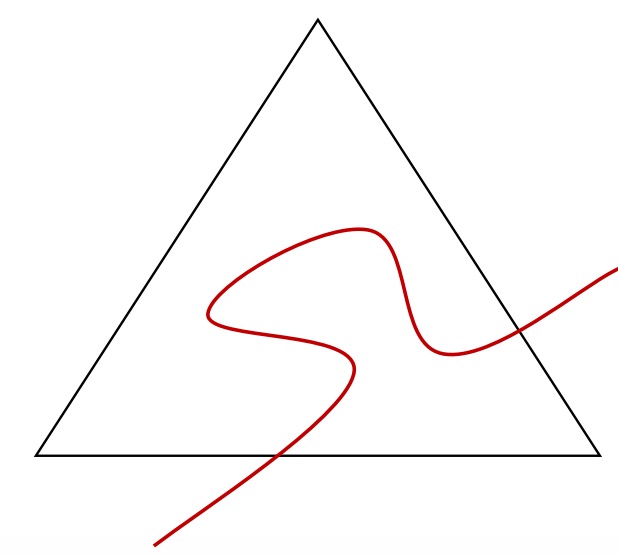}
	\caption{Two cases violating assumptions:  violating $\bf(A1)$(left) and violating $\bf(A2)$(right).}
	\label{violate assump}
\end{figure}

\begin{figure}[htbp]
	\centering	
	\includegraphics[height = 5.0 cm,width=6.5 cm]{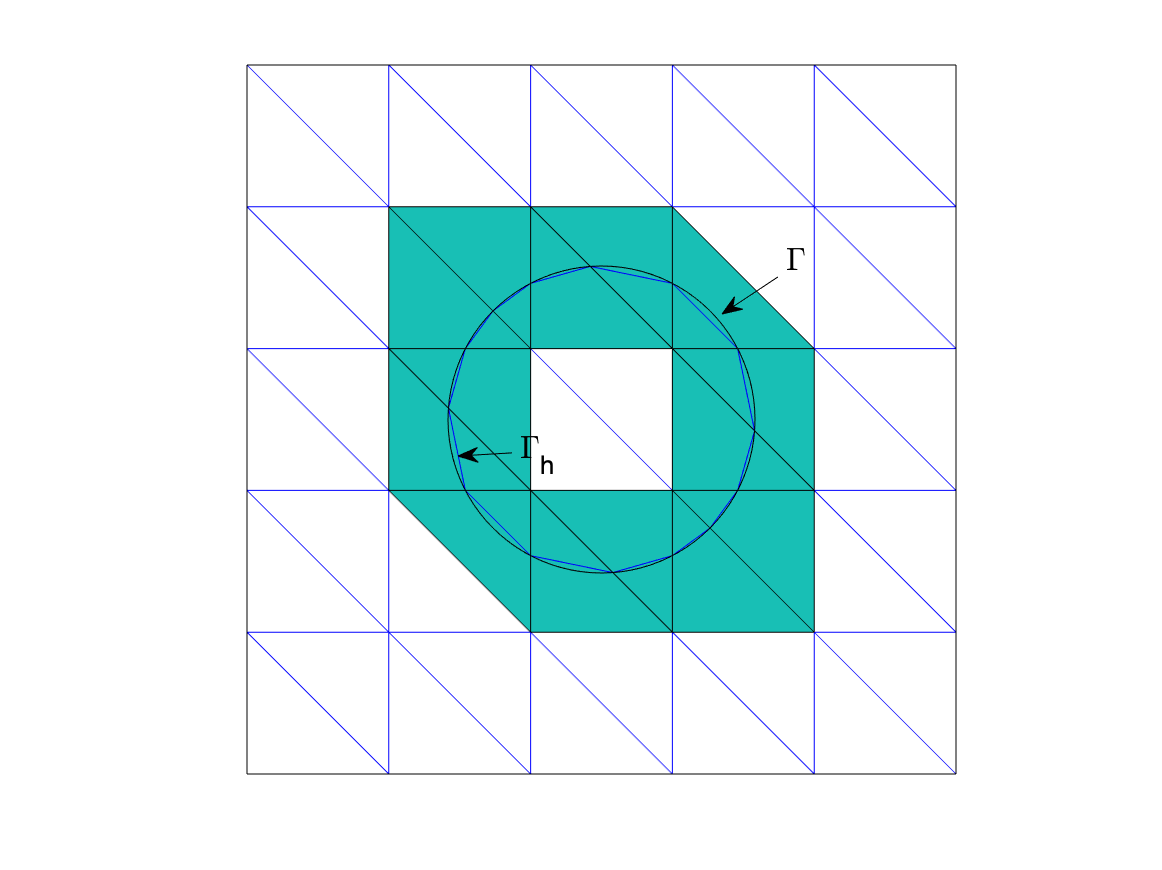} 				
	\caption{An example (a triangulation with circular interface):   $\mathcal{T}_h^{\Gamma}$ (green part),  $\varepsilon_h^{\Gamma}$ (collection of {  blue segments} inside the elements of $\mathcal{T}_h^{\Gamma}$), and $\varepsilon_h^*$ (collection of all  triangle edges not in $\varepsilon_h^{\Gamma}$).  }\label{interfaceelem}	
\end{figure}

Let  $\mathcal{\varepsilon}_h$ be the set of all edges(faces) of all elements in $\mathcal{T}_h$.
Denote 
by $\mathcal{\varepsilon}_h^{\Gamma} $  the partition of  the fold line/plane approximation of $\Gamma$
with respect to $\mathcal{T}_h$,  i.e.
$$\mathcal{\varepsilon}_h^{\Gamma} := \{ F:\ F=\Gamma_{K,h}, \text{ or } F= \Gamma\cap \partial K \text{ if $\Gamma\cap \partial K$ is an edge/face of $K$},\    K\in \mathcal{T}_h\}.$$
We set
$\mathcal{\varepsilon}_h^* :=\mathcal{\varepsilon}_h\setminus\mathcal{\varepsilon}_h^{\Gamma} $. For any $K\in \mathcal{T}_h$ and   $F\in \mathcal{\varepsilon}_h^*\cup \mathcal{\varepsilon}_h^{\Gamma},$    let $h_K$  and $h_F$ be respectively   the diameters of $K$ and   $F$, and let $\bm{n}_K$ be the unit outward  normal vector along   $\partial K$. Denote by 
$h: = \max\limits_{K\in \mathcal{T}_h}h_K$  the mesh size of $\mathcal{T}_h$, and by   $\nabla_h$ and $\nabla_h\cdot$ the piecewise-defined gradient and divergence operators with respect to $\mathcal{T}_h$, respectively.

Since the rest of the paper deals with the discrete problem, in what follows and without ambiguity, let $\Omega_{1}$ and $\Omega_2$ denote the two sides of $\Gamma_h$ rather than of $\Gamma$, and set $K_i = K\cap \Omega_{i}$ for $i = 1,2$. 

Throughout the paper, we use $a\lesssim b$ $(a\apprge b)$ to denote $a\leq Cb$ $(a\geq Cb)$, 
where   $C$ is a generic positive constant   independent of mesh parameters $h, h_K, h_e$,   the coefficients $\nu_i, \alpha_i$ $(i =1,2)$ and the  location of the interface relative to the mesh.

\subsection{ X-HDG scheme}
The X-HDG method is based on the following first-order formulations of Darcy-Stokes-Brinkman interface problem \eqref{pb1}: 
\begin{subequations}\label{firstorderscheme}
	\begin{align}
	\bm{L} = \nu\nabla \bm{ u} &\quad   \text{in} \quad\Omega_1\cup\Omega_2, \label{firstorderscheme-a}\\
	- \nabla\cdot\bm{L}+\nabla p + \alpha \bm{u}= \bm{f} & \quad  \text{in} \quad\Omega_1\cup\Omega_2,\\
	\nabla\cdot\bm{u} = 0 & \quad  \text{in} \quad\Omega_1\cup\Omega_2,\\
	\bm{u}=\bm{g}_D& \quad \text{on} \quad \partial\Omega,\\
	\llbracket \bm{u} \rrbracket= \bm{0},\ \llbracket (\bm{L}-p\bm{I})\bm{n} \rrbracket = \bm{g}_N^{\Gamma}&\quad \text{on}\quad \Gamma.
	\end{align}
\end{subequations}

Let $\chi_i$ be the characteristic function on $\Omega_{i}$ for  $i=1,2$.  
For any integer $r\geq 0$, $F\in \mathcal{\varepsilon}_h^*\cup \mathcal{\varepsilon}_h^{\Gamma}$ and   $K\in \mathcal{T}_h$, let  $Q_{r}^b: L^2(D)\rightarrow P_r(D)$  and  $Q_{r}: L^2(\tilde D)\rightarrow P_r(\tilde D)$  be the standard $L^2$ orthogonal projection operators    for   $D=F\cap\bar{\Omega}_i$ and $\tilde D=K\cap\Omega_i$, respectively.  Vector or tensor analogues of   $Q_{r}^b$ and $Q_{r}$  are denoted by $\bm{Q}_{r}^b$ and  $\bm{Q}_{r}$, respectively.
Set $$ \oplus\chi_iP_{r}(K) := \chi_1P_{r}(K)+\chi_2P_{r}(K), \quad r=0,1.$$
We  introduce  the following X-HDG finite element spaces:
\begin{align*}
\bm{W}_h =& \{\bm{w}\in L^2(\Omega)^{d\times d}:\ \forall K\in \mathcal{T}_h, \bm{w}|_K \in P_0(K)^{d\times d} \ \text{if} \ K\cap \Gamma = \varnothing; \bm{w}|_K \in (\oplus\chi_i P_0(K))^{d\times d} \ \text{if}\  K\cap \Gamma \neq \varnothing\}, \\
\bm{V}_h =& \{\bm{v}\in L^2(\Omega)^d:  \ \forall K\in \mathcal{T}_h, \bm{v}|_K \in  P_1(K)^{d}  \ \text{if} \  K\cap \Gamma = \varnothing; \bm{v}|_K \in (\oplus\chi_i  P_1(K))^{d} \ \text{if}\ K\cap \Gamma \neq \varnothing\}, \\
Q_h = & \{q\in L_0^2(\Omega): \ \forall K\in \mathcal{T}_h, q|_K \in P_0(K) \ \text{if} \  K\cap \Gamma = \varnothing;  q|_K \in \oplus\chi_i P_0(K) \ \text{if} \ K\cap \Gamma \neq \varnothing\}, \\
\bm{M}_h =& \{ \bm{\mu}\in L^2(\varepsilon_h^*)^d: \forall F\in \varepsilon_h^*, \bm{\mu}|_F \in P_0(F)^d \ \text{if}\   F\cap \Gamma = \varnothing;  \bm{\mu}|_F \in (\oplus\chi_i P_0(F) )^{d}\ \text{if}\   F\cap \Gamma \neq \varnothing\}, \\
\bm{\tilde{M}_h} = &\{\tilde{\bm{\mu}} \in L^2(\mathcal{\varepsilon}_h^{\Gamma})^d: 
\tilde{\bm{\mu}}|_{F}\in P_m(F)^d,\forall F\in\mathcal{\varepsilon}_h^{\Gamma}\}  \text{ with } m=0 \text{ or } 1,\\
\bm{M}_h(\bm{g}_D) =& \{\bm{\mu}\in \bm{M}_h : \bm{\mu}|_{F} = \bm{Q}_0^b\bm{g}_D, \forall F\in \varepsilon_h \text{ with }F\subset \partial\Omega\}.
\end{align*}
To describe the X-HDG scheme,   we also define 
\begin{align*}
(\cdot,\cdot)_{\mathcal{T}_h}:=\sum\limits_{K\in  \mathcal{T}_h}(\cdot,\cdot)_K,\quad  \langle\cdot,\cdot\rangle_{\partial\mathcal{T}_h\setminus \mathcal{\varepsilon}_h^{\Gamma}}:=\sum\limits_{K\in  \mathcal{T}_h}\langle\cdot,\cdot\rangle_{\partial K\backslash \mathcal{\varepsilon}_h^{\Gamma}},
\end{align*}
and, for  scalars $q$, vector $\bm{u},\bm{v}$ and tensor $\bm{w}$ with $q_i=q|_{\bar{\Omega}_i\cap F }, \bm{u}_i=\bm{u}|_{\bar{\Omega}_i\cap F}, \bm{v}_i=\bm{v}|_{\bar{\Omega}_i\cap F}$ and $\bm{w}_i=\bm{w}|_{\bar{\Omega}_i\cap F}$, 
\begin{align*}
\langle q,\bm{v}\cdot \bm{n}\rangle_{*,\mathcal{\varepsilon}_h^{\Gamma}} :&=\sum\limits_{F\in \mathcal{\varepsilon}_h^{\Gamma}}\langle q_1,\bm{v}_1\cdot \bm{n}_1\rangle_{\bar{\Omega}_1\cap F}+\langle q_2,\bm{v}_2\cdot \bm{n}_2\rangle_{\bar{\Omega}_2\cap F}, \\
\langle \bm{u},\bm{v}\rangle_{*,\mathcal{\varepsilon}_h^{\Gamma}} :&=\sum\limits_{F\in \mathcal{\varepsilon}_h^{\Gamma}}\langle \bm{u}_1,\bm{v}_1\rangle_{\bar{\Omega}_1\cap F}+\langle \bm{u}_2,\bm{v}_2\rangle_{\bar{\Omega}_2\cap F},\\
\langle \bm{w}\bm{n},\bm{v}\rangle_{*,\mathcal{\varepsilon}_h^{\Gamma}} :&=\sum\limits_{F\in \mathcal{\varepsilon}_h^{\Gamma}}\langle \bm{w}_1\bm{n}_1,\bm{v}_1\rangle_{\bar{\Omega}_1\cap F}+\langle \bm{w}_2\bm{n}_2,\bm{v}_2\rangle_{\bar{\Omega}_2\cap F} ,
\end{align*}
where $\bm{n}_i$ denotes  the unit normal vector along $\Gamma_h$ pointing from   $\Omega_{i}$ to $\Omega_{j}$ with $i,j=1,2$ and $i\neq j$. 

The eXtended HDG method seeks $(\bm{L}_h,\bm{u}_h,p_h,\bm{\hat{u}}_h,\bm{\tilde{u}}_h)\in \bm{W}_h\times \bm{V_h}\times Q_h\times \bm{M}_h(\bm{g}_D)\times \bm{\tilde{M}_h}$ satisfying
\begin{subequations}\label{xhdgscheme}
	\begin{align}
	(\nu^{-1}\bm{L}_h,\bm{w})_{\mathcal{T}_h} 
	- \langle \hat{\bm{u}}_h,\bm{w}\bm{n}\rangle_{\partial\mathcal{T}_h\setminus \mathcal{\varepsilon}_h^{\Gamma}} -\langle \tilde{\bm{u}}_h,\bm{w}\bm{n}\rangle_{*,\mathcal{\varepsilon}_h^{\Gamma}} =& 0, \label{xhdg1}\\
	(\alpha\bm{u}_h,\bm{v})_{\mathcal{T}_h} + \langle \tau(\bm{Q}_0^b\bm{u}_h-\bm{\hat{u}}_h),\bm{v}\rangle_{\partial\mathcal{T}_h\setminus \mathcal{\varepsilon}_h^{\Gamma}} 
	+  \langle \tau(\bm{Q}_m^b\bm{u}_h-\bm{\tilde{u}}_h),\bm{v}\rangle_{*,\mathcal{\varepsilon}_h^{\Gamma}}  =& (\bm{f},\bm{v}) ,\label{xhdg2}\\
	 \langle\bm{\hat{u}}_h\cdot\bm{n},q\rangle_{\partial\mathcal{T}_h\setminus \mathcal{\varepsilon}_h^{\Gamma}} + \langle\bm{\tilde{u}}_h\cdot\bm{n},q\rangle_{*,\mathcal{\varepsilon}_h^{\Gamma}}=& 0,\label{xhdg3}\\
	\langle \bm{L}_h \bm{n},\bm{\mu}\rangle_{\partial \mathcal{T}_h\setminus \mathcal{\varepsilon}_h^{\Gamma}}-\langle p_h\bm{n},\bm{\mu}\rangle_{\partial\mathcal{T}_h\setminus \mathcal{\varepsilon}_h^{\Gamma}}-\langle \tau(\bm{Q}_0^b\bm{u}_h-\bm{\hat{u}}_h),\bm{\mu}\rangle_{\partial \mathcal{T}_h\setminus\mathcal{\varepsilon}_h^{\Gamma}} =& 0,\label{xhdg5}\\
	\langle \bm{L}_h \bm{n},\bm{\tilde{\mu}}\rangle_{*,\mathcal{\varepsilon}_h^{\Gamma}}-\langle p_h\bm{n},\bm{\tilde{\mu}}\rangle_{*,\mathcal{\varepsilon}_h^{\Gamma}}-\langle \tau(\bm{Q}_m^b\bm{u}_h-\bm{\tilde{u}}_h),\bm{\tilde{\mu}}\rangle_{*,\mathcal{\varepsilon}_h^{\Gamma}}=& \langle \bm{g}_N^{\Gamma_h},\bm{\tilde{\mu}}\rangle_{*,\mathcal{\varepsilon}_h^{\Gamma}} , \label{xhdg6} 
	\end{align}
\end{subequations}
for all $(\bm{w},\bm{v},q,\bm{\mu},\bm{\tilde{\mu}})\in \bm{W}_h\times \bm{V_h}\times Q_h\times \bm{M}_h(0)\times \bm{\tilde{M}_h}$. Here the stabilization function $\tau$ is defined as following: for any  $K\in \mathcal{T}_h $
and $i=1,2$,
\begin{align}\label{stabilization function}
\tau|_{F\cap \bar\Omega_{i}} &= \nu_i h_K^{-1} ,\quad {\rm for} \ F\subset \partial K\setminus  {\mathcal{\varepsilon}_h^{\Gamma}} \ {\rm or } \ F\in \mathcal{\varepsilon}_h^{\Gamma} .
\end{align}
When   $F\in \mathcal{\varepsilon}_h^{\Gamma}$ is a line segment/straight plane, we take $\bm{g}_N^{\Gamma_h}|_{F} := \bm{g}_N^{\Gamma}$, and when  $F=\Gamma_{K,h}\neq \Gamma_K$ for some $K\in \mathcal{T}_h^{\Gamma}$, we set $\bm{g}_N^{\Gamma_h}|_{F}$ to be some linear interpolation of  $\bm{g}_N^{\Gamma}$  using data of $\bm{g}_N^{\Gamma}$ at two (2D case) /three (3D case) intersection points of  $\Gamma_K$ and $\Gamma_{K,h}$. 
\begin{rem}
	From the first-order system \eqref{firstorderscheme}, there should be some terms like $(\bm{u}_h,\nabla_h\cdot \bm{w})_{\mathcal{T}_h} $, $-(\nabla_h\cdot\bm{L}_h,\bm{v})_{\mathcal{T}_h}$, $(\nabla_h p_h,\bm{v})_{\mathcal{T}_h}$, and $-(\bm{u}_h,\nabla q)_{\mathcal{T}_h}$ in the scheme \eqref{xhdgscheme}. In fact, they are all vanish  since $\bm{W}_h$ and $  Q_h$ are piecewise constant  tensor/scalar spaces.
\end{rem}
\begin{rem}
	We note that in the implementation, we can locally eliminate the $\bm{L}_h, \bm{u}_h$ defined in the interior of  elements, and the reduced system only involves the unknowns of $p_h$,   $\hat{\bm{u}}_h$ and $\tilde{\bm{u}}_h$.
\end{rem}

\begin{thm}\label{wellposed}
The X-HDG scheme \eqref{xhdgscheme} (with $m=0$ or $1$) admits a unique solution.
\end{thm}
\begin{proof}
	Since the \eqref{xhdgscheme} is a linear square system, it suffices to show that if all of the given data vanish, i.e. $\bm{f} = \bm{g}_D  = \bm{g}_N^{\Gamma_h} = 0$, then we get the zero solution. By taking the $(\bm{w},\bm{v},q,\bm{\mu},\bm{\tilde{\mu}}) = (\bm{L}_h,\bm{u}_h,p_h,\bm{\hat{u}_h},\bm{\tilde{u}_h})$ in \eqref{xhdgscheme} and adding these equations together, we have
		\begin{align*}
	(\nu^{-1}\bm{L}_h,\bm{L}_h)_{\mathcal{T}_h} + (\alpha\bm{u}_h,\bm{u}_h)_{\mathcal{T}_h}+\langle \tau(\bm{Q}_0^b\bm{u}_h-\bm{\hat{u}}_h),(\bm{u}_h-\bm{\hat{u}}_h)\rangle_{\partial\mathcal{T}_h\setminus \mathcal{\varepsilon}_h^{\Gamma}} + \langle \tau(\bm{Q}_m^b\bm{u}_h-\bm{\tilde{u}}_h),(\bm{u}_h-\bm{\tilde{u}}_h)\rangle_{*,\mathcal{\varepsilon}_h^{\Gamma}} = 0 ,
	\end{align*}
		which indicates $\bm{L}_h = \bm{0}$,
	\begin{align}
	\bm{Q}_0^b\bm{u}_h-\bm{\hat{u}}_h &= \bm{0}, \quad \text{ on} \ \partial K\setminus  \mathcal{\varepsilon}_h^{\Gamma}, \forall K\in \mathcal{T}_h,\label{eq1} \\
	\{\bm{Q}_m^b\bm{u}_h-\bm{\tilde{u}}_h\} &= \bm{0},  \quad \text{ on}\  \mathcal{\varepsilon}_h^{\Gamma}.\label{eq2}
	\end{align}
	where $\{\cdot\} $ is defined by  $\{w\} = \frac{1}{2}(w_1+w_2)$ with $w_i=w|_{\bar{\Omega}_i\cap\Gamma} $ for $i=1,2$. These  relations, together with 
	the equation \eqref{xhdg1}, the definition of projection and integration by parts, yield
	\begin{align*}
	0=(\nu^{-1}\bm{L}_h,\bm{w})_{\mathcal{T}_h} +(\nabla_h \bm{u}_h, \bm{w})_{\mathcal{T}_h} - \langle \bm{Q}_0^b\bm{u}_h-\bm{\hat{u}}_h,\bm{w} \bm{n}\rangle_{\partial\mathcal{T}_h\setminus \mathcal{\varepsilon}_h^{\Gamma}} -\langle \bm{Q}_m^b\bm{u}_h-\bm{\tilde{u}}_h,\bm{w}\bm{n}\rangle_{*,\mathcal{\varepsilon}_h^{\Gamma}} = (\nabla_h \bm{u}_h, \bm{w})_{\mathcal{T}_h} .
	\end{align*}
	Taking the $\bm{w}_h = \nabla_h \bm{u}_h$ in this relation gives  $\nabla_h \bm{u}_h = \bm{0}$. Then  $\bm{u}_h $ is piecewise constant, which, together with \eqref{eq1},  \eqref{eq2} and the fact that $\bm{\hat{u}}_h = \bm{0}$ on $\partial \Omega$, implies 
	\begin{align*}
	  \bm{u}_h =  \bm{\hat{u}}_h  = \bm{\tilde{u}}_h  =\bm{0}.
	\end{align*}
The thing left is to show $p_h=0$. In view of  \eqref{xhdg5} and \eqref{xhdg6}, we have 
	\begin{align*}
	\llbracket p_h\rrbracket = 0, &\ \text{ on} \ \partial K \text{ and }  \mathcal{\varepsilon}_h^{\Gamma}, \forall K\in \mathcal{T}_h.  
	\end{align*}
Thus, $p_h$ is a constant in $\Omega$, and the fact  $p_h\in L_0^2(\Omega)$ means $p_h=0$. This completes the proof.
\end{proof}

\section{A priori error estimation: a case of fold line/plane interface}
This section is devoted to an error analysis of the X-HDG scheme \eqref{xhdgscheme} with a fold line/plane  interface $\Gamma$.   We note that in this case $\Gamma_K=\Gamma_{K,h}$ is a line segment/quadrilateral  for any $K\in \mathcal{T}_h^{\Gamma}$, and
$\bm{g}_N^{\Gamma_h} = \bm{g}_N^{\Gamma}$ in the equation \eqref{xhdg6}.

\subsection{Optimal error estimation for velocity gradient and pressure}

Firstly we introduce the following standard estimates  for the  $L^2$ orthogonal projection operators $Q_r$ and $Q_r^b$ (cf. \cite{chen-Xie2016WG,HanChenWangXie2019Extended}).

\begin{lem}\label{ineq} 
	Let $s$ be an integer with $1\leq s\leq r+1$. For any $K\in \mathcal{T}_h$ and $ v\in H^s\left((K\cap\Omega_{1})\cup (K\cap\Omega_{2}) \right)$, we have
	\begin{align*}
	\lVert v-Q_{r}v\rVert_{0,K}+h\lVert v-Q_{r}v\rVert_{1,K}&\lesssim h_K^s\lVert v\rVert_{s,K}, \\
	\lVert v-Q_{r}v\rVert_{0,\partial K}+\lVert v-Q_{r}v\rVert_{0,\Gamma_{K}}&\lesssim h_K^{s-1/2}\lVert v\rVert_{s,K}, \\ 
	\lVert v-Q_{r}^bv\rVert_{0,\partial K}  + \lVert v-Q_{r}^bv\rVert_{0,\Gamma_{K}}
	&\lesssim h_K^{s-1/2}\lVert v\rVert_{s,K},
	\end{align*}
	where the notations $\lVert \cdot\rVert_{s,K} $ and $\lVert\cdot\rVert_{0,\partial K}$ are understood respectively  as   $\lVert \cdot\rVert_{s,K} = \sum\limits_{i=1}^2\lVert \cdot\rVert_{s,K\cap\Omega_{i}}$ and $\lVert \cdot\rVert_{s,\partial K} = \sum\limits_{i=1}^{2}\lVert \cdot\rVert_{s,\partial K\cap  \bar{\Omega}_{i}}$ when  $K\in  \mathcal{T}_h^\Gamma$. 
\end{lem}

For simplicity of presentation, denote
\begin{align}\label{erroroperator}
\bm{e}_h^L: &= \bm{L}_h - \bm{Q}_0\bm{L}, \quad  \bm{e}_h^u: =\bm{u}_h - \bm{Q}_1\bm{u}, \quad  e_h^p: = p_h - Q_0p, \quad
 \bm{e}_h^{\hat{u}}: = \bm{\hat{u}}_h - \bm{Q}_0^b\bm{u}, \quad  \bm{e}_h^{\tilde{u}}: = \bm{\tilde{u}}_h - \bm{Q}_m^b\bm{u},
\end{align}
where, for $ i = 1,2$ and   $m=0,1$,
\begin{align*}
&(\bm{Q}_0\bm{L})|_{K\cap\Omega_{i}}:= \bm{Q}_0(\bm{L}|_{K\cap\Omega_{i}}), \quad 
(\bm{Q}_1\bm{u})|_{K\cap\Omega_{i}}:= \bm{Q}_1(\bm{u}|_{K\cap\Omega_{i}}), \ \quad 
(Q_0p)|_{K\cap\Omega_{i}}:= Q_0(p|_{K\cap\Omega_{i}}),\quad  \forall K\in\mathcal{T}_h ,\nonumber\\
&(\bm{Q}_0^b\bm{u})|_{F\cap\Omega_{i}}:= \bm{Q}_0^b(\bm{u}|_{F\cap\Omega_{i}}), \quad  \forall F\in\varepsilon_h^{*}, \\
&(\bm{Q}_m^b\bm{u})|_F:= 
\bm{Q}_m^b(\bm{u}|_{F}), \quad \forall F\in\mathcal{\varepsilon}_h^{\Gamma}.
\end{align*} 
Then we have the following lemma for error equations.
\begin{lem}
	For all $(\bm{w},\bm{v},q,\bm{\mu},\bm{\tilde{\mu}})\in \bm{W}_h\times \bm{V}_h\times Q_h\times \bm{M}_h(0)\times \bm{\tilde{M}}_h$, it holds
	\begin{subequations}\label{errorequation}
		\begin{align}
		(\nu^{-1}\bm{e}_h^L,\bm{w})_{\mathcal{T}_h} 
		-\langle \bm{e}_h^{\hat{u}},\bm{w}\bm{n} \rangle_{\partial\mathcal{T}_h\setminus \mathcal{\varepsilon}_h^{\Gamma}} -\langle \bm{e}_h^{\tilde{u}},\bm{w}\bm{n} \rangle_{*,\mathcal{\varepsilon}_h^{\Gamma}} &= 0 , \label{errorequation1}\\
		 (\alpha \bm{e}_h^u,\bm{v})_{\mathcal{T}_h}+ \langle \tau(Q_0^b\bm{e}_h^u-\bm{e}_h^{\hat{u}}),\bm{v}\rangle_{\partial\mathcal{T}_h\setminus \mathcal{\varepsilon}_h^{\Gamma}} 
		+  \langle \tau(Q_m^b\bm{e}_h^u-\bm{e}_h^{\tilde{u}}),\bm{v}\rangle_{*,\mathcal{\varepsilon}_h^{\Gamma}}  &= \sum_{i=1}^{2}L_i(\bm{v}) ,\label{errorequation2}\\
		 \langle\bm{e}_h^{\hat{u}}\cdot\bm{n},q\rangle_{\partial\mathcal{T}_h\setminus \mathcal{\varepsilon}_h^{\Gamma}} + \langle\bm{e}_h^{\tilde{u}}\cdot\bm{n},q\rangle_{*,\mathcal{\varepsilon}_h^{\Gamma}}&= 0,\label{errorequation3}\\
		\langle \bm{e}_h^L\bm{n},\bm{\mu}\rangle_{\partial \mathcal{T}_h\setminus \mathcal{\varepsilon}_h^{\Gamma}}-\langle e_h^p\bm{n},\bm{\mu}\rangle_{\partial \mathcal{T}_h\setminus \mathcal{\varepsilon}_h^{\Gamma}}-\langle \tau(Q_0^b\bm{e}_h^u-\bm{e}_h^{\hat{u}}),\bm{\mu}\rangle_{\partial \mathcal{T}_h\setminus \mathcal{\varepsilon}_h^{\Gamma}}&= -L_1(\bm{\mu}), \label{errorequation4}\\
		\langle \bm{e}_h^L\bm{n},\bm{\tilde{\mu}}\rangle_{*,\mathcal{\varepsilon}_h^{\Gamma}}-\langle e_h^p\bm{n},\bm{\tilde{\mu}}\rangle_{*,\mathcal{\varepsilon}_h^{\Gamma}}-\langle \tau(Q_m^b\bm{e}_h^u-\bm{e}_h^{\tilde{u}}),\bm{\tilde{\mu}}\rangle_{*,\mathcal{\varepsilon}_h^{\Gamma}}&= -L_2(\bm{\tilde{\mu}}) \label{errorequation5},
		\end{align}
	\end{subequations}
	where for any $\psi \in H^1(\Omega_1\cup\Omega_2)\cup\bm{W}_h\cup  \bm{V}_h\cup  \bm{M}_h\cup  \bm{\tilde{M}}_h$, 
	\begin{align*}
	L_1(\psi) &= \langle (\bm{Q}_0\bm{L}-\bm{L})\bm{n},\psi \rangle_{\partial \mathcal{T}_h\setminus \mathcal{\varepsilon}_h^{\Gamma}}+\langle \tau (\bm{Q}_0^b(\bm{u}-\bm{Q}_1\bm{u})),\psi \rangle_{\partial \mathcal{T}_h\setminus \mathcal{\varepsilon}_h^{\Gamma}} +\langle (Q_0p-p)\bm{n},\psi \rangle_{\partial \mathcal{T}_h\setminus \mathcal{\varepsilon}_h^{\Gamma}}, \\
	L_2(\psi) &= \langle (\bm{Q}_0\bm{L}-\bm{L})\bm{n},\psi \rangle_{*,\mathcal{\varepsilon}_h^{\Gamma}}+ \langle \tau (\bm{Q}_m^b(\bm{u}-\bm{Q}_1\bm{u})),\psi \rangle_{*,\mathcal{\varepsilon}_h^{\Gamma}}+\langle (Q_0p-p)\bm{n},\psi \rangle_{*,\mathcal{\varepsilon}_h^{\Gamma}}.
	\end{align*}
\end{lem}
\begin{proof}
	Let $(\bm{L},\bm{u},p)$ be the solution of \eqref{firstorderscheme}. From the definitions of the projection operators we obtain
	\begin{align*}
	(\nu^{-1}\bm{Q}_0\bm{L},\bm{w})_{\mathcal{T}_h}-\langle \bm{Q}_0^b\bm{u},\bm{w}\bm{n} \rangle_{\partial \mathcal{T}_h\setminus\mathcal{\varepsilon}_h^{\Gamma}} -\langle \bm{Q}_m^b\bm{u},\bm{w}\bm{n} \rangle_{*,\mathcal{\varepsilon}_h^{\Gamma}} &=0,    \\
     \langle (\bm{Q}_0\bm{L}-\bm{L})\bm{n},\bm{v} \rangle_{\partial \mathcal{T}_h\setminus\mathcal{\varepsilon}_h^{\Gamma}}+\langle (\bm{Q}_0\bm{L}-\bm{L})\bm{n},\bm{v} \rangle_{*,\mathcal{\varepsilon}_h^{\Gamma}} 
	+ (\alpha\bm{Q}_1\bm{u},\bm{v} \rangle_{\mathcal{T}_h}&\\
	-\langle (Q_0p-p)\bm{n},\bm{v} \rangle_{\partial \mathcal{T}_h\setminus\mathcal{\varepsilon}_h^{\Gamma}}-\langle (Q_0p-p)\bm{n},\bm{v} \rangle_{*,\mathcal{\varepsilon}_h^{\Gamma}} &= (\bm{f},\bm{v}), \\
	\langle \bm{Q}_0^b \bm{u}\cdot\bm{n},q\rangle_{\partial\mathcal{T}_h\setminus \mathcal{\varepsilon}_h^{\Gamma}} + \langle  \bm{Q}_0^b\bm{u}\cdot\bm{n},q\rangle_{*,\mathcal{\varepsilon}_h^{\Gamma}}&= 0,
	\end{align*}
for any $(\bm{w},\bm{v},q)\in \bm{W}_h\times \bm{V}_h\times Q_h$. Then, subtracting \eqref{xhdg1}, \eqref{xhdg2} and \eqref{xhdg3}  respectively from the above three equations yields   \eqref{errorequation1}, \eqref{errorequation2} and \eqref{errorequation3}. Finally, \eqref{errorequation4}, \eqref{errorequation5} follows from \eqref{xhdg5}, \eqref{xhdg6} and the relations
	\begin{align*}
		\langle (\bm{L}-p\bm{I})\bm{n},\bm{\bm{\mu}} \rangle_{\partial \mathcal{T}_h\setminus\mathcal{\varepsilon}_h^{\Gamma}} = \bm{0} , \quad 
		\langle (\bm{L}-p\bm{I})\bm{n},\bm{\tilde{\mu}} \rangle_{\mathcal{\varepsilon}_h^{\Gamma}} = \langle \bm{g}_N^{\Gamma_h},\bm{\tilde{\mu}}  \rangle_{\mathcal{\varepsilon}_h^{\Gamma}}		
		\end{align*}
		for $\bm{\mu}\in  \bm{M}_h(0),\ \bm{\tilde{\mu}}\in \bm{\tilde{M}}_h$.
\end{proof}

Define a seminorm  $\interleave \cdot\interleave: \bm{w},\bm{v},\bm{\mu},\bm{\tilde{\mu}} \in {L^2}(\Omega)^{d\times d} \times {L^2}(\Omega)^d\times   {L}^2(\varepsilon_h^*)^d\times {L^2}(\ \varepsilon_h^\Gamma)^d$ by 
\begin{align}\label{energynorm}
\interleave (\bm{w},\bm{v},\bm{\mu},\tilde{\mu})\interleave^2 : = \lVert \nu^{-\frac{1}{2}}\bm{w}\rVert_{0,\mathcal{T}_h}^2 +\lVert \alpha^{\frac{1}{2}}\bm{v}\rVert_{0,\mathcal{T}_h}^2 +\lVert \tau^{1/2}(\bm{Q}_0^b\bm{v}-\bm{\mu})\rVert_{\partial\mathcal{T}_h\setminus \mathcal{\varepsilon}_h^{\Gamma}}^2+\lVert \tau^{1/2}(\bm{Q}_m^b\bm{v}-\bm{\tilde{\mu}})\rVert_{*,\mathcal{\varepsilon}_h^{\Gamma}}^2, 
\end{align}
where
$$ 
\lVert  \cdot \rVert_{0,\mathcal{T}_h}^2 := \sum_{K\in\mathcal{T}_h}
\lVert \cdot\rVert_{0,K
}^2, 
\ \
\lVert \cdot\rVert_{\partial\mathcal{T}_h\setminus\varepsilon_h^\Gamma}^2: = \sum\limits_{K\in  \mathcal{T}_h}\langle\cdot,\cdot\rangle_{\partial K\backslash \varepsilon_h^\Gamma},\ \
 \lVert \cdot \rVert_{ \varepsilon_h^\Gamma}^2 := \langle \cdot,\cdot\rangle_{*,\varepsilon_h^\Gamma}. $$

\begin{lem}\label{estimateofenergynorm}
	Let $(\bm{L},\bm{u},p)\in {H^{1}}(\Omega_1\cup \Omega_2)^{d\times d}\times {H^{2}}(\Omega_1\cup \Omega_2)^d\times H^1(\Omega_1\cup \Omega_2) $ and $(\bm{L}_h,\bm{u}_h,p_h,\bm{\hat{u}}_h,\bm{\tilde{u}}_h)\in \bm{W}_h\times V_h\times Q_h\times M_h(g)\times \tilde{M}_h$ be the solutions of the  problem \eqref{firstorderscheme} and the X-HDG scheme \eqref{xhdgscheme}, respectively.  Then   it holds
	\begin{align}\label{estimategrade}
	\lVert  \nu^{\frac{1}{2}}\nabla_h \bm{e}_h^u\rVert_{0,\mathcal{T}_h} \lesssim \interleave (\bm{e}_h^L,\bm{e}_h^u,\bm{e}_h^{\hat{u}} ,\bm{e}_h^{\tilde{u}})\interleave  \lesssim h (\lVert \nu^{\frac{1}{2}}\bm{u}\rVert_{2,\Omega_{1}\cup\Omega_{2}}+\lVert \nu^{-\frac{1}{2}}p\rVert_{1,\Omega_{1}\cup\Omega_{2}}) .
	\end{align}
\end{lem}
\begin{proof} We first show 
	\begin{align}\label{leftone}
	\lVert  \nu^{\frac{1}{2}}\nabla_h \bm{e}_h^u\rVert_{0,\mathcal{T}_h} \lesssim 
	\interleave (\bm{e}_h^L,\bm{e}_h^u,\bm{e}_h^{\hat{u}},\bm{e}_h^{\tilde{u}})\interleave.
		\end{align}
	In fact,	taking $\bm{w} = \nu\nabla \bm{e}_h^u$ in \eqref{errorequation1} and applying integration by parts yield
	\begin{align*}
	(\bm{e}_h^L,\nabla \bm{e}_h^u)_{\mathcal{T}_h}-(\nu\nabla \bm{e}_h^u, \nabla \bm{e}_h^u)_{\mathcal{T}_h}-\langle \nu (\bm{e}_h^u-\bm{e}_h^{\hat{u}}),\nabla \bm{e}_h^u\bm{n} \rangle_{\partial\mathcal{T}_h\setminus \mathcal{\varepsilon}_h^{\Gamma}} -\langle \nu (\bm{e}_h^u-\bm{e}_h^{\tilde{u}}),\nabla \bm{e}_h^u\bm{n} \rangle_{*,\mathcal{\varepsilon}_h^{\Gamma}} &= 0,
	\end{align*}
	which, together with the property of projection, implies
	\begin{align*}
	\lVert\nu^{\frac{1}{2}}\nabla \bm{e}_h^u\rVert_{0,\mathcal{T}_h}^2 = (\bm{e}_h^L,\nabla \bm{e}_h^u)_{{\mathcal{T}_h}}-\langle \nu (\bm{Q}_0^b\bm{e}_h^u-\bm{e}_h^{\hat{u}}),\nabla \bm{e}_h^u\bm{n} \rangle_{\partial\mathcal{T}_h\setminus \mathcal{\varepsilon}_h^{\Gamma}} -\langle \nu (\bm{Q}_m^b\bm{e}_h^u-\bm{e}_h^{\tilde{u}}),\nabla \bm{e}_h^u\bm{n} \rangle_{*,\mathcal{\varepsilon}_h^{\Gamma}} 
	\end{align*}
	for $m=0,1$.
	In view of the Cauchy-Schwarz inequality, the trace inequality and the definition of $\interleave \cdot\interleave$, we then have 
	\begin{align*}
	\lVert \nu^{\frac{1}{2}}\nabla \bm{e}_h^u\rVert_{0,\mathcal{T}_h} \leq \lVert \nu^{-\frac{1}{2}}\bm{e}_h^L\rVert_{0,\mathcal{T}_h}+\lVert  \tau^{1/2}(\bm{Q}_0^b\bm{e}_h^u-\bm{e}_h^{\hat{u}})\rVert_{\partial\mathcal{T}_h\setminus \mathcal{\varepsilon}_h^{\Gamma}} +\lVert  \tau^{1/2}(\bm{Q}_m^b\bm{e}_h^u-\bm{e}_h^{\tilde{u}})\rVert_{\mathcal{\varepsilon}_h^{\Gamma}}\leq \interleave (\bm{e}_h^L,\bm{e}_h^u,\bm{e}_h^{\hat{u}},\bm{e}_h^{\tilde{\mu}})\interleave,
	\end{align*}
	where we recall that $\tau$ is given by \eqref{stabilization function}. 
	
	The thing left is to estimate the term $\interleave (\bm{e}_h^L,\bm{e}_h^u,\bm{e}_h^{\hat{u}},\bm{e}_h^{\tilde{\mu}})\interleave$.
	Taking $(\bm{w},\bm{v},q,\bm{\mu} ,\bm{\tilde{\mu}}) = (\bm{e}_h^L,\bm{e}_h^u,e_h^p,\bm{e}_h^{\hat{u}},\bm{e}_h^{\tilde{\mu}})$ in  \eqref{errorequation} and adding up the five equations, we  obtain 
$$\interleave (\bm{e}_h^L,\bm{e}_h^u,\bm{e}_h^{\hat u} ,\bm{e}_h^{\tilde{\mu}}) \interleave^2 = \sum_{i=1}^{2}E_i,$$	
where
	\begin{align*}
	E_1 &= \langle (\bm{Q}_0\bm{L}-\bm{L})\bm{n},\bm{e}_h^u-\bm{e}_h^{\hat{u}} \rangle_{\partial\mathcal{T}_h\setminus \mathcal{\varepsilon}_h^{\Gamma}}+\langle (\bm{Q}_0\bm{L}-\bm{L})\bm{n},\bm{e}_h^u-\bm{e}_h^{\tilde{u}} \rangle_{*,\mathcal{\varepsilon}_h^{\Gamma}}\\
	&\qquad + \langle (Q_0p-p)\bm{n},\bm{e}_h^u-\bm{e}_h^{\hat{u}} \rangle_{\partial\mathcal{T}_h\setminus \mathcal{\varepsilon}_h^{\Gamma}} +\langle (Q_0p-p)\bm{n},\bm{e}_h^u-\bm{e}_h^{\tilde{u}} \rangle_{*,\mathcal{\varepsilon}_h^{\Gamma}} ,\\
	E_2 &= \langle \tau \bm{Q}_0^b(\bm{u}-\bm{Q}_1\bm{u}),\bm{e}_h^u-\bm{e}_h^{\hat{u}} \rangle_{\partial\mathcal{T}_h\setminus \mathcal{\varepsilon}_h^{\Gamma}} +\langle \tau \bm{Q}_m^b(\bm{u}-\bm{Q}_1\bm{u}),\bm{e}_h^u-\bm{e}_h^{\tilde{u}} \rangle_{*,\mathcal{\varepsilon}_h^{\Gamma}} .
	\end{align*}	
	We just need to estimate   $E_i$  ($i = 1,2$) term by term. From Lemma \ref{ineq}, the Cauchy-Schwarz inequality and \eqref{leftone} it follows
	\begin{align*}
	E_1 
	\lesssim& h\big(\lVert \nu^{-\frac{1}{2}}\bm{L}\rVert_{1,\Omega_{1}\cup\Omega_{2}} +  \lVert \nu^{-\frac{1}{2}}p\rVert_{1,\Omega_{1}\cup\Omega_{2}}\big)\big(\lVert \tau^{\frac{1}{2}}(\bm{e}_h^u-\bm{e}_h^{\hat{u}}) \rVert_{0,\partial\mathcal{T}_h\setminus\mathcal{\varepsilon}_h^{\Gamma}} +\lVert \tau^{\frac{1}{2}}(\bm{e}_h^u-\bm{e}_h^{\tilde{u}}) \rVert_{\mathcal{\varepsilon}_h^{\Gamma}} \big)\\
	\lesssim & h \big(\lVert \nu^{-\frac{1}{2}}\bm{L}\rVert_{1,\Omega_{1}\cup\Omega_{2}} +  \lVert \nu^{-\frac{1}{2}}p\rVert_{1,\Omega_{1}\cup\Omega_{2}}\big)\big(\lVert \tau^{\frac{1}{2}}(\bm{Q}_0^b\bm{e}_h^u-\bm{e}_h^{\hat{u}}) \rVert_{0,\partial\mathcal{T}_h\setminus\mathcal{\varepsilon}_h^{\Gamma}} +\lVert \tau^{\frac{1}{2}}(\bm{Q}_m^b\bm{e}_h^u-\bm{e}_h^{\tilde{u}}) \rVert_{\mathcal{\varepsilon}_h^{\Gamma}} \\
	&\quad + \lVert \tau^{\frac{1}{2}}(\bm{e}_h^u-\bm{Q}_0^b\bm{e}_h^u) \rVert_{0,\partial\mathcal{T}_h\setminus\mathcal{\varepsilon}_h^{\Gamma}} +\lVert \tau^{\frac{1}{2}}(\bm{e}_h^u-\bm{Q}_m^b\bm{e}_h^u) \rVert_{\mathcal{\varepsilon}_h^{\Gamma}} \big)\\
	\lesssim& h \big(\lVert \nu^{-\frac{1}{2}}\bm{L}\rVert_{1,\Omega_{1}\cup\Omega_{2}}  + \lVert \nu^{-\frac{1}{2}}p\rVert_{1,\Omega_{1}\cup\Omega_{2}}\big) \big( \interleave (\bm{e}_h^L,\bm{e}_h^u,\bm{e}_h^{\hat{u}} ,\bm{e}_h^{\tilde{u}})\interleave  + \lVert \nu^{\frac{1}{2}}\nabla_h \bm{e}_h^u\rVert_{0,\mathcal{T}_h}  \big)\\
	\lesssim& h\big(\lVert \nu^{\frac{1}{2}}\bm{u}\rVert_{2,\Omega_{1}\cup\Omega_{2}} + \lVert \nu^{-\frac{1}{2}}p\rVert_{1,\Omega_{1}\cup\Omega_{2}}\big)\interleave (\bm{e}_h^L,\bm{e}_h^u,\bm{e}_h^{\hat{u}} ,\bm{e}_h^{\tilde{u}})\interleave .
	\end{align*}
	Similarly,  we have
	\begin{align*}
	E_2 
	\leq& \big( \lVert \tau (\bm{u}-\bm{Q}_1\bm{u}) \rVert_{0,\partial\mathcal{T}_h\setminus \mathcal{\varepsilon}_h^{\Gamma}} +\lVert \tau (\bm{u}-\bm{Q}_1\bm{u})\rVert_{0, \mathcal{\varepsilon}_h^{\Gamma}} \big) \big( \lVert \tau^{\frac{1}{2}}(\bm{e}_h^u-\bm{e}_h^{\hat{u}}) \rVert_{0,\partial\mathcal{T}_h\setminus\mathcal{\varepsilon}_h^{\Gamma}} +\lVert \tau^{\frac{1}{2}}(\bm{e}_h^u-\bm{e}_h^{\tilde{u}}) \rVert_{\mathcal{\varepsilon}_h^{\Gamma}} \big)\\
	\lesssim& h \lVert \nu^{\frac{1}{2}}\bm{u}\rVert_{2,\Omega_{1}\cup\Omega_{2}} \interleave (\bm{e}_h^L,\bm{e}_h^u,\bm{e}_h^{\hat{u}} ,\bm{e}_h^{\tilde{u}})\interleave.
	\end{align*}
As a result, the desired conclusion follows.
\end{proof}

\begin{lem}\label{est_p}
	Under the same conditions as in Lemma \ref{estimateofenergynorm}, it holds 
	\begin{align}\label{estimatep}
	\lVert e_h^p\rVert_{0,\mathcal{T}_h} \lesssim  h(\nu_{max}^{\frac{1}{2}}+\alpha_{max}^{\frac{1}{2}})
	\left(\lVert \nu^{\frac{1}{2}}\bm{u}\rVert_{2,\Omega_{1}\cup\Omega_{2}}+\lVert \nu^{-\frac{1}{2}}p\rVert_{1,\Omega_{1}\cup\Omega_{2}} \right ).
	\end{align}
	Here $
	\nu_{max} = \max\limits_{i=1,2}{\nu_i}$ and $\alpha_{max} = \max\limits_{i=1,2}{\alpha_i}$. 
\end{lem}
\begin{proof}
	Since $e_h^p \in L_0^2(\Omega)$,   there exists  $\bm{v^*}\in H_0^1(\Omega)$ such that 
	\begin{align}\label{37}
	\lVert e_h^p\rVert_{0,\mathcal{T}_h} \lesssim 
	\frac{(\nabla\cdot\bm{v^*},e_h^p)}{\lVert \bm{v^*}\rVert_{1,\Omega_{1}\cup\Omega_{2}}}.
	\end{align}
	In view of   integration by parts, the  properties of projections, and  \eqref{errorequation2},  \eqref{errorequation4} and  \eqref{errorequation5} with   $(\bm{v},\bm{\mu},\bm{\tilde{\mu}}) = ( \bm{\bm{Q}_1v^*}, \bm{Q}_0^b\bm{v^*},\bm{Q}_m^b\bm{v^*})$,  we have 
	\begin{align*}
	(\nabla\cdot\bm{v^*},e_h^p)_{\mathcal{T}_h} =& \langle \bm{Q}_0^b\bm{v^*},e_h^p \bm{n}\rangle_{\partial \mathcal{T}_h\setminus\mathcal{\varepsilon}_h^{\Gamma}}+\langle \bm{Q}_m^b\bm{v^*} ,e_h^p\bm{n}\rangle_{*,\mathcal{\varepsilon}_h^{\Gamma}}=T_1+T_2+T_3, 
	\end{align*}
%
where
	\begin{align*}
	T_1 = & \langle \bm{e}_h^L\bm{n},\bm{Q}_0^b\bm{v^*}\rangle_{\partial\mathcal{T}_h\setminus \mathcal{\varepsilon}_h^{\Gamma}}+\langle \bm{e}_h^L\bm{n},\bm{Q}_m^b\bm{v^*}\rangle_{*,\mathcal{\varepsilon}_h^{\Gamma}} +(\alpha \bm{e}_h^u,\bm{Q}_1\bm{v^*})_{\mathcal{T}_h},\\ 
	T_2 = & \langle \tau(\bm{Q}_0^b\bm{e}_h^u-\bm{e}_h^{\hat{u}}),\bm{Q}_1\bm{v^*}-\bm{Q}_0^b\bm{v^*}\rangle_{\partial \mathcal{T}_h\setminus \mathcal{\varepsilon}_h^{\Gamma}}+\langle \tau(\bm{Q}_m^b\bm{e}_h^u-\bm{e}_h^{\tilde{u}}),\bm{Q}_1\bm{v^*}-\bm{Q}_m^b\bm{v^*}\rangle_{*,\mathcal{\varepsilon}_h^{\Gamma}}, \\
	T_3 = & -L_1(\bm{Q}_1\bm{v^*}-\bm{Q}_0^b\bm{v^*}) -L_2(\bm{Q}_1\bm{v^*}-\bm{Q}_m^b\bm{v^*}) .
	\end{align*}
	 From integration by parts, the Cauchy-Schwarz inequality and the properties of projections it follows
	\begin{align*}
	T_1 =&\ \ (\bm{e}_h^L,\nabla_h\bm{Q}_1\bm{v^*})_{\mathcal{T}_h}+\langle \bm{e}_h^L\bm{n},\bm{Q}_1\bm{v^*}-\bm{Q}_0^b\bm{v^*}\rangle_{\partial\mathcal{T}_h\setminus \mathcal{\varepsilon}_h^{\Gamma}}+\langle \bm{e}_h^L\bm{n},\bm{Q}_1\bm{v^*}-\bm{Q}_m^b\bm{v^*}\rangle_{*,\mathcal{\varepsilon}_h^{\Gamma}}+(\alpha \bm{e}_h^u,\bm{Q}_1\bm{v^*})_{\mathcal{T}_h} \\
	\lesssim &\ \ \left( \nu_{max}^{\frac{1}{2}}\lVert \nu^{-\frac{1}{2}}\bm{e}_h^L\rVert_{0,\mathcal{T}_h}
	+ \alpha_{max}^{\frac{1}{2}}\lVert \alpha^{\frac{1}{2}}\bm{e}_h^u\rVert_{0,\mathcal{T}_h}\right)\lVert \bm{v^*}\rVert_{1,\Omega} \\
	 \lesssim& \ \ (\nu_{max}^{\frac{1}{2}}+\alpha_{max}^{\frac{1}{2}})\interleave (\bm{e}_h^L,\bm{e}_h^u,\bm{e}_h^{\hat{u}} ,\bm{e}_h^{\tilde{u}})\interleave \lVert  \bm{v^*}\rVert_{1,\Omega} , \\
	T_2 \lesssim &\ \  \nu_{max}^{\frac{1}{2}}\interleave  (\bm{e}_h^L,\bm{e}_h^u,\bm{e}_h^{\hat{u}} ,\bm{e}_h^{\tilde{u}})\interleave \lVert \bm{v^*}\rVert_{1,\Omega},\\
	T_3  
	\lesssim& \ \  \nu_{max}^{\frac{1}{2}}h\left(\lVert \nu^{-\frac{1}{2}}\bm{L}\rVert_{1,\Omega_{1}\cup\Omega_{2}}+\lVert \nu^{\frac{1}{2}}\bm{u}\rVert_{2,\Omega_{1}\cup\Omega_{2}}+\lVert \nu^{-\frac{1}{2}}p\rVert_{1,\Omega_{1}\cup\Omega_{2}} \right) \lVert \bm{v^*}\rVert_{1,\Omega} .
	\end{align*}
	So by \eqref{37} and the relation \eqref{firstorderscheme-a} we have
	\begin{align*}
	\lVert e_h^p\rVert_{0,\mathcal{T}_h}  \lesssim (\nu_{max}^{\frac{1}{2}}+\alpha_{max}^{\frac{1}{2}})\left(
	\interleave (\bm{e}_h^L,\bm{e}_h^u,\bm{e}_h^{\hat{u}} ,\bm{e}_h^{\tilde{u}})\interleave+ h( \lVert \nu^{\frac{1}{2}}\bm{u}\rVert_{2,\Omega_{1}\cup\Omega_{2}}+\lVert \nu^{-\frac{1}{2}}p\rVert_{1,\Omega_{1}\cup\Omega_{2}}\right),
	\end{align*}
	which, together with Lemma \ref{estimateofenergynorm}, yields the desired result \eqref{estimatep}.
\end{proof}

Based on Lemmas \ref{ineq}, \ref{estimateofenergynorm}, \ref{est_p}, and the triangle inequality,  we can easily obtain the following main result.
\begin{thm}\label{1norm}
	Let $(\bm{L},\bm{u},p)\in {H^{1}}(\Omega_1\cup \Omega_2)^{d\times d}\times {H^{2}}(\Omega_1\cup \Omega_2)^d\times H^1(\Omega_1\cup \Omega_2) $ and $(\bm{L}_h,\bm{u}_h,p_h,\bm{\hat{u}}_h,\bm{\tilde{u}}_h)\in \bm{W}_h\times V_h\times Q_h\times M_h(g)\times \tilde{M}_h$ be the solutions of the  problem \eqref{firstorderscheme} and the X-HDG scheme \eqref{xhdgscheme}, respectively. Then it holds
\begin{align}
	 \lVert \nu^{-\frac{1}{2}}(\bm{L}-\bm{L}_h)\rVert_{0,\mathcal{T}_h}& +\lVert \nu^{\frac{1}{2}}(\nabla \bm{u}- \nabla_h\bm{u}_h)\rVert_{0,\mathcal{T}_h}
	 + (\nu_{max}^{\frac{1}{2}}+\alpha_{max}^{\frac{1}{2}})^{-1} \lVert p-p_h\rVert_{0,\mathcal{T}_h} \nonumber \\	
	& \lesssim  h \left(\lVert \nu^{\frac{1}{2}}\bm{u}\rVert_{2,\Omega_{1}\cup\Omega_{2}} +\lVert \nu^{-\frac{1}{2}}p\rVert_{1,\Omega_{1}\cup\Omega_{2}}\right),\\
		  \lVert \alpha^{\frac{1}{2}}(  \bm{u}- \bm{u}_h)\rVert_{0,\mathcal{T}_h} 
	&\lesssim  h \left(\lVert \nu^{\frac{1}{2}}\bm{u}\rVert_{2,\Omega_{1}\cup\Omega_{2}}+\lVert \alpha^{\frac{1}{2}}\bm{u}\rVert_{1,\Omega_{1}\cup\Omega_{2}}+\lVert \nu^{-\frac{1}{2}}p\rVert_{1,\Omega_{1}\cup\Omega_{2}}\right).
	\end{align}
\end{thm}

\subsection{$L^2$ error estimation for velocity}
In this subsection, we shall derive an $L^2$ error estimate for the  velocity approximation by    the Aubin-Nitsche's technique of duality argument. To this end, we introduce the auxiliary problem
\begin{subequations}\label{dual}
	\begin{align}
	\bm{\Phi} = \nu\nabla \bm{ \phi} &\quad   \text{in} \quad\Omega_1\cup\Omega_2,\label{dual1}\\
	- \nabla\cdot\bm{\Phi}+\nabla \psi  + \alpha \bm{\phi}= \bm{e}_h^{u} & \quad  \text{in} \quad\Omega_1\cup\Omega_2,\label{dual2}\\
	\nabla\cdot\bm{\phi} = 0 & \quad  \text{in} \quad\Omega_1\cup\Omega_2,\label{dual3}\\
	\bm{\phi}=\bm{0}& \quad \text{on} \quad \partial\Omega,\label{dual4}\\
	\llbracket\bm{\phi}\rrbracket= \bm{0},\ \llbracket(\bm{\Phi}-\psi\bm{I})\bm{n}\rrbracket = \bm{0}&\quad \text{on}\quad \Gamma\label{dual5} ,
	\end{align}
\end{subequations}
and assume the following regularity estimate: 
\begin{align}\label{regularestimate}
\lVert \bm{\Phi}\rVert_{1,\Omega_1\cup\Omega_2} +\lVert \nu\bm{\phi}\rVert_{2,\Omega_1\cup\Omega_2} +\lVert \alpha\bm{\phi}\rVert_{2,\Omega_1\cup\Omega_2} +\lVert \psi\rVert_{1,\Omega_1\cup\Omega_2} \lesssim \lVert \bm{e}_h^{u}\rVert_{0,\mathcal{T}_h}.
\end{align}
Here we recall that $\bm{e}_h^u=\bm{u}_h - \bm{Q}_1\bm{u}$.
\begin{thm}\label{velocity0norm}
Let $(\bm{L},\bm{u},p)\in {H^{1}}(\Omega_1\cup \Omega_2)^{d\times d}\times {H^{2}}(\Omega_1\cup \Omega_2)^d\times 
H^1(\Omega_1\cup \Omega_2) $ and $(\bm{L}_h,\bm{u}_h,p_h,\bm{\hat{u}}_h,\bm{\tilde{u}}_h)\in \bm{W}_h\times 
V_h\times Q_h\times M_h(g)\times \tilde{M}_h$ be the solutions of the problem \eqref{firstorderscheme} and the X-HDG scheme 
\eqref{xhdgscheme}, respectively.  Then, under the assumption \eqref{regularestimate}  it holds 
\begin{align}\label{est_L2}
\lVert \bm{u}-\bm{u}_h\rVert_{0,\mathcal{T}_h} \lesssim h^2\nu_{min}^{-\frac{1}{2}}\left(\lVert \nu^{\frac{1}{2}}\bm{u}\rVert_{2,\Omega_{1}\cup\Omega_{2}}+\lVert \nu^{-\frac{1}{2}}p\rVert_{1,\Omega_{1}\cup\Omega_{2}}\right )
\end{align}
for  (i) $m = 1$ or (ii) $m = 0$ and 
$\bm{g}_N^{\Gamma}  $ is a constant on any $F \in  \mathcal{\varepsilon}_h^{\Gamma} $. Here $\nu_{min}: = \min\limits_{i=1,2}{\nu_i}$.
\end{thm}
\begin{proof}
Testing the equations \eqref{dual2} and \eqref{dual3} by $\bm{e}_h^u$ and $e_h^p$, respectively,   adding them up,  and using  integration by parts and the properties of projections, we obtain 
\begin{align*}
\lVert \bm{e}_h^u\rVert_{0,\mathcal{T}_h}^2 
=& -(\nabla_h\cdot \bm{\Phi},\bm{e}_h^u)_{\mathcal{T}_h}+(\nabla_h\psi,\bm{e}_h^u)_{\mathcal{T}_h}+
(\alpha\bm{\phi},\bm{e}_h^u)_{\mathcal{T}_h}-(\nabla_h\cdot \bm{\phi},e_h^p)_{\mathcal{T}_h} \\
=&\langle (\bm{Q}_0\bm{\Phi}-\bm{\Phi})\bm{n},\bm{e}_h^u\rangle_{\partial \mathcal{T}_h\setminus 
\mathcal{\varepsilon}_h^{\Gamma}} +\langle (\bm{Q}_0\bm{\Phi}-\bm{\Phi})\bm{n},\bm{e}_h^u\rangle_{*,\mathcal{\varepsilon}_h^{\Gamma}} +\langle \psi -Q_0\psi,\bm{e}_h^u\cdot\bm{n}\rangle_{\partial \mathcal{T}_h\setminus \mathcal{\varepsilon}_h^{\Gamma}}\\
&+\langle \psi 
-Q_0\psi,\bm{e}_h^u\cdot\bm{n}\rangle_{*,\mathcal{\varepsilon}_h^{\Gamma}} +(\alpha\bm{Q}_1\bm{\phi},\bm{e}_h^u)_{\mathcal{T}_h} -\langle e_h^p\bm{n},\bm{Q}_0^b\bm{\phi}\rangle_{\partial \mathcal{T}_h\setminus \mathcal{\varepsilon}_h^{\Gamma}} -\langle e_h^p\bm{n},\bm{Q}_m^b\bm{\phi}
\rangle_{*,\mathcal{\varepsilon}_h^{\Gamma}} .
\end{align*} 
Due to the  error equations 
\eqref{errorequation1} and \eqref{errorequation3}, we have 
\begin{align*}
	(\nu^{-1}\bm{e}_h^L,\bm{Q}_0\bm{\Phi})_{\mathcal{T}_h} 
	-\langle \bm{e}_h^{\hat{u}},\bm{Q}_0\bm{\Phi}\bm{n} \rangle_{\partial\mathcal{T}_h\setminus \mathcal{\varepsilon}_h^{\Gamma}} -\langle \bm{e}_h^{\tilde{u}},\bm{Q}_0\bm{\Phi}\bm{n} \rangle_{*,\mathcal{\varepsilon}_h^{\Gamma}} &= 0 , \\
	\langle\bm{e}_h^{\hat{u}}\cdot\bm{n},Q_0\psi\rangle_{\partial\mathcal{T}_h\setminus \mathcal{\varepsilon}_h^{\Gamma}} + \langle\bm{e}_h^{\tilde{u}}\cdot\bm{n},Q_0\psi\rangle_{*,\mathcal{\varepsilon}_h^{\Gamma}}&= 0,
\end{align*}
which, together  with  the fact   $ \llbracket(\bm{\Phi}-\psi\bm{I})\bm{n}\rrbracket_\Gamma= \bm{0}$,  imply
\begin{align*}
\lVert \bm{e}_h^u\rVert_{0,\mathcal{T}_h}^2 
=&\langle (\bm{Q}_0\bm{\Phi}-\bm{\Phi})\bm{n},\bm{e}_h^u-\bm{e}_h^{\hat{u}}\rangle_{\partial \mathcal{T}
_h\setminus \mathcal{\varepsilon}_h^{\Gamma}} +\langle (\bm{Q}_0\bm{\Phi}-\bm{\Phi})\bm{n},\bm{e}_h^u-\bm{e}_h^{\tilde{u}}
\rangle_{*,\mathcal{\varepsilon}_h^{\Gamma}} \\
&+\langle \psi -Q_0\psi,(\bm{e}_h^u-\bm{e}_h^{\hat{u}})\cdot\bm{n}\rangle_{\partial \mathcal{T}_h\setminus 
\mathcal{\varepsilon}_h^{\Gamma}}+\langle \psi -Q_0\psi,(\bm{e}_h^u-\bm{e}_h^{\tilde{u}})\cdot\bm{n}\rangle_{*,\mathcal{\varepsilon}_h^{\Gamma}} \\
&+
(\alpha \bm{Q}_1\bm{\phi},\bm{e}_h^u)_{\mathcal{T}_h}  - \langle e_h^p\bm{n},\bm{Q}_0^b\bm{\phi}\rangle_{\partial \mathcal{T}_h\setminus \mathcal{\varepsilon}_h^{\Gamma}} - \langle e_h^p\bm{n},
\bm{Q}_m^b\bm{\phi}\rangle_{*,\mathcal{\varepsilon}_h^{\Gamma}} +(\bm{Q}_0\bm{\Phi},\nu^{-1}\bm{e}_h^L)_{\mathcal{T}_h}.
\end{align*} 
Notice that by  \eqref{dual1}, the properties of projections and  integration by parts it holds
\begin{align*}
	(\bm{Q}_0\bm{\Phi},\nu^{-1}\bm{e}_h^L)_{\mathcal{T}_h} =&(\bm{\Phi},\nu^{-1}\bm{e}_h^L)_{\mathcal{T}_h} = (\bm{\phi},\bm{e}_h^L)_{\mathcal{T}_h}\\
	=& \langle \bm{\phi},\bm{e}_h^L\bm{n}
	\rangle_{\partial\mathcal{T}_h\setminus \mathcal{\varepsilon}_h^{\Gamma}} +\langle \bm{\phi},\bm{e}_h^L\bm{n}\rangle_{*,\mathcal{\varepsilon}_h^{\Gamma}} \\
	=&\langle \bm{Q}_0^b\bm{\phi},\bm{e}_h^L\bm{n}
	\rangle_{\partial\mathcal{T}_h\setminus \mathcal{\varepsilon}_h^{\Gamma}} +\langle \bm{Q}_m^b\bm{\phi},\bm{e}_h^L\bm{n}
	\rangle_{*,\mathcal{\varepsilon}_h^{\Gamma}},
\end{align*}
and that taking $(\bm{v},\bm{\mu},\bm{\tilde{\mu}})=(\bm{Q}_1\bm{\phi},\bm{Q}_0^b\bm{\phi},\bm{Q}_m^b\bm{\phi})$ in 
\eqref{errorequation2},\eqref{errorequation4}-\eqref{errorequation5} shows
\begin{align*}
(\alpha\bm{e}_h^u,\bm{Q}_1\phi)_{\mathcal{T}_h} =- \langle \tau(\bm{Q}_0^b\bm{e}_h^u-\bm{e}_h^{\hat{u}}),\bm{Q}
_1\bm{\phi}\rangle_{\partial\mathcal{T}_h\setminus \mathcal{\varepsilon}_h^{\Gamma}} 
- \langle \tau(\bm{Q}_m^b\bm{e}_h^u-\bm{e}_h^{\tilde{u}}),\bm{Q}_1\bm{\phi}\rangle_{*,\mathcal{\varepsilon}_h^{\Gamma}} &+ \sum_{i=1}
^{2}L_i(\bm{Q}_1\bm{\phi}) ,\\
\langle \bm{e}_h^L\bm{n},\bm{Q}_0^b\bm{\phi}\rangle_{\partial \mathcal{T}_h\setminus \mathcal{\varepsilon}_h^{\Gamma}}-\langle 
e_h^p\bm{n},\bm{Q}_0^b\bm{\phi}\rangle_{\partial \mathcal{T}_h\setminus \mathcal{\varepsilon}_h^{\Gamma}}=\langle \tau(\bm{Q}_0^b\bm{e}
_h^u-\bm{e}_h^{\hat{u}}),\bm{Q}_0^b\bm{\phi}\rangle_{\partial \mathcal{T}_h\setminus \mathcal{\varepsilon}_h^{\Gamma}}& -L_1(\bm{Q}
_0^b\bm{\phi}),\\
\langle \bm{e}_h^L\bm{n},\bm{Q}_m^b\bm{\phi}\rangle_{*,\mathcal{\varepsilon}_h^{\Gamma}}-\langle e_h^p\bm{n},\bm{Q}_m^b\bm{\phi}
\rangle_{*,\mathcal{\varepsilon}_h^{\Gamma}}=\langle \tau(\bm{Q}_m^b\bm{e}_h^u-\bm{e}_h^{\tilde{u}}),\bm{Q}_m^b\bm{\phi}\rangle_{*,\mathcal{\varepsilon}_h^{\Gamma}}
& -L_2(\bm{Q}_m^b\bm{\phi}).
\end{align*}
The four equations above mean that 
 \begin{align*}
 &(\alpha \bm{Q}_1\bm{\phi},\bm{e}_h^u)_{\mathcal{T}_h}  - \langle e_h^p\bm{n},\bm{Q}_0^b\bm{\phi}\rangle_{\partial \mathcal{T}_h\setminus \mathcal{\varepsilon}_h^{\Gamma}} - \langle e_h^p\bm{n},
\bm{Q}_m^b\bm{\phi}\rangle_{*,\mathcal{\varepsilon}_h^{\Gamma}} +(\bm{Q}_0\bm{\Phi},\nu^{-1}\bm{e}_h^L)_{\mathcal{T}_h}\\
=&\langle \tau(\bm{Q}_0^b\bm{e}_h^u-\bm{e}_h^{\hat{u}}),\bm{Q}_0^b\bm{\phi}-\bm{Q}_1\bm{\phi}
\rangle_{\partial\mathcal{T}_h\setminus \mathcal{\varepsilon}_h^{\Gamma}} + \langle \tau(\bm{Q}_m^b\bm{e}_h^u-\bm{e}_h^{\tilde{u}}),\bm{Q}
_m^b\bm{\phi}-\bm{Q}_1\bm{\phi}\rangle_{*,\mathcal{\varepsilon}_h^{\Gamma}} \\
&\ \ +L_1(\bm{Q}_1\bm{\phi}-\bm{Q}_0^b\bm{\phi})+L_2(\bm{Q}_1\bm{\phi}-\bm{Q}_m^b\bm{\phi}).
\end{align*} 
As a result, we obtain
%
\begin{align}\label{411}
\lVert \bm{e}_h^u\rVert_{0,\mathcal{T}_h}^2 = \sum_{j=1}^{4}I_j
\end{align}
with
\begin{align*}
I_1: =& \langle (\bm{Q}_0\bm{\Phi}-\bm{\Phi})\bm{n},\bm{e}_h^u-\bm{e}_h^{\hat{u}}\rangle_{\partial \mathcal{T}
_h\setminus \mathcal{\varepsilon}_h^{\Gamma}} +\langle (\bm{Q}_0\bm{\Phi}-\bm{\Phi})\bm{n},\bm{e}_h^u-\bm{e}_h^{\tilde{u}}
\rangle_{*,\mathcal{\varepsilon}_h^{\Gamma}}, \\
I_2: = &\langle \psi -Q_0\psi,(\bm{e}_h^u-\bm{e}_h^{\hat{u}})\cdot\bm{n}\rangle_{\partial \mathcal{T}_h\setminus 
\mathcal{\varepsilon}_h^{\Gamma}}+\langle \psi -Q_0\psi,(\bm{e}_h^u-\bm{e}_h^{\tilde{u}})\cdot\bm{n}\rangle_{*,\mathcal{\varepsilon}_h^{\Gamma}},\\
I_3: = &\langle \tau(\bm{Q}_0^b\bm{e}_h^u-\bm{e}_h^{\hat{u}}),\bm{Q}_0^b\bm{\phi}-\bm{Q}_1\bm{\phi}
\rangle_{\partial\mathcal{T}_h\setminus \mathcal{\varepsilon}_h^{\Gamma}} + \langle \tau(\bm{Q}_m^b\bm{e}_h^u-\bm{e}_h^{\tilde{u}}),\bm{Q}
_m^b\bm{\phi}-\bm{Q}_1\bm{\phi}\rangle_{*,\mathcal{\varepsilon}_h^{\Gamma}} , \\
I_4: =&L_1(\bm{Q}_1\bm{\phi}-\bm{Q}_0^b\bm{\phi})+L_2(\bm{Q}_1\bm{\phi}-\bm{Q}_m^b\bm{\phi}) .
\end{align*}
From the Cauchy-Schwarz inequality, Lemmas \ref{ineq} and \ref{estimateofenergynorm}, and the regularity assumption
\eqref{regularestimate}  it follows
\begin{align}\label{412}
I_1 +I_2+I_3
\lesssim & \ h\nu_{min}^{-\frac{1}{2}}\left(\lVert \bm{\Phi}\rVert_{1,\Omega_{1}\cup\Omega_{2}}+\lVert \nu\bm{\phi}\rVert_{2,\Omega_{1}\cup\Omega_{2}} +\lVert \psi\rVert_{1,\Omega_{1}\cup\Omega_{2}} \right) \interleave (\bm{e}_h^L,\bm{e}
_h^u,\bm{e}_h^{\hat{u}},\bm{e}_h^{\tilde{u}})\interleave\nonumber \\
\lesssim &\  h^2 \nu_{min}^{-\frac{1}{2}} \lVert \bm{e}_h^u\rVert_{0,\mathcal{T}_h}\left(\lVert \nu^{\frac{1}{2}}\bm{u}\rVert_{2,\Omega_{1}\cup\Omega_{2}}+\lVert \nu^{-\frac{1}{2}}p\rVert_{1,\Omega_{1}\cup\Omega_{2}}\right ).
\end{align}

In light of the property of projection, it holds
\begin{align}\label{413}
I_4= &\  \left(\langle (\bm{Q}_0\bm{L}-\bm{L})\bm{n},\bm{Q}_1\bm{\phi}-\bm{\phi}\rangle_{\partial \mathcal{T}
_h\setminus \mathcal{\varepsilon}_h^{\Gamma}}+\langle \tau (\bm{Q}_0^b(\bm{u}-\bm{Q}_1\bm{u})),(\bm{Q}_1\bm{\phi}-\bm{\phi})+(\bm{\phi}-\bm{Q}_0^b\bm{\phi} )
\rangle_{\partial \mathcal{T}_h\setminus \mathcal{\varepsilon}_h^{\Gamma}} \right.\nonumber \\
& \ +\langle (Q_0p-p)\bm{n},\bm{Q}_1\bm{\phi}-\bm{\phi} 
\rangle_{\partial \mathcal{T}_h\setminus \mathcal{\varepsilon}_h^{\Gamma}}  + \langle (\bm{Q}_0\bm{L}-\bm{L})\bm{n},\bm{Q}_1\bm{\phi}-\bm{\phi}\rangle_{*,\mathcal{\varepsilon}_h^{\Gamma}}\nonumber \\
&\ \left.+  \langle \tau (\bm{Q}_m^b(\bm{u}-\bm{Q}_1\bm{u})),(\bm{Q}_1\bm{\phi}-\bm{\phi})+(\bm{\phi}-\bm{Q}_m^b\bm{\phi}) \rangle_{*,\mathcal{\varepsilon}_h^{\Gamma}}+\langle (Q_0p-p)
\bm{n},\bm{Q}_1\bm{\phi}-\bm{\phi}\rangle_{*,\mathcal{\varepsilon}_h^{\Gamma}} \right)\nonumber\\
&\ +\left(\langle (\bm{Q}_0\bm{L}-\bm{L})\bm{n},\bm{\phi}-\bm{Q}_0^b\bm{\phi}\rangle_{\partial \mathcal{T}
_h\setminus \mathcal{\varepsilon}_h^{\Gamma}}+\langle (Q_0p-p)\bm{n},\bm{\phi}-\bm{Q}_0^b\bm{\phi} \rangle_{\partial \mathcal{T}
_h\setminus \mathcal{\varepsilon}_h^{\Gamma}}\right) \nonumber\\
&\ +\left( \langle (\bm{Q}_0\bm{L}-\bm{L})\bm{n},\bm{\phi}-\bm{Q}_m^b\bm{\phi}\rangle_{*,\mathcal{\varepsilon}_h^{\Gamma}} +\langle (Q_0p-p)\bm{n}, \bm{\phi}-\bm{Q}_m^b\bm{\phi}\rangle_{*,\mathcal{\varepsilon}_h^{\Gamma}} \right)\nonumber\\
=:&\  \tilde{I_1} + \tilde{I_2} +\tilde{I_3} . 
\end{align}
Again by Lemma \ref{ineq} and \eqref{regularestimate} we get
\begin{align}
\tilde{I_1}
%
\lesssim &\ h^{2}\lVert \nu\bm{\phi}\rVert_{2,\Omega_{1}\cup\Omega_{2}}
(\lVert \nu^{-1}\bm{L}\rVert_{1,\Omega_{1}\cup\Omega_{2}} +\lVert \bm{u}
\rVert_{2,\Omega_{1}\cup\Omega_{2}} + \lVert\nu^{-1} p\rVert_{1,\Omega_{1}\cup\Omega_{2}}) \nonumber \\
\lesssim&\  h^2 \nu_{min}^{-\frac{1}{2}} \lVert \bm{e}_h^u\rVert_{0,\mathcal{T}_h}\left(\lVert \nu^{\frac{1}{2}}\bm{u}\rVert_{2,\Omega_{1}\cup\Omega_{2}}+\lVert \nu^{-\frac{1}{2}}p\rVert_{1,\Omega_{1}\cup\Omega_{2}}\right ) ,
\end{align}
and, for case (i) with $m = 1$,  
\begin{align}\label{416}
\tilde{I_3} =&\  \langle (\bm{Q}_0\bm{L}-\bm{L})\bm{n},\bm{\phi}-\bm{Q}_1^b\bm{\phi}\rangle_{*,\mathcal{\varepsilon}_h^{\Gamma}} +\langle (Q_0p-p)\bm{n}, \bm{\phi}-\bm{Q}_1^b\bm{\phi}\rangle_{*,\mathcal{\varepsilon}_h^{\Gamma}}\nonumber\\
\lesssim&\  h^2\lVert \nu\bm{\phi}\rVert_{2,\Omega_{1}\cup\Omega_{2}}(\lVert \nu^{-1}\bm{L}
\rVert_{1,\Omega_{1}\cup\Omega_{2}} + \lVert\nu^{-1} p\rVert_{1,\Omega_{1}\cup\Omega_{2}}) \nonumber \\
\lesssim&\  h^2 \nu_{min}^{-\frac{1}{2}} \lVert \bm{e}_h^u\rVert_{0,\mathcal{T}_h}\left(\lVert \nu^{\frac{1}{2}}\bm{u}\rVert_{2,\Omega_{1}\cup\Omega_{2}}+\lVert \nu^{-\frac{1}{2}}p\rVert_{1,\Omega_{1}\cup\Omega_{2}}\right ).
\end{align}
Since  $\llbracket (\bm{L}-p\bm{I})\bm{n} \rrbracket _F =0 $  for $F \in \mathcal{\varepsilon}_h\setminus\mathcal{\varepsilon}_h^{\Gamma} $ with $F\nsubseteq \partial\Omega$, and $
\bm{\phi} |_{\partial\Omega}= \bm{0}$,  we have
\begin{align}
\tilde{I_2} 
=&\  \langle (\bm{Q}_0\bm{L}-\bm{L})\bm{n},\bm{\phi}-\bm{Q}_0^b\bm{\phi}\rangle_{\partial \mathcal{T}
_h\setminus \mathcal{\varepsilon}_h^{\Gamma}}+\langle (Q_0p-p)\bm{n},\bm{\phi}-\bm{Q}_0^b\bm{\phi} \rangle_{\partial \mathcal{T}
_h\setminus \mathcal{\varepsilon}_h^{\Gamma}} \nonumber\\
=&\   \langle (\bm{Q}_0\bm{L},\bm{\phi}-\bm{Q}_0^b\bm{\phi}\rangle_{\partial \mathcal{T}
_h\setminus \mathcal{\varepsilon}_h^{\Gamma}}+\langle (Q_0p,\bm{\phi}-\bm{Q}_0^b\bm{\phi} \rangle_{\partial \mathcal{T}
_h\setminus \mathcal{\varepsilon}_h^{\Gamma}}\nonumber\\
=&\ 0 .
\end{align}
Similarly, for case (ii) with  $m = 0$ and $\llbracket (\bm{L}-p\bm{I})\bm{n} \rrbracket  =\bm{g}_N^{\Gamma} =const.$ on any $F \in  \mathcal{\varepsilon}_h^{\Gamma} $,  it holds
\begin{align}\label{416}
\tilde{I_3} 
=& \langle (\bm{Q}_0\bm{L}-\bm{L})\bm{n},\bm{\phi}-\bm{Q}_0^b\bm{\phi}\rangle_{*,\mathcal{\varepsilon}_h^{\Gamma}} +\langle (Q_0p-p)\bm{n}, \bm{\phi}-\bm{Q}_0^b\bm{\phi}\rangle_{*,\mathcal{\varepsilon}_h^{\Gamma}} = 0.
\end{align}

Finally, combining \eqref{411}-\eqref{416}  completes the proof.
\end{proof}
\begin{rem}
	For the Stokes equation with $\nu_1 = \nu_2 = \nu$ and $\alpha=0$, from Theorems \ref{1norm} and \ref{velocity0norm} we easily obtain 
%
	\begin{align*}
		\lVert \bm{u}-\bm{u}_h\rVert_{0,\mathcal{T}_h}+h\lVert \nabla \bm{u}- \nabla_h\bm{u}_h\rVert_{0,\mathcal{T}_h} +\nu^{-1}h\lVert p-p_h\rVert_{0,\mathcal{T}_h}&\lesssim 
		h^2(\lVert  \bm{u}\rVert_{2,\Omega_{1}\cup\Omega_{2}} + \nu^{-1}\lVert  p\rVert_{1,\Omega_{1}\cup\Omega_{2}}) .
	\end{align*}
\end{rem}

\section{Application of X-HDG method to curved domains}
Let  $\Omega\subset \mathbb{R}^d (d = 2, 3) $  be  a curved domain with piecewise smooth boundary.  Consider the     following   problem: 
\begin{align}
\label{pb2}
\left \{
\begin{array}{rl}
- \nu\Delta \bm{u}+\nabla p +\alpha\bm{u}= \bm{f}&   \text{in} \quad\Omega,\\ 
\nabla\cdot \bm{u} = 0&   \text{in} \quad\Omega,\\ 
\bm{u}=\bm{g}_D&  \text{on} \quad  \partial\Omega.
\end{array}
\right.
\end{align}
Here  $\nu>0$ and   $\alpha\geq 0$ are two constants,  and $\bm{g}_D$ satisfies the compatibility condition \eqref{13-gD}.

\begin{figure}[htp]	
	\centering
	\includegraphics[height = 4.5 cm,width=5  cm]{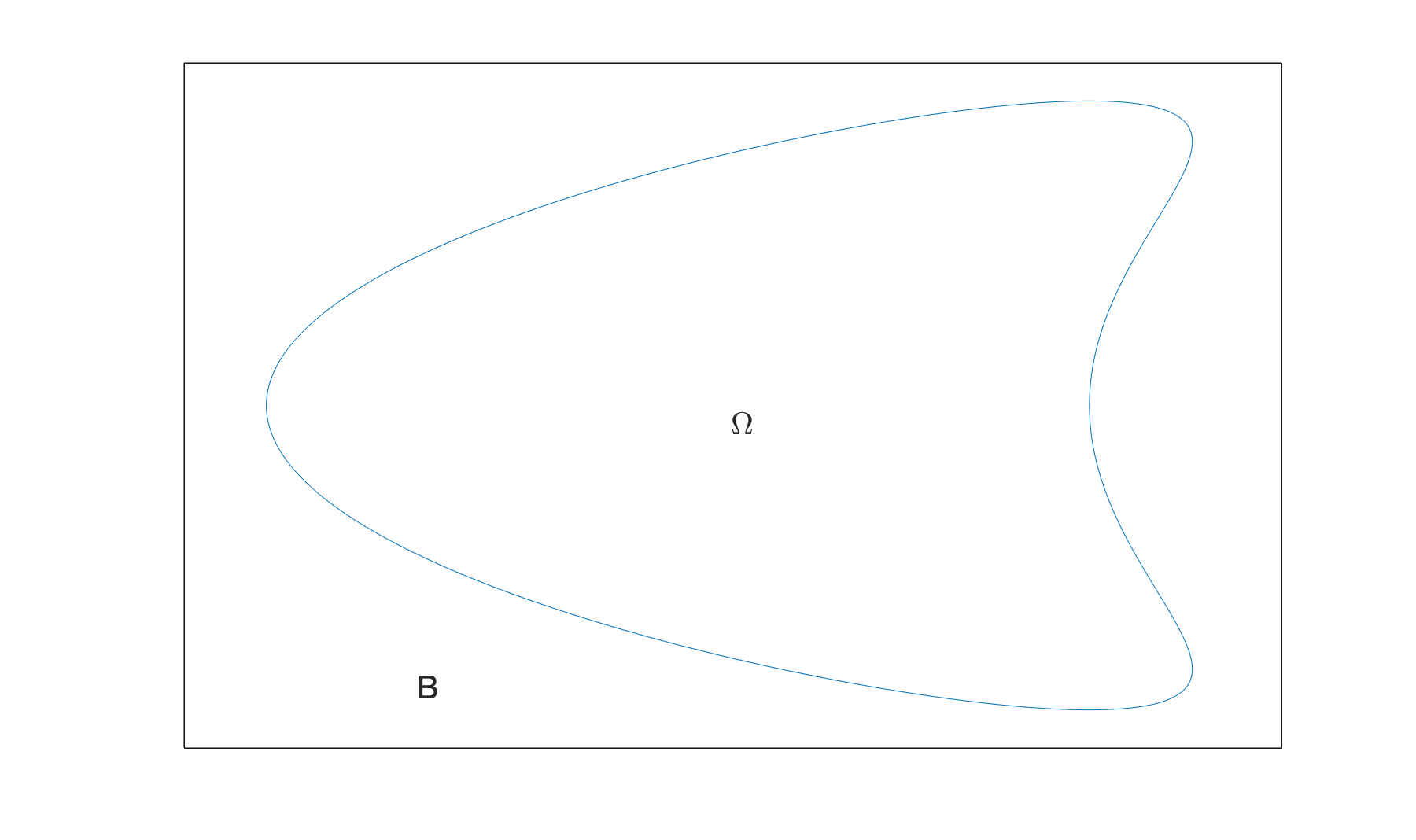} 
	\includegraphics[height = 4.5 cm,width=5  cm]{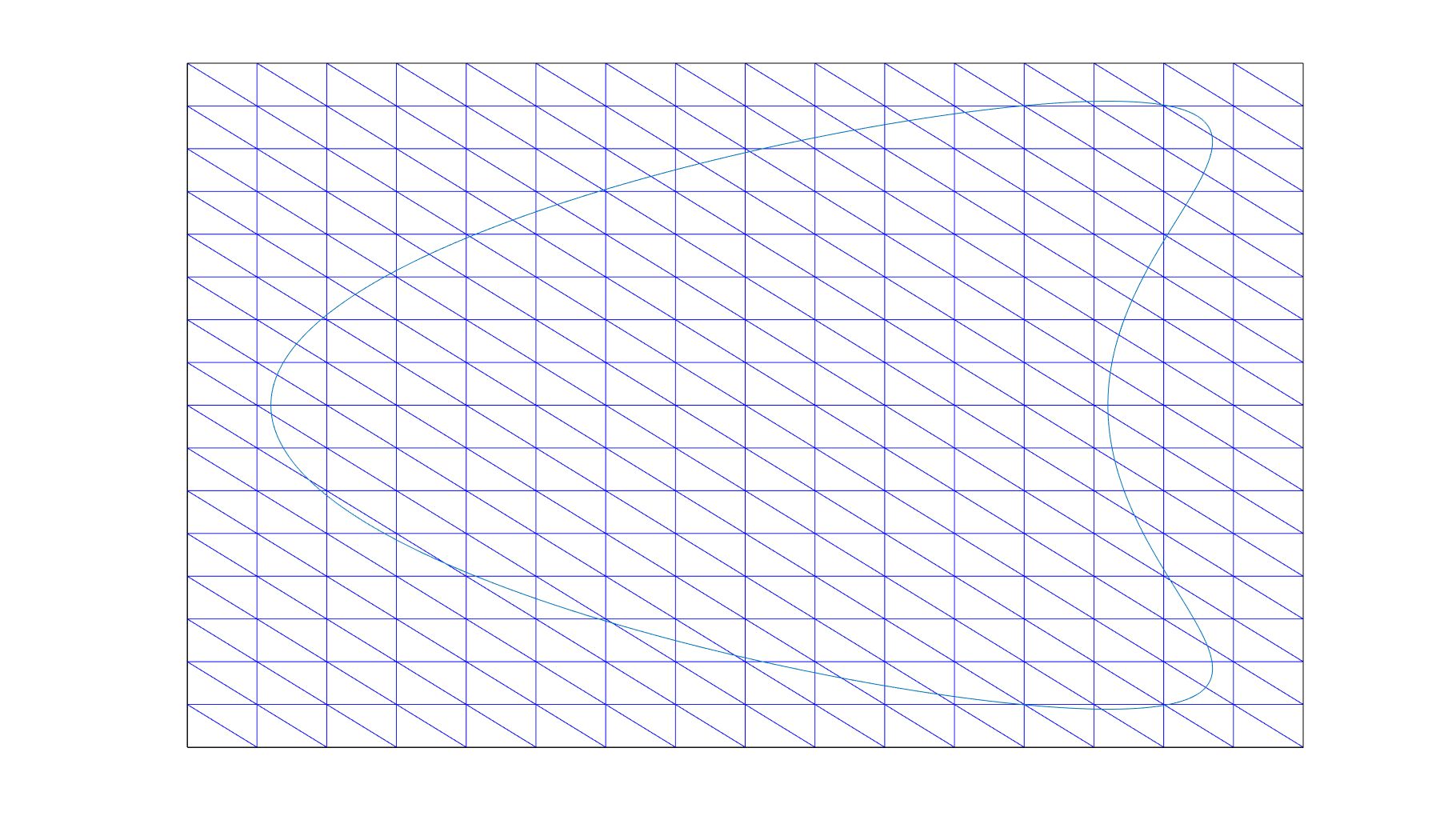} 
	\caption{The geometry of a curved domain (left) and a boundary unfitted mesh (right).}\label{domain1}
\end{figure}

Let $\mathbb{B}\supset \Omega$ be a simpler  domain than $\Omega$ (cf. Figure \ref{domain1}), and denote $\Omega^{c} := \mathbb{B}\backslash \bar{\Omega}$.  Then we can rewrite  problem \eqref{pb2} as an interface problem:  
\begin{subequations}\label{pb3}
	\begin{align} 
	- \nu\Delta \bm{u}+\nabla p + \alpha\bm{u}= \chi_{\Omega}\bm{f} , &\text { in } \Omega\cup\Omega^{c} , \\ 
	\nabla\cdot \bm{u} = 0, &\text { in } \Omega \cup\Omega^{c},\\
	\llbracket \bm{u} \rrbracket=\bm{g}_{D}, &\text { on } \Gamma: =\partial\Omega, \\
	\bm{u} \equiv \bm{0}, \ p  \equiv 0,  & \text { in } \Omega^{c}. \label{43d}
	\end{align}
\end{subequations}
Here $\chi_{\Omega}$ is the characteristic function on $\Omega$, which satisfying $\chi_{\Omega} =1$ in $\Omega$ and $\chi_{\Omega} =0$ in $\Omega^{c}$.
We note that the problem \eqref{pb3} is a special interface problem with $\partial \Omega$ being the interface, for which we  only need to approximate the solution in $\Omega$ due to \eqref{43d}.

Let $\mathcal{T}_h=\cup\{K\}$ be a shape-regular partition of the domain $\Omega$ consisting of arbitrary open polygons/polyhedrons. 
For any $K$ satisfing $K\cap \Gamma \neq \emptyset$, called an boundary element, let  $\Gamma_K := K\cap \Gamma$ be the part of $ \Gamma$ in $K$, and   $\Gamma_{K,h}$ be the straight line/plane segment connecting the intersection between $\Gamma_K$ and $\partial K$.  

Define the following sets of elements or edges/faces:
\begin{align*}
\mathcal{T}_h^{i} :=&\{K\in \mathcal{T}_h: K\cap\Omega = K\}, \\
\mathcal{T}_h^{\Gamma} :=&\{K\cap \Omega:  \forall K\in \mathcal{T}_h \ \text{with}\ K\cap\partial\Omega \neq \emptyset\}, \\
\mathcal{T}_h^{*} :=& \mathcal{T}_h^{i}\cup \mathcal{T}_h^{\Gamma},\\
\mathcal{\varepsilon}_h^{i} :=&\{F\cap \Omega:  \forall  \text{  edge/face $F$ of all elements in}\  \mathcal{T}_h \}, \\
\mathcal{\varepsilon}_h^{\Gamma} :=& \{F:\ F = \Gamma_{K,h}, \forall K\in \mathcal{T}_h^{\Gamma}, \ \text{or} \ F \ \text{is an edge/face of some $K$} \in \mathcal{T}_h^{i} \ \text{with} \ F\subset \bar{K}\cap \partial\Omega\},\\
\mathcal{\varepsilon}_h :=& \mathcal{\varepsilon}_h^{i}\cup \mathcal{\varepsilon}_h^{\Gamma} .
\end{align*} 
We  also introduce  the following X-HDG finite element spaces:
\begin{align*}
\bm{W}_h :=& \{\bm{w}\in L^2(\Omega)^{d\times d}: \bm{w}|_K \in P_0(K)^{d\times d} \  \forall K\in   \mathcal{T}_h^{*}\}, \\
\bm{V}_h :=& \{\bm{v}\in L^2(\Omega)^d: \bm{v}|_K \in  P_1(K)^d \ \forall   K\in   \mathcal{T}_h^{*}\}, \\
Q_h :=& \{q\in L_0^2(\Omega): q|_K \in  P_0(K) \ \forall K\in   \mathcal{T}_h^{*}\}, \\
\bm{M}_h  :=& \{\bm{\mu}\in L^2(\varepsilon_h^i)^d: \ \bm{\mu}|_F \in P_0(F)^d \ 
\forall   F\in \varepsilon_h^i \},\\
\tilde{\bm{M}}_h: =& \{\tilde{\bm{\mu}}\in L^2(\varepsilon_h^{\Gamma})^d : \  \tilde{\bm{\mu}}|_F \in P_m(F)^d,\  \forall F\in \varepsilon_h^{\Gamma} \} \text{ with $m = 0,1$.}
\end{align*}
Then the X-HDG scheme for \eqref{pb3}  is given as follows: find $(\bm{L}_h,\bm{u}_h,p_h,\bm{\hat{u}}_h,\bm{\tilde{u}}_h)\in \bm{W}_h\times \bm{V_h}\times Q_h\times \bm{M}_h\times \tilde{\bm{M}}_h$ such that
\begin{subequations}\label{xhdgscheme_1}
	\begin{align}
	(\nu^{-1}\bm{L}_h,\bm{w})_{\mathcal{T}_h^*} 
	- \langle \hat{\bm{u}}_h,\bm{w}\bm{n}\rangle_{\partial\mathcal{T}_h^*\setminus \varepsilon_h^{\Gamma}} -\langle \tilde{\bm{u}}_h,\bm{w}\bm{n}\rangle_{\mathcal{\varepsilon}_h^{\Gamma}} =& 0, \label{xhdg1_1}\\
	(\alpha\bm{u}_h,\bm{v})_{\mathcal{T}_h^*}+ \langle \tau(\bm{Q}_0^b\bm{u}_h-\bm{\hat{u}}_h),\bm{v}\rangle_{\partial\mathcal{T}_h^*\setminus \varepsilon_h^{\Gamma}} +  \langle \tau(\bm{Q}_m^b\bm{u}_h-\bm{\tilde{u}}_h),\bm{v}\rangle_{\mathcal{\varepsilon}_h^{\Gamma}}  =& (\bm{f},\bm{v}) ,\label{xhdg2_1}\\
\langle\bm{\hat{u}}_h\cdot\bm{n},q\rangle_{\partial\mathcal{T}_h^*\setminus \varepsilon_h^{\Gamma}} + \langle\bm{\tilde{u}}_h\cdot\bm{n},q\rangle_{\mathcal{\varepsilon}_h^{\Gamma}}=& 0,\label{xhdg3_1}\\
	\langle \bm{L}_h \bm{n},\bm{\mu}\rangle_{\partial \mathcal{T}_h^*\setminus \varepsilon_h^{\Gamma}}-\langle p_h\bm{n},\bm{\mu}\rangle_{\partial\mathcal{T}_h^*\setminus \varepsilon_h^{\Gamma}}-\langle \tau(\bm{Q}_0^b\bm{u}_h-\bm{\hat{u}}_h),\bm{\mu}\rangle_{\partial \mathcal{T}_h^*\setminus\varepsilon_h^{\Gamma}} =& 0,\label{xhdg4_1} \\
\langle \bm{\hat{u}}_h, \bm{\mu^*}\rangle_{\mathcal{\varepsilon}_h^{\Gamma}}=&\langle \bm{g}_D^h, \bm{\mu^*}\rangle_{\mathcal{\varepsilon}_h^{\Gamma}} \label{xhdg5_1},
	\end{align}
\end{subequations}
for all $(\bm{w},\bm{v},q,\bm{\mu},\bm{\tilde{\mu}})\in \bm{W}_h\times \bm{V_h}\times Q_h\times \bm{M}_h\times \tilde{\bm{M}}_h$. Here
$$\langle w,v\rangle_{\partial\mathcal{T}_h^*\setminus \varepsilon_h^{\Gamma}} := \sum\limits_{K\in  \mathcal{T}_h^*}\langle w,v\rangle_{\partial K\setminus \varepsilon_h^{\Gamma}}, \quad  \langle w,v\rangle_{ \varepsilon_h^{\Gamma}} := \sum\limits_{F\in  \varepsilon_h^{\Gamma}}\langle w,v\rangle_{F}$$
 for any scalars/vectors $w$ and $ v$,
and the stabilization coefficient is given by
\begin{align}
\tau|_F = \nu h_K^{-1}, \forall F\subset \partial K\  {\rm with } \ K\in \mathcal{T}_h^*.
\end{align}
When   $F\in \mathcal{\varepsilon}_h^{\Gamma}$ is a line segment/straight plane, we take $\bm{g}_D^h|_{F} := \bm{g}_D$, and when  $F=\Gamma_{K,h}\neq \Gamma_K$ for some $K\in \mathcal{T}_h^{\Gamma}$, we set $\bm{g}_D^h|_{F}$ to be some  linear interpolation of   $\bm{g}_D$  using data of $\bm{g}_D$ at two (2D case)/three (3D case) intersection points of  $\Gamma_K$ and $\Gamma_{K,h}$. 

By following the same line as in the proof of Theorem \ref{wellposed}, we can obtain the existence and uniqueness of the solution to \eqref{xhdgscheme_1}.
\begin{thm}\label{wellposed1}
	The X-HDG scheme \eqref{xhdgscheme_1}  admits a unique solution.
\end{thm}

\section{Numerical experiments}

In this section, we provide five 2-dimensional numerical examples 
to verify the performance of the proposed X-HDG method. 
\begin{exmp}Square domain with  circular interface  \cite{Adjerid2015An}: $\bm{g}_N^{\Gamma} = 0$\label{Circlejumpzero}
\end{exmp}

Consider the problem \eqref{pb1} 
 with $\Omega = [-1,1]^2$, $\Omega_1 = \{(x,y)\in\Omega: r = \sqrt{x^2+y^2}>r_0 = \sqrt{3/10}\}$,  $\Omega_2 = \Omega\setminus\bar{\Omega}_1$, and
$\alpha_1 = \alpha_2 = 0$.  The exact solution $(\bm{u} = (u_1,u_2), p)$ is given 
by
\begin{eqnarray*}
u_1(x,y) = &
\frac{1}{\nu_i}y(x^2+y^2-0.3) & \text{ in}\ \Omega_i, \quad i=1,2,\\
u_2(x,y) =& 
-\frac{1}{\nu_i}x(x^2+y^2-0.3)  & \text{ in}\ \Omega_i, \quad i=1,2,\\
p(x,y) = &\frac{1}{10}(x^3-y^3), & \text{ in}\ \Omega_1\cup\Omega_2. 
\end{eqnarray*}
The force term, boundary conditions and interface conditions can be derived explicitly.

We take $(\nu_1,\nu_2) = (10^{-3},1), (1,10^{-3}), (10^{-3},10^{-3})$, and use   $N\times N$ uniform triangular/rectangular meshes  (cf.  Figure \ref{circledomain}) in the X-HDG scheme \eqref{xhdgscheme}. Error results of the  numerical solutions   are listed in Table \ref{circlek_1}, and the  solutions $\bm{u}_h$ and $p_h$ at $(\nu_1, \nu_2) =(10^{-3},1)$ and $128\times 128$ triangular 
mesh are shown in Figure \ref{circle0001_1fig}. 

From Table \ref{circlek_1} we can see that in all cases  the X-HDG method yields  optimal convergence orders,  i.e.  first order    for $\lVert L-L_h \rVert_0$, $\lVert \nabla u-\nabla_hu_h \rVert_0$ and $\lVert p-p_h \rVert_0$, and second  order  for  $\lVert u-u_{h}\rVert_0$.
     These results are  conformable to Theorems \ref{1norm} and  \ref{velocity0norm}.

\begin{figure}[htbp]
\centering
\includegraphics[height = 4.5 cm,width=5 cm]{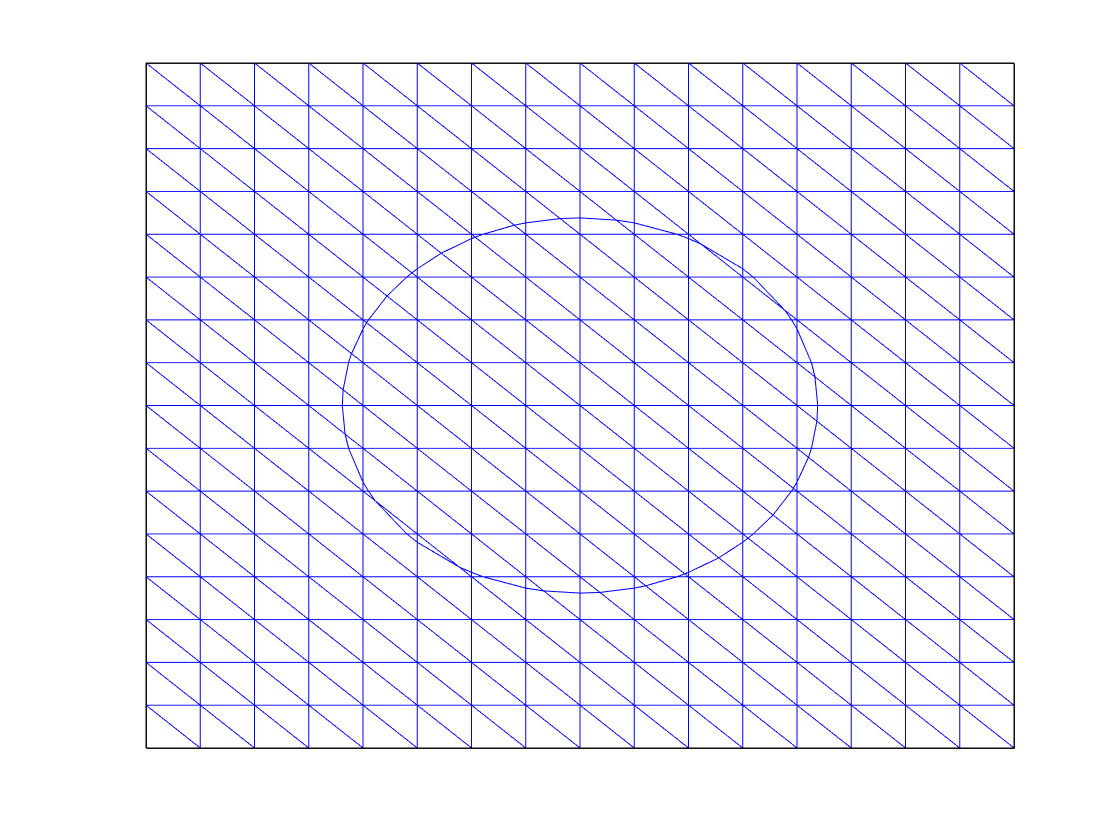} 
\includegraphics[height = 4.5 cm,width=5 cm]{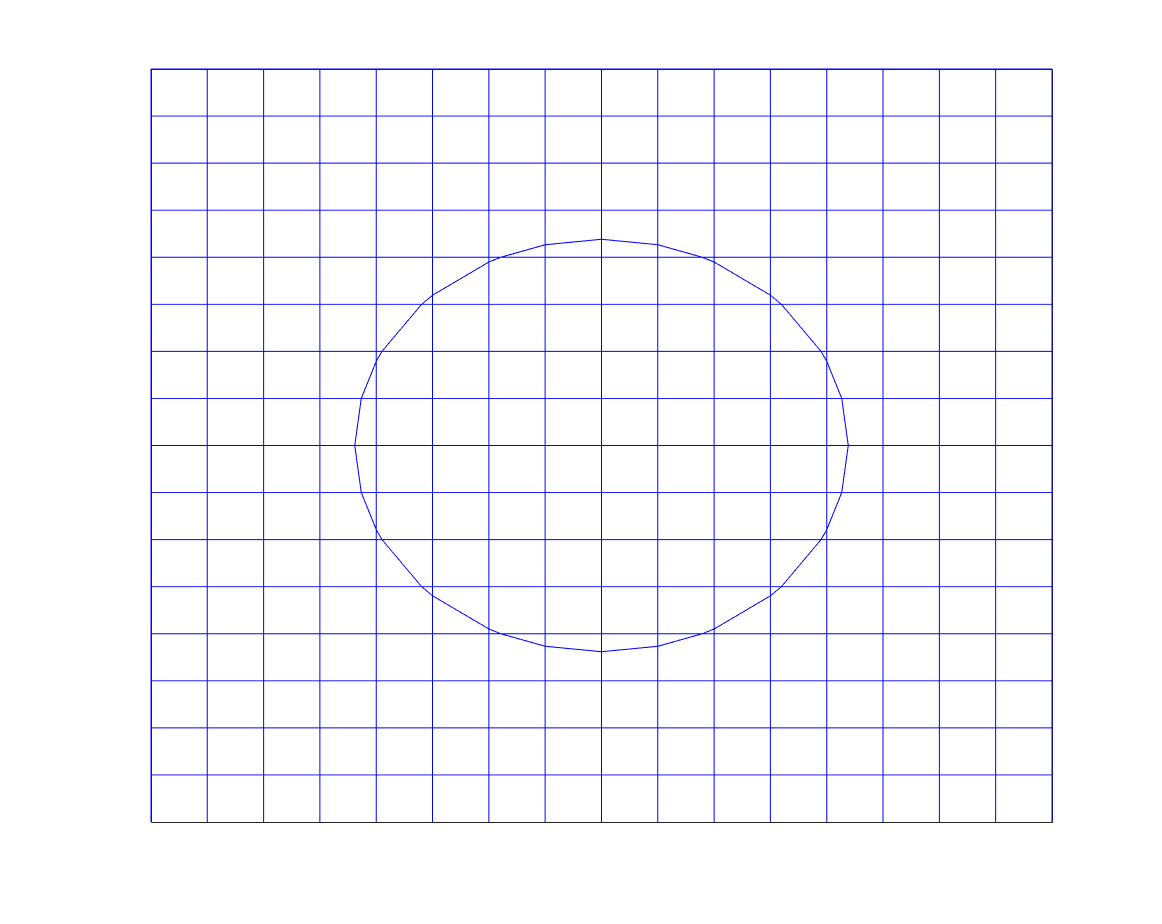} 
\caption{The square domain with circular interface at $16\times16$ meshes: triangular mesh(left) and rectangular mesh(right). }
\label{circledomain}
\end{figure}

\begin{figure}[ht]
	\centering
	\begin{minipage}[t]{0.3\textwidth}
		\includegraphics[height = 5 cm,width=5  cm]{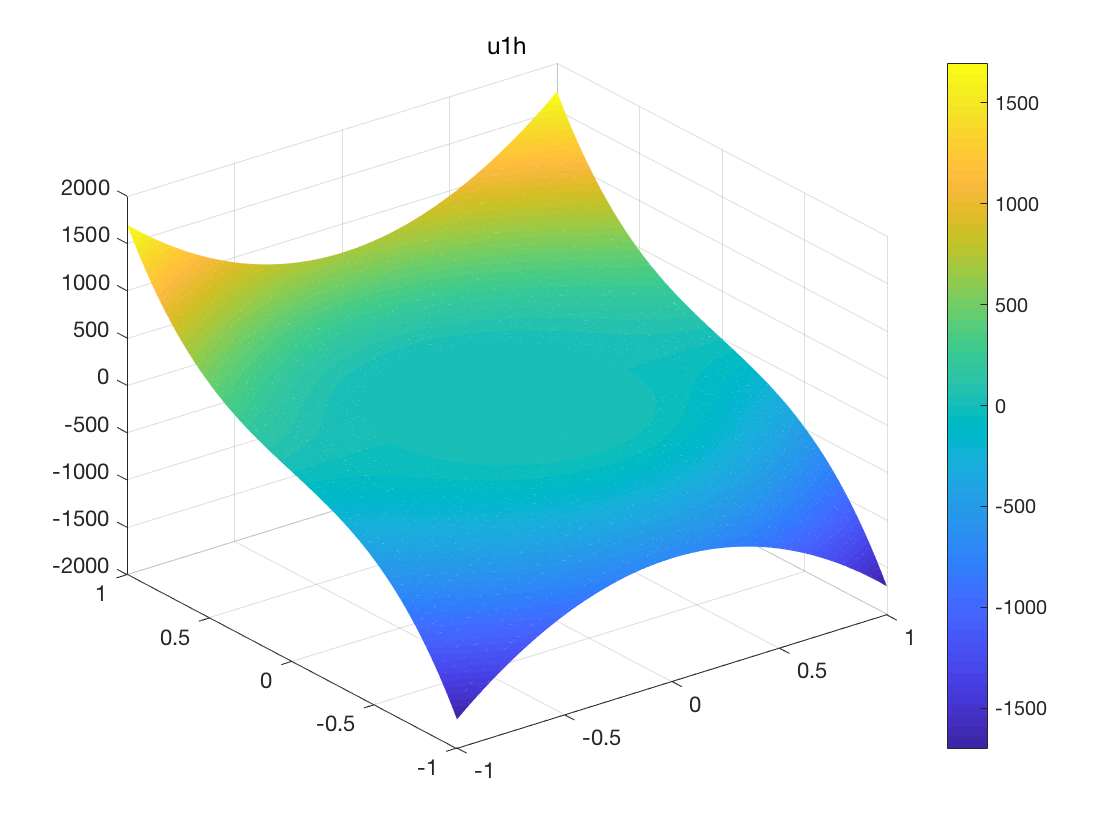} 
	\end{minipage}
	\begin{minipage}[t]{0.3\textwidth}
		\includegraphics[height = 5 cm,width=5  cm]{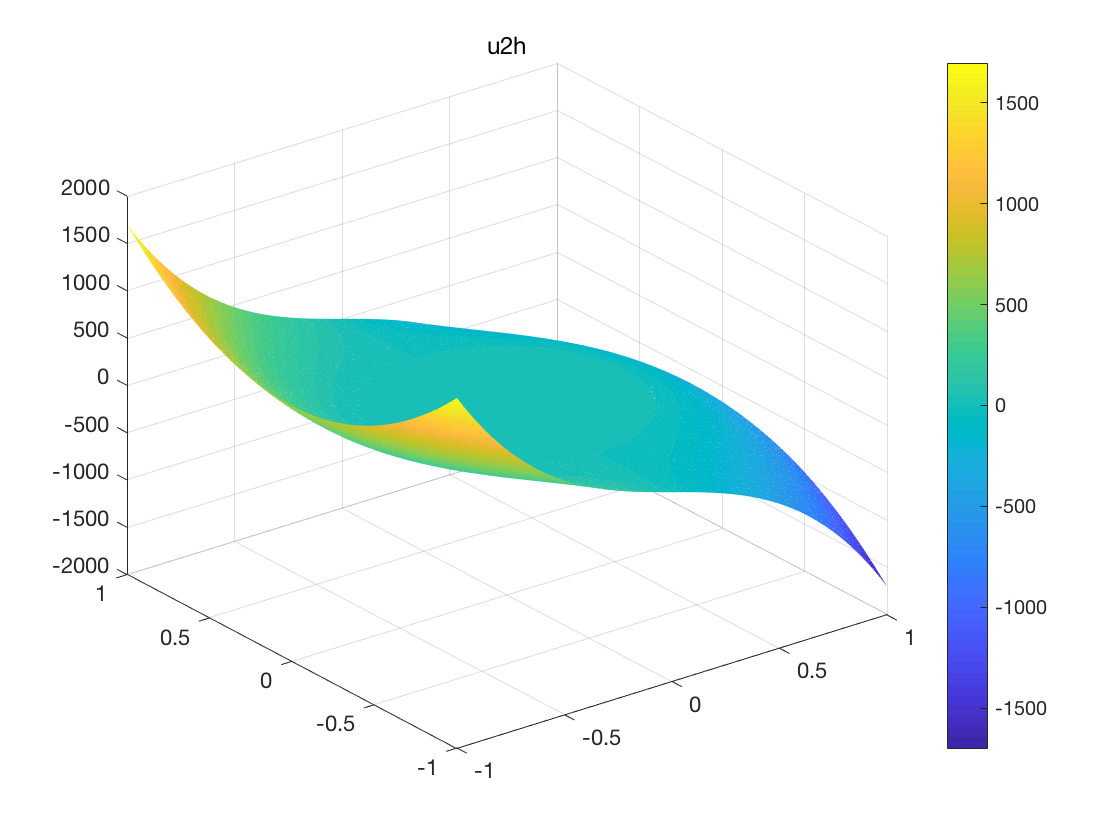}
	\end{minipage}
	\begin{minipage}[t]{0.3\textwidth}
		\includegraphics[height = 5 cm,width=5 cm]{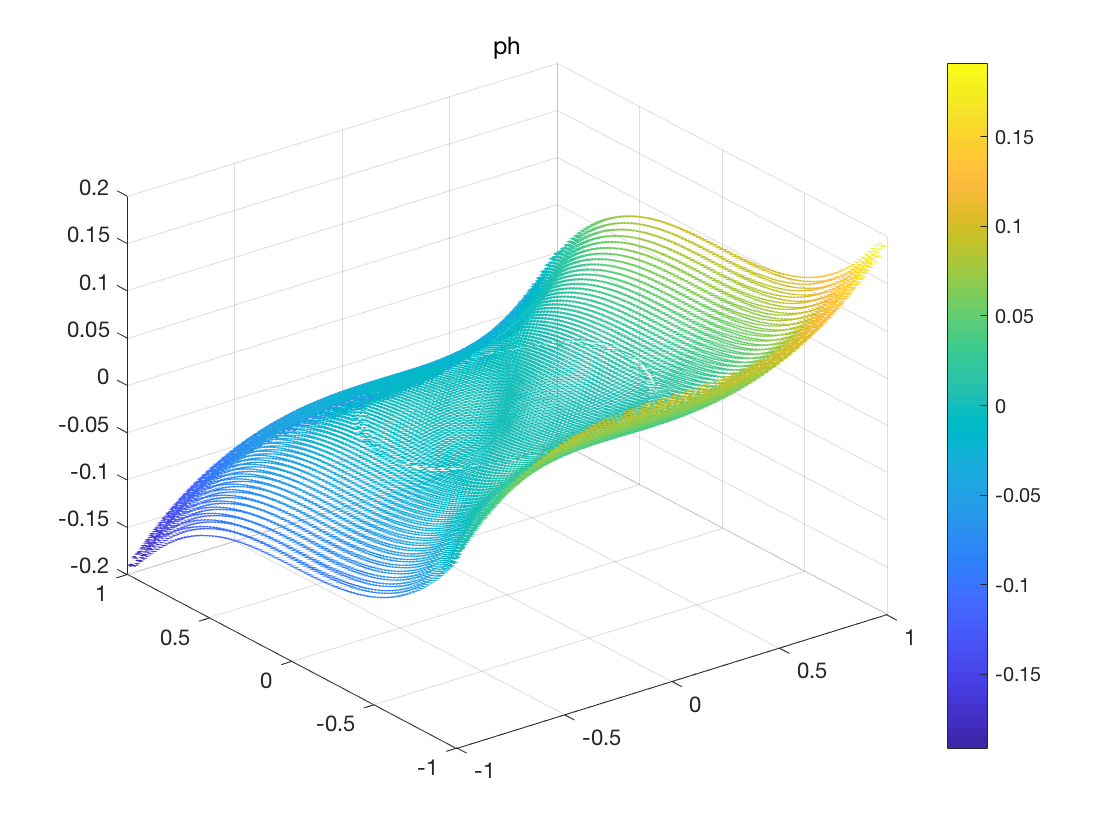}
	\end{minipage}
	\tiny\caption{The X-HDG solutions $u_{1h}$(left), $u_{2h}$(middle) and $p_h$(right) at $
		(\nu_1,\nu_2)=(10^{-3},1)$ and $128\times 128$ triangular mesh: Example \ref{Circlejumpzero}. }\label{circle0001_1fig}
\end{figure}

\begin{table}[H]
\normalsize
\caption{History of convergence for the X-HDG scheme \eqref{xhdgscheme}: Example \ref{Circlejumpzero}}\label{circlek_1}
\centering
\footnotesize
\subtable[ Triangular meshes]{
\begin{tabular}{p{0.4cm}<{\centering}|p{0.4cm}<{\centering}|p{1.4cm}<{\centering}|p{1.45cm}<{\centering}|
p{0.45cm}<{\centering}|p{1.45cm}<{\centering}|p{0.45cm}<{\centering}|p{1.45cm}<{\centering}|p{0.45cm}<{\centering}|
p{1.45cm}<{\centering}|p{0.45cm}<{\centering}}
\hline 
\multirow{1}{*}{$\nu_1$}&\multirow{1}{*}{$\nu_2$}& \multirow{2}{*}{mesh}& 
\multicolumn{2}{c|}{$\frac{\lVert u-u_{h}\rVert_0}{\lVert u\rVert_0}$ }&\multicolumn{2}{c|}{$\frac{\lVert L-L_h \rVert_0}{\lVert L\rVert_0}$}&\multicolumn{2}{c|}{$\frac{\lVert \nabla u-\nabla_hu_h \rVert_0}{\lVert \nabla 
u\rVert_0}$}&\multicolumn{2}{c}{$\frac{\lVert p-p_h \rVert_0}{\lVert p\rVert_0}$}\cr\cline{4-11} 
&&&error&order&error&order&error&order&error&order\cr 
\cline{1-11}
\multirow{4}{*}{$1$}&\multirow{4}{*}{$10^{-3}$}
&$16\times 16$ &1.4633E-01 &-- &8.8492E-02
 &-- &3.0103E-01 &-- &3.3596E-01 &-- \\
&&$32\times 32$ &3.7945E-02 &1.95 &4.4647E-02 &0.99 &1.4818E-01 
&1.02 &1.7004E-01 &0.98 \\
&&$64\times 64$ &9.4540E-03 &2.00 &2.2413E-02
 &0.99 &7.4780E-02 &0.99 &8.4837E-02 &1.00 \\
&&$128\times 128$&2.3843E-03 &1.99 &1.1227E-02
 &1.00 &3.7645E-02 &0.99 &4.2430E-02 &1.00 \\
\hline
\multirow{4}{*}{$10^{-3}$}&\multirow{4}{*}{$1$}
&$16\times 16$ &2.5011E-02 &-- &8.8606E-02
 &-- &9.8525E-02 &-- &3.4740E-01 &-- \\
&&$32\times 32$ &6.1643E-03 &2.02 &4.4654E-02 &0.99 &4.9169E-02 
&1.03 &1.7214E-01 &1.01 \\
&&$64\times 64$ &1.5228E-03 &2.02 &2.2415E-02
 &0.99 &2.4599E-02 &1.00 &8.5033E-02 &1.02 \\
&&$128\times 128$ &3.7796E-04 &2.01
&1.1227E-02 &1.00 &1.2301E-02 &1.00 &4.2462E-02 &1.00 \\
\hline
\multirow{4}{*}{$10^{-3}$}&\multirow{4}{*}{$10^{-3}$}
&$16\times 16$ &2.5701E-02 &-- &8.8539E-02
 &-- &1.0056E-01 &-- &3.3728E-01 &-- \\
&&$32\times 32$ &6.3464E-03 &2.02 &4.4647E-02 &0.99 &5.0251E-02 
&1.00 &1.7016E-01 &0.99 \\
&&$64\times 64$ &1.5722E-03 &2.01 &2.2414E-02
 &0.99 &2.5169E-02 &1.00 &8.4820E-02 &1.00 \\
&&$128\times 128$ &3.9075E-04 &2.01
&1.1227E-02 &1.00 &1.2595E-02 &1.00 &4.2430E-02 &1.00 \\
\hline
\end{tabular}
}
\subtable[ Rectangular meshes]{
\begin{tabular}{p{0.4cm}<{\centering}|p{0.4cm}<{\centering}|p{1.4cm}<{\centering}|p{1.45cm}<{\centering}|
p{0.45cm}<{\centering}|p{1.45cm}<{\centering}|p{0.45cm}<{\centering}|p{1.45cm}<{\centering}|p{0.45cm}<{\centering}|
p{1.45cm}<{\centering}|p{0.45cm}<{\centering}}
\hline 
\multirow{1}{*}{$\nu_1$}&\multirow{1}{*}{$\nu_2$}& \multirow{2}{*}{mesh}& 
\multicolumn{2}{c|}{$\frac{\lVert u-u_{h}\rVert_0}{\lVert u\rVert_0}$ }&\multicolumn{2}{c|}{$\frac{\lVert L-L_h \rVert_0}{\lVert L\rVert_0}$}&\multicolumn{2}{c|}{$\frac{\lVert \nabla u-\nabla_hu_h \rVert_0}{\lVert \nabla 
u\rVert_0}$}&\multicolumn{2}{c}{$\frac{\lVert p-p_h \rVert_0}{\lVert p\rVert_0}$}\cr\cline{4-11} 
&&&error&order&error&order&error&order&error&order\cr 
\cline{1-11}
\multirow{4}{*}{$1$}&\multirow{4}{*}{$10^{-3}$}
&$16\times 16$ &2.2520E-01 &-- &9.4123E-02
 &-- &2.8633E-01 &-- &2.6266E-01 &-- \\
&&$32\times 32$ &6.0494E-02 &1.90 &4.6733E-02 &1.01 &1.3749E-01 
&1.06 &1.1403E-01 &1.20 \\
&&$64\times 64$ &1.5630E-02 &1.95 &2.3310E-02
 &1.00 &6.9627E-02 &0.98 &4.5547E-02 &1.32 \\
&&$128\times 128$&3.9901E-03 &1.97 &1.1647E-02
 &1.00 &3.4854E-02 &1.00 &1.9095E-02 &1.25 \\
\hline
\multirow{4}{*}{$10^{-3}$}&\multirow{4}{*}{$1$}
&$16\times 16$ &4.2156E-02 &-- &9.4080E-02
 &-- &9.2973E-02 &-- &2.3793E-01 &-- \\
&&$32\times 32$ &1.0484E-02 &2.01 &4.6745E-02 &1.01 &4.5792E-02 
&1.02 &1.1481E-01 &1.05 \\
&&$64\times 64$ &2.6118E-03 &2.01 &2.3310E-02
 &1.00 &2.2778E-02 &1.01 &4.4972E-02 &1.35 \\
&&$128\times 128$ &6.5275E-04 &2.00
&1.1647E-02 &1.00 &1.1372E-02 &1.00 &1.8905E-02 &1.25 \\
\hline
\multirow{4}{*}{$10^{-3}$}&\multirow{4}{*}{$10^{-3}$}
&$16\times 16$ &4.3569E-02 &-- &9.4070E-02
 &-- &9.5143E-02 &-- &2.5155E-01 &-- \\
&&$32\times 32$ &1.0876E-02 &2.00 &4.6728E-02 &1.01 &4.6851E-02 
&1.02 &1.1247E-01 &1.16 \\
&&$64\times 64$ &2.7165E-03 &2.00 &2.3307E-02
 &1.00 &2.3329E-02 &1.01 &4.4363E-02 &1.34 \\
&&$128\times 128$ &6.7919E-04 &2.00
&1.1646E-02 &1.00 &1.1649E-02 &1.00 &1.8786E-02 &1.24 \\
\hline
\end{tabular}
}
\end{table}

\begin{exmp}Square domain with  circular interface  \cite{Adjerid2015An}: $\bm{g}_N^{\Gamma} \neq 0$\label{Circlejump}
\end{exmp}

Consider  the same domain   and interface as   in  Example 
\ref{Circlejumpzero}. The exact solution  $(\bm{u} = (u_1,u_2), p)$ is given by
\begin{align*}
u_1(x,y) &=
1+\frac{1}{\nu_i}y{\rm sin}(x^2+y^2-0.3) \ \text{ in}\ \Omega_i, \quad i=1,2,\\
u_2(x,y) &=
2-\frac{1}{\nu_i}x{\rm sin}(x^2+y^2-0.3)  \ \text{ in}\ \Omega_i,  \quad i=1,2,\\
p(x,y) &= \left \{
\begin{array}{rl}
e^{x+y}-1.3798535909816816 \  \ \text{ in}\ \Omega_1,\\
\sqrt{1+x^2+y^2}-1.3798535909816816 \  \ \text{ in}\ \Omega_2,
\end{array}
\right. 
\end{align*}
and the coefficients $\nu$ and $\alpha$ are taken  as: (i) $(\nu_1,\nu_2) = (1,10^{-3}), (\alpha_1, \alpha_2) = (0,0)$;  (ii) $
(\nu_1,\nu_2) = (10^{-2},1), (\alpha_1, \alpha_2) = (1,0)$. 

From  Table \ref{circletablek_1}  we can observe 
that the convergence rates of $
\lVert \bm{u}-\bm{u}_h\rVert_0, \lVert \bm{L}-\bm{L}_h\rVert_0, \lVert \nabla\bm{u}-\nabla_h\bm{u}_h\rVert_0$ and $\lVert p-p_h\rVert_0$ are all optimal at triangular  and rectangular meshes.
We also  plot in Figure \ref{circlefigk_1_0001} the the numerical solutions $\bm{u}_h$ and $ p_h$ at $128\times 128$ triangular mesh 
with $(\nu_1,\nu_2) = (1,10^{-3}),  (\alpha_1, \alpha_2) = (0,0)$.
\begin{table}[H]
\normalsize
\caption{History of convergence for the X-HDG scheme \eqref{xhdgscheme}: Example \ref{Circlejump} }\label{circletablek_1}
\centering
\footnotesize
\subtable[ Triangular meshes]{
\begin{tabular}{p{0.4cm}<{\centering}|p{0.4cm}<{\centering}|p{0.4cm}<{\centering}|p{0.4cm}<{\centering}|
p{1.4cm}<{\centering}|p{1.45cm}<{\centering}|p{0.45cm}<{\centering}|p{1.45cm}<{\centering}|p{0.45cm}<{\centering}|
p{1.45cm}<{\centering}|p{0.45cm}<{\centering}|p{1.45cm}<{\centering}|p{0.45cm}<{\centering}}
\hline 
\multirow{1}{*}{$\alpha_1$}&\multirow{1}{*}{$\alpha_2$}&\multirow{1}{*}{$\nu_1$}&\multirow{1}{*}{$
\nu_2$}& \multirow{2}{*}{mesh}& 
\multicolumn{2}{c|}{$\frac{\lVert u-u_{h}\rVert_0}{\lVert u\rVert_0}$ }&\multicolumn{2}{c|}{$\frac{\lVert L-L_h 
\rVert_0}{\lVert L\rVert_0}$}&\multicolumn{2}{c|}{$\frac{\lVert \nabla u-\nabla_hu_h \rVert_0}{\lVert \nabla u\rVert_0}$}
&\multicolumn{2}{c}{$\frac{\lVert p-p_h \rVert_0}{\lVert p\rVert_0}$}\cr\cline{6-13} 
&&&&&error&order&error&order&error&order&error&order\cr 
\cline{1-13}
\multirow{4}{*}{$0$}&\multirow{4}{*}{$0$}&\multirow{4}{*}{$1$}&\multirow{4}{*}{$10^{-3}$}
&$16\times 16$ &1.4780E-01 &-- &9.1176E-02
 &-- &3.0346E-01 &-- &4.9323E-02 &-- \\
&&&&$32\times 32$ &3.8392E-02 &1.94
&4.6059E-02 &0.99 &1.4955E-01 &1.02 &2.3874E-02 &1.05 \\
&&&&$64\times 64$ &9.5775E-03 &2.00
&2.3134E-02 &0.99 &7.5525E-02 &0.99 &1.1738E-02 &1.02 \\
&&&&$128\times 128$ &2.4162E-03 &1.99
&1.1590E-02 &1.00 &3.8060E-02 &0.99 &5.8293E-03 &1.01 \\
\hline
\multirow{4}{*}{$1$}&\multirow{4}{*}{$0$}&\multirow{4}{*}{$10^{-2}$}&\multirow{4}{*}{$1$}
&$16\times 16$ &1.4245E-02 &-- &9.5200E-02 &-- &9.3735E-02 
 & -- &1.8235E-01 &--\\
&&&&$32\times 32$ &4.3225E-03 &1.72
&4.6587E-02 &1.03 &5.0761E-02 &0.88 &5.9483E-02 &1.62 \\
&&&&$64\times 64$ &1.1439E-03 &1.92
&2.3201E-02 &1.01 &2.6161E-02 &0.96 &1.8745E-02 &1.67 \\
&&&&$32\times 32$ &2.8998E-04 &1.98
&1.1598E-02 &1.00 &1.3196E-02 &0.99 &6.9188E-03 &1.44 \\
\hline
\end{tabular}
}
\subtable[ Rectangular meshes]{
\begin{tabular}{p{0.4cm}<{\centering}|p{0.4cm}<{\centering}|p{0.4cm}<{\centering}|p{0.4cm}<{\centering}|p{1.4cm}
<{\centering}|p{1.45cm}<{\centering}|p{0.45cm}<{\centering}|p{1.45cm}<{\centering}|p{0.45cm}<{\centering}|p{1.45cm}
<{\centering}|p{0.45cm}<{\centering}|p{1.45cm}<{\centering}|p{0.45cm}<{\centering}}
\hline 
\multirow{1}{*}{$\alpha_1$}&\multirow{1}{*}{$\alpha_2$}&\multirow{1}{*}{$\nu_1$}&\multirow{1}{*}{$\nu_2$}& 
\multirow{2}{*}{mesh}& 
\multicolumn{2}{c|}{$\frac{\lVert u-u_{h}\rVert_0}{\lVert u\rVert_0}$ }&\multicolumn{2}{c|}{$\frac{\lVert L-L_h 
\rVert_0}{\lVert L\rVert_0}$}&\multicolumn{2}{c|}{$\frac{\lVert \nabla u-\nabla_hu_h \rVert_0}{\lVert \nabla u\rVert_0}$}
&\multicolumn{2}{c}{$\frac{\lVert p-p_h \rVert_0}{\lVert p\rVert_0}$}\cr\cline{6-13} 
&&&&&error&order&error&order&error&order&error&order\cr 
\cline{1-13}
\multirow{4}{*}{$0$}&\multirow{4}{*}{$0$}&\multirow{4}{*}{$1$}&\multirow{4}{*}{$10^{-3}$}
&$16\times 16$ &2.2885E-01 &-- &9.6862E-02
 &-- &2.8766E-01 &-- &1.0153E-01 &-- \\
&&&&$32\times 32$ &6.1575E-02 &1.89 &4.7445E-02 &1.03 &1.3779E-01 &1.06 &4.6034E-02 
&1.14 \\
&&&&$64\times 64$ &1.5929E-02 &1.95 &2.3464E-02
 &1.02 &6.9707E-02 &0.98 &2.0969E-02 &1.13 \\
&&&&$128\times 128$ &4.0677E-03 &1.97 &1.1677E-02
 &1.01 &3.4873E-02 &1.00 &9.9306E-03 &1.08 \\
\hline
\multirow{4}{*}{$1$}&\multirow{4}{*}{$0$}&\multirow{4}{*}{$10^{-2}$}&\multirow{4}{*}{$1$}
&$16\times 16$ &2.2764E-02 &-- &1.1420E-01 &-- &9.7250E-02 
 & -- &3.6125E-01 &--\\
&&&&$32\times 32$ &8.0154E-03 &1.51 &5.1637E-02
 &1.15 &4.7396E-02 &1.04 &1.4858E-01 &1.28 \\
&&&&$64\times 64$ &2.2922E-03 &1.81 &2.4185E-02
 &1.09 &2.3104E-02 &1.04 &4.7582E-02 &1.64 \\
&&&&$128\times 128$ &5.9780E-04 &1.94 &1.1778E-02
 &1.04 &1.1426E-02 &1.02 &1.5112E-03 &1.65 \\
\hline
\end{tabular}
}
\end{table}
\begin{figure}[ht]
\centering
\begin{minipage}[t]{0.3\textwidth}
\includegraphics[height = 4.7 cm,width=5  cm]{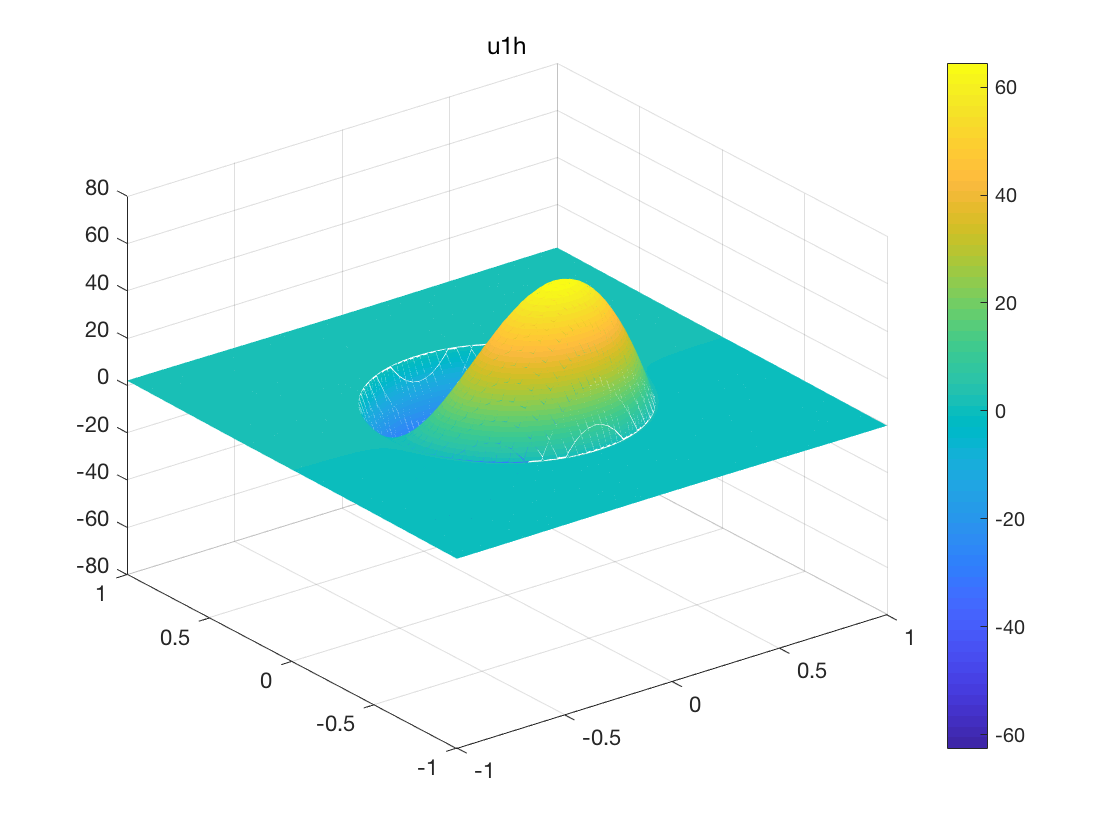} 
\end{minipage}
\begin{minipage}[t]{0.3\textwidth}
\includegraphics[height = 4.7 cm,width=5  cm]{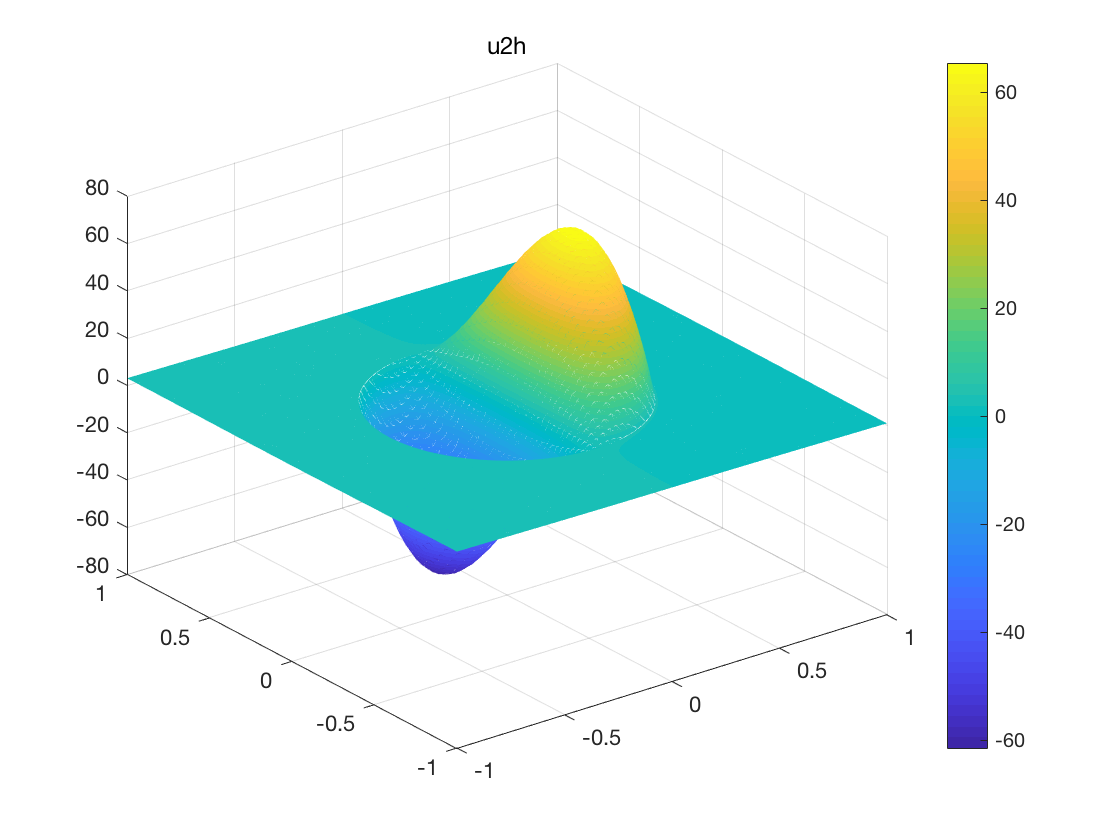}
\end{minipage}
\begin{minipage}[t]{0.3\textwidth}
\includegraphics[height = 4.7 cm,width=5  cm]{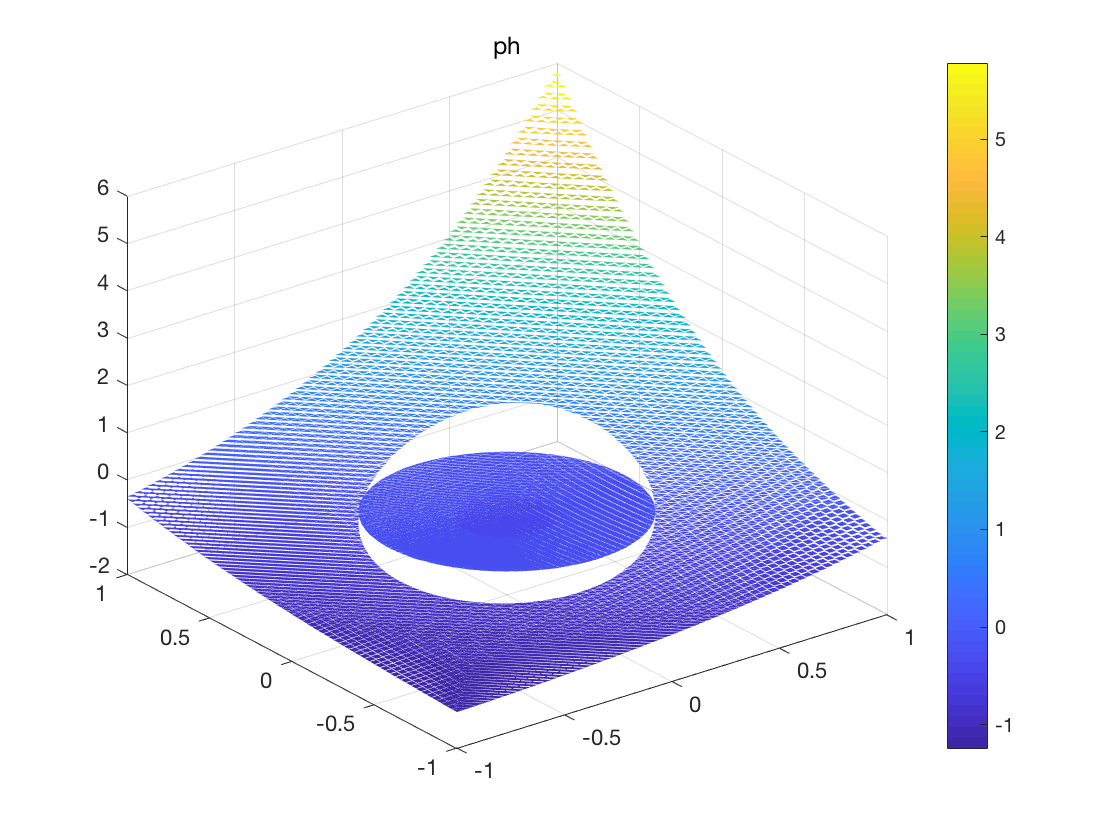}
\end{minipage} \tiny\caption{The X-HDG solutions $u_{1h}$(left), $u_{2h}$(middle) and $p_h$(right) at $
(\nu_1,\nu_2) = (1,10^{-3}), \alpha_1 = \alpha_2 = 0$ and $128\times 128$ triangular mesh: Example \ref{Circlejump}. }
\label{circlefigk_1_0001}
\end{figure}

\begin{exmp}A  laminar flow test in a square domain with straight line interface: $\bm{g}_N^{\Gamma}\neq0 $
\label{exampleline}
\end{exmp}

Take $\Omega = [0,1]^2$ (Figure
\ref{linedomain}) in \eqref{pb1}. Two kinds of fluids  with different viscosity   flow in the subdomains $\Omega_{1}=[0,1]\times 
[b_0,1]$ and $\Omega_{2}=[0,1]\times [0,b_0]$, respectively, with $b_0 = 0.4031$. The exact solution is given by
\begin{eqnarray*}
u_1(x,y) = &1-e^{\lambda_i}{\rm sin}(\frac{\pi}{b_0} y)&  {\rm in}\ \Omega_i, \quad i=1,2, \\
u_2(x,y) = &0  & {\rm in}\ \Omega_1\cup \Omega_2,\\
p(x,y) = &\frac{1}{2}e^{2\lambda_ix}-b_0(\frac{1}{4\lambda_2}e^{2\lambda_2}-1) -(1-b_0)(\frac{1}{4\lambda_1}
e^{2\lambda_1}-1) & {\rm in}\ \Omega_i, \quad i=1,2 ,
\end{eqnarray*}
where  
\begin{align*}
\lambda_i = \frac{1}{2\nu_i} - \sqrt{\frac{1}{4\nu_i^2}+4\pi^2}, \ i = 1,2,
\end{align*}
and $(\nu_1,\nu_2) = (1,10^{-2})$. We take $(\alpha_1,\alpha_2) = (0,0),
(0,1)$.

From  Table \ref{linetablek_1}, we can see that the X-HDG method yields optimal convergence rates for the numerical solutions at both
triangular   and rectangular meshes. 
We also show  in Figure \ref{lineuh10e0_10e-2_0_0} the numerical solutions $\bm{u}_{1h}$ and $p_h$ at $(\nu_1,\nu_2) = (1,10^{-2}),  (\alpha_1, \alpha_2) = (0,0)$ and  
$128\times 128$ triangular mesh.

\begin{figure}[htbp]
\centering
\includegraphics[height = 4.7 cm,width=5 cm]{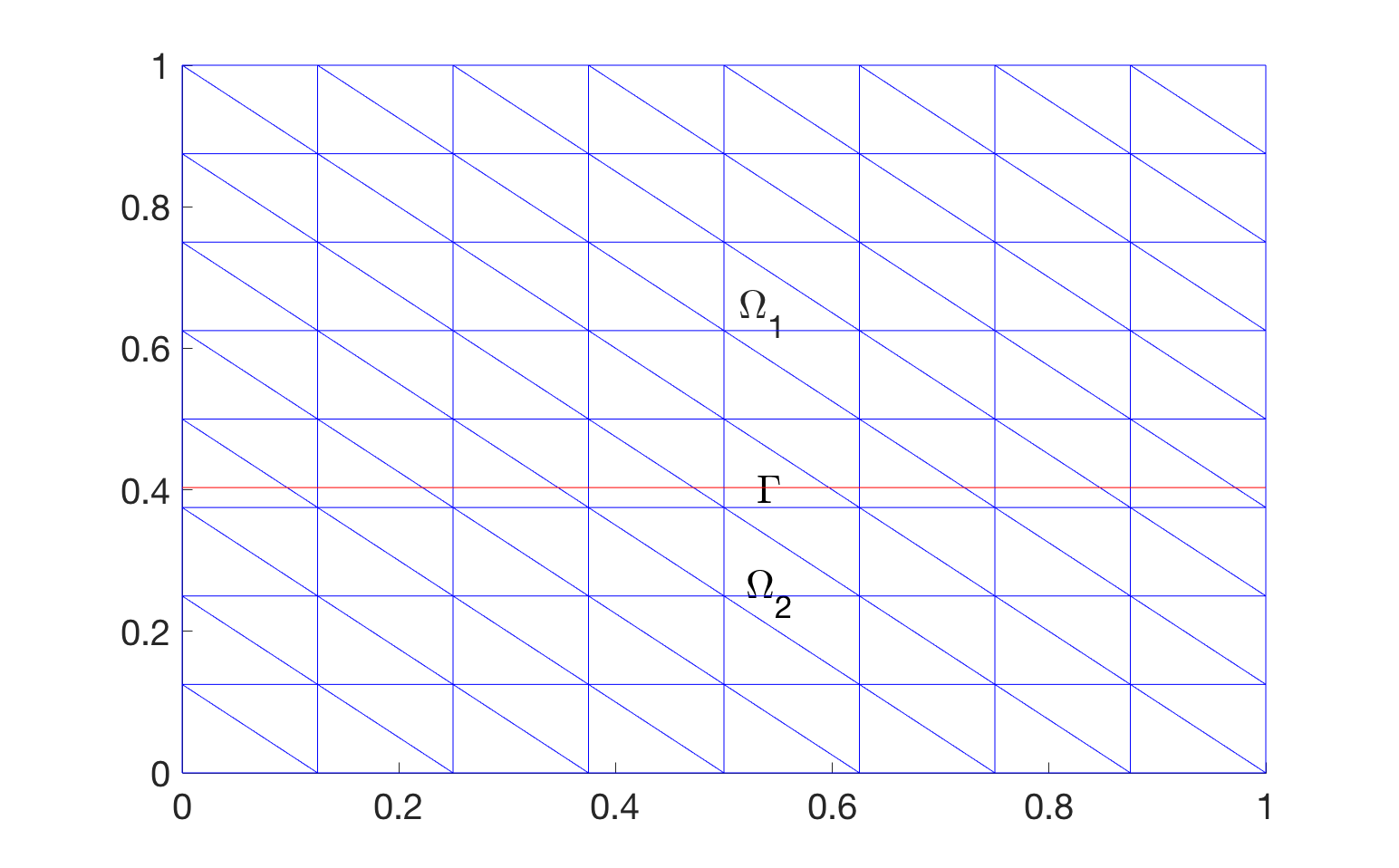} 
\includegraphics[height = 4.7 cm,width=5 cm]{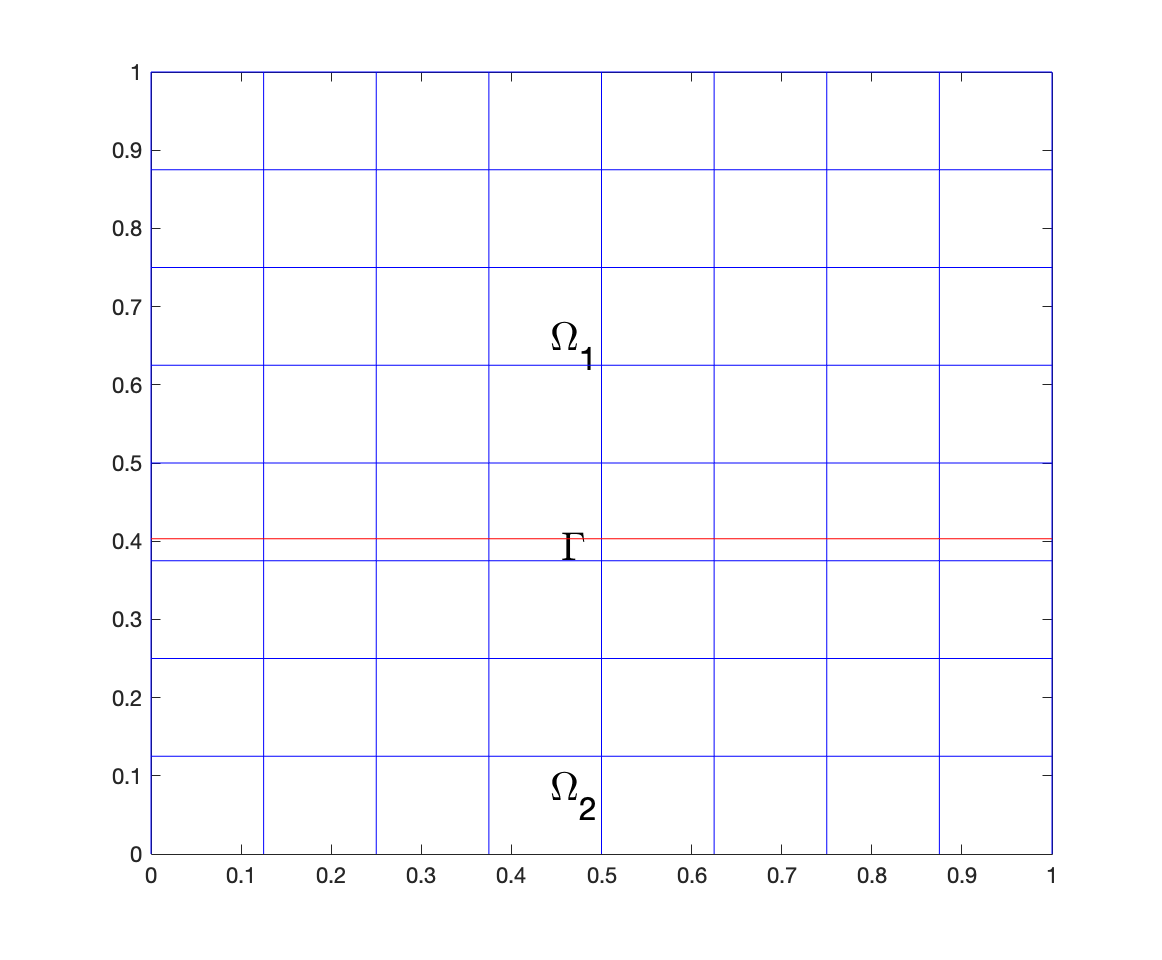} 
\caption{The domain with straight line interface at $8\times 8$ meshes: triangular mesh(left) and rectangular 
mesh(right). }\label{linedomain}
\end{figure}

\begin{table}[H]
\normalsize
\caption{History of convergence for the X-HDG scheme \eqref{xhdgscheme}: Example \ref{exampleline}}\label{linetablek_1}
\centering
\footnotesize
\subtable[ Triangular meshes]{
\begin{tabular}{p{0.4cm}<{\centering}|p{0.4cm}<{\centering}|p{0.4cm}<{\centering}|p{0.4cm}<{\centering}|
p{1.4cm}<{\centering}|p{1.45cm}<{\centering}|p{0.45cm}<{\centering}|p{1.45cm}<{\centering}|p{0.45cm}<{\centering}|
p{1.45cm}<{\centering}|p{0.45cm}<{\centering}|p{1.45cm}<{\centering}|p{0.45cm}<{\centering}}
\hline 
\multirow{1}{*}{$\alpha_1$}&\multirow{1}{*}{$\alpha_2$}&\multirow{1}{*}{$\nu_1$}&\multirow{1}{*}{$
\nu_2$}& \multirow{2}{*}{mesh}& 
\multicolumn{2}{c|}{$\frac{\lVert u-u_{h}\rVert_0}{\lVert u\rVert_0}$ }&\multicolumn{2}{c|}{$\frac{\lVert L-L_h \rVert_0}{\lVert L\rVert_0}$}&\multicolumn{2}{c|}{$\frac{\lVert \nabla u-\nabla_hu_h \rVert_0}{\lVert \nabla 
u\rVert_0}$}&\multicolumn{2}{c}{$\frac{\lVert p-p_h \rVert_0}{\lVert p\rVert_0}$}\cr\cline{6-13} 
&&&&&error&order&error&order&error&order&error&order\cr 
\cline{1-13}
\multirow{5}{*}{$0$}&\multirow{5}{*}{$0$}&\multirow{5}{*}{$1$}&\multirow{5}{*}{$10^{-2}$}
&$8\times 8$ &1.0441E-01 &-- &2.5745E-01
 &-- &3.5136E-01 &-- &1.5905E-01 &-- \\
&&&&$16\times 16$ &2.6874E-02 &1.96
&1.3388E-01 &0.94 &1.6803E-01 &1.06 &8.3824E-02 &0.92 \\
&&&&$32\times 32$ &6.8710E-03 &1.97
 &6.8331E-02 &0.97 &8.3614E-02 &1.01 &4.1910E-02 &1.00 \\
&&&&$64\times 64$ &1.7249E-03 &1.99
&3.4397E-02 &0.99 &4.1890E-02 &1.00 &2.0649E-02 &1.02 \\
&&&&$128\times 128$ &4.3194E-04 &2.00
&1.7253E-02 &1.00 &2.1019E-02 &0.99 &1.0208E-02 &1.02 \\
\hline
 \multirow{5}{*}{$0$}&\multirow{5}{*}{$1$}&\multirow{5}{*}{$1$}&\multirow{5}{*}{$10^{-2}$}
&$8\times 8$ &7.5215E-02 &-- &2.4127E-01
 &-- &2.8421E-01 &-- &1.9447E-01 &-- \\
&&&&$16\times 16$ &2.3765E-02 &1.66
&1.2792E-01 &0.92 &1.5548E-01 &0.87 &9.2509E-02 &1.07 \\
 &&&&$32\times 32$ &6.5246E-03 &1.86 &6.7253E-02 &0.93
 &8.1713E-02 &0.93 &4.3366E-02 &1.09 \\
&&&&$64\times 64$ &1.6715E-03 &1.96
&3.4247E-02 &0.97 &4.1634E-02 &0.97 &2.0852E-02 &1.06 \\
&&&&$128\times 128$ &4.2082E-04 &1.99
&1.7234E-02 &0.99 &2.0985E-02 &0.99 &1.0235E-02 &1.03 \\
\hline
\end{tabular}
}
\subtable[ Rectangular meshes]{
\begin{tabular}{p{0.4cm}<{\centering}|p{0.4cm}<{\centering}|p{0.4cm}<{\centering}|p{0.4cm}<{\centering}|p{1.4cm}
<{\centering}|p{1.45cm}<{\centering}|p{0.45cm}<{\centering}|p{1.45cm}<{\centering}|p{0.45cm}<{\centering}|p{1.45cm}
<{\centering}|p{0.45cm}<{\centering}|p{1.45cm}<{\centering}|p{0.45cm}<{\centering}}
\hline 
\multirow{1}{*}{$\alpha_1$}&\multirow{1}{*}{$\alpha_2$}&\multirow{1}{*}{$\nu_1$}&\multirow{1}{*}{$\nu_2$}& 
\multirow{2}{*}{mesh}& 
\multicolumn{2}{c|}{$\frac{\lVert u-u_{h}\rVert_0}{\lVert u\rVert_0}$ }&\multicolumn{2}{c|}{$\frac{\lVert L-L_h 
\rVert_0}{\lVert L\rVert_0}$}&\multicolumn{2}{c|}{$\frac{\lVert \nabla u-\nabla_hu_h \rVert_0}{\lVert \nabla u\rVert_0}$}
&\multicolumn{2}{c}{$\frac{\lVert p-p_h \rVert_0}{\lVert p\rVert_0}$}\cr\cline{6-13} 
&&&&&error&order&error&order&error&order&error&order\cr 
\cline{1-13}
\multirow{5}{*}{$0$}&\multirow{5}{*}{$0$}&\multirow{5}{*}{$1$}&\multirow{5}{*}{$10^{-2}$}
&$8\times 8$ &1.5667E-01 &-- &7.3078E-01
 &-- &3.4141E-01 &-- &1.8268E-01 &-- \\
&&&&$16\times 16$ &4.1512E-02 &1.92 &2.8501E-01 &1.36 &1.5353E-01 &1.15 
&9.9056E-02 &0.88 \\
&&&&$32\times 32$ &1.0668E-02 &1.96 &1.1163E-01 &1.35
 &7.2846E-02 &1.08 &5.5724E-02 &0.97 \\
&&&&$64\times 64$ &2.6985E-03 &1.98 &4.5773E-02
 &1.29 &3.5621E-02 &1.03 &2.5207E-02 &1.00 \\
&&&&$128\times 128$ &6.7743E-04 &1.99 &1.9938E-02
 &1.20 &1.7657E-02 &1.01 &1.2504E-02 &1.01 \\
\hline
\multirow{5}{*}{$0$}&\multirow{5}{*}{$1$}&\multirow{5}{*}{$1$}&\multirow{5}{*}{$10^{-2}$}
&$8\times 8$ &9.5822E-02 &-- &7.9015E-01
 &-- &2.8685E-01 &-- &2.2368E-01 &-- \\
&&&&$16\times 16$ &3.3711E-02 &1.51 &3.1917E-01 &1.31 &1.4574E-01 &0.98 &1.1284E-01 
&0.99 \\
&&&&$32\times 32$ &9.8025E-03 &1.78 &1.2139E-01 &1.39
 &7.2240E-02 &1.01 &5.3337E-02 &1.08 \\
&&&&$64\times 64$ &2.5797E-03 &1.93 &4.7640E-02
 &1.35 &3.5580E-02 &1.02 &2.5633E-02 &1.06 \\
&&&&$128\times 128$ &6.5462E-04 &1.98 &2.0227E-02
 &1.24 &1.7654E-02 &1.01 &1.2561E-02 &1.03 \\
\hline
\end{tabular}
}
\end{table}

\begin{figure}[ht]
\centering
\begin{minipage}[t]{0.4\textwidth}
\includegraphics[height = 4.7 cm,width=5 cm]{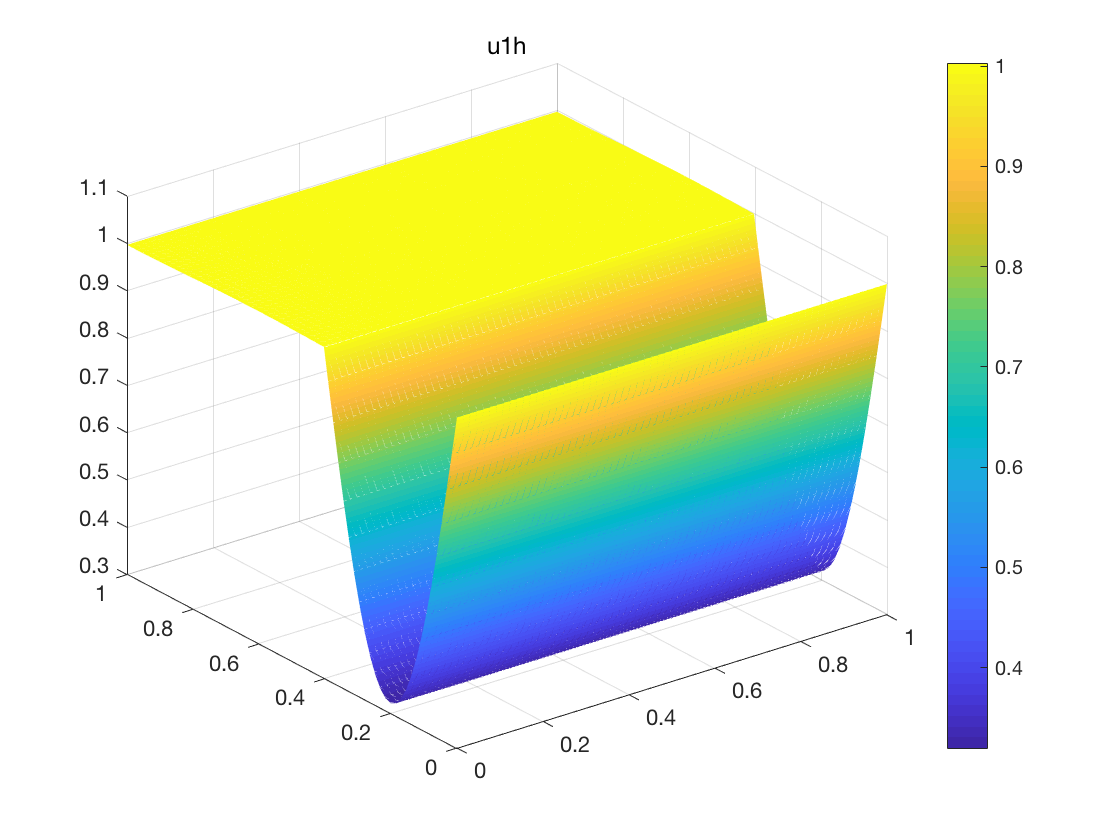} 
\end{minipage}
\begin{minipage}[t]{0.4\textwidth}
\includegraphics[height = 4.7 cm,width=5 cm]{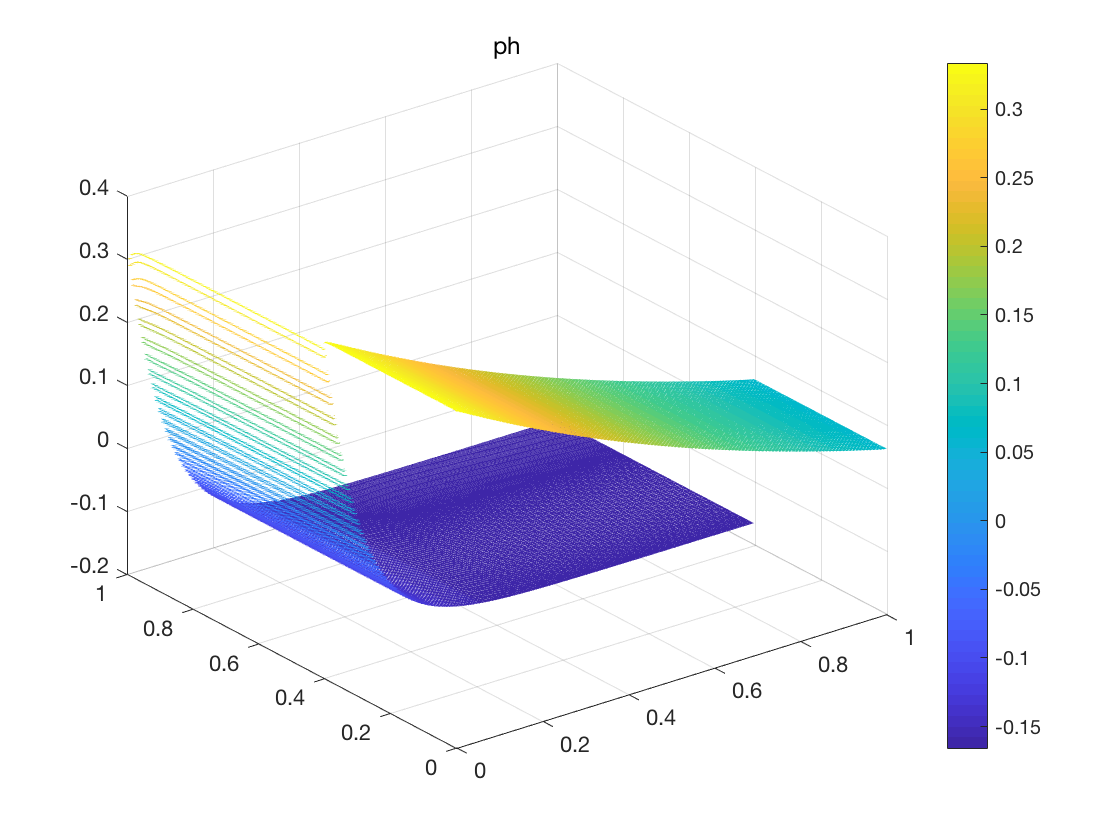} 
\end{minipage}
\tiny\caption{The X-HDG solutions $u_{1h}$(left) and $p_h$(right) at $(\nu_1,\nu_2) = (1,10^{-2}), \alpha_1 = 
\alpha_2 = 0$ and  $128\times 128$ triangular mesh:  Example \ref{exampleline}. }
\label{lineuh10e0_10e-2_0_0}
\end{figure}

\begin{exmp}Curved domain test 1: circular boundary 
\label{circle1}
\end{exmp}

Set $\Omega = \{(x-x_0)^2+(y-y_0)^2 \leq  r^2\}$ in the model problem \eqref{pb2} with a homogeneous 
boundary condition, where $x_0 = y_0 = \frac{1}{2}$ and $r =\frac{\sqrt{3}}{4}$ (Figure \ref{domain_1}). The exact solution is given by
\begin{align*}
u_1(x,y) &= y(x^2+y^2-r^2), \\
u_2(x,y) &= -x(x^2+y^2-r^2), \\
p(x,y) &= \frac{1}{10}(x^3-y^3),
\end{align*}
and the coefficients $(\nu,\alpha)=(1,0), (1,1), (0.01,1).$  

We take $\mathbb{B} = [0,1]^2$ in the X-HDG scheme  \eqref{xhdgscheme_1}. 
Table \ref{circletablek_4} shows that the boundary-unfitted X-HDG method is of optimal convergence rates for the numerical solutions.  Figure \ref{circlefigdomain} plots  the numerical solutions $\bm{u}_h$ and $ p_h$ at $128\times 
128$ triangular mesh.

\begin{figure}[htp]
\centering
\includegraphics[height = 4.7 cm,width=5  cm]{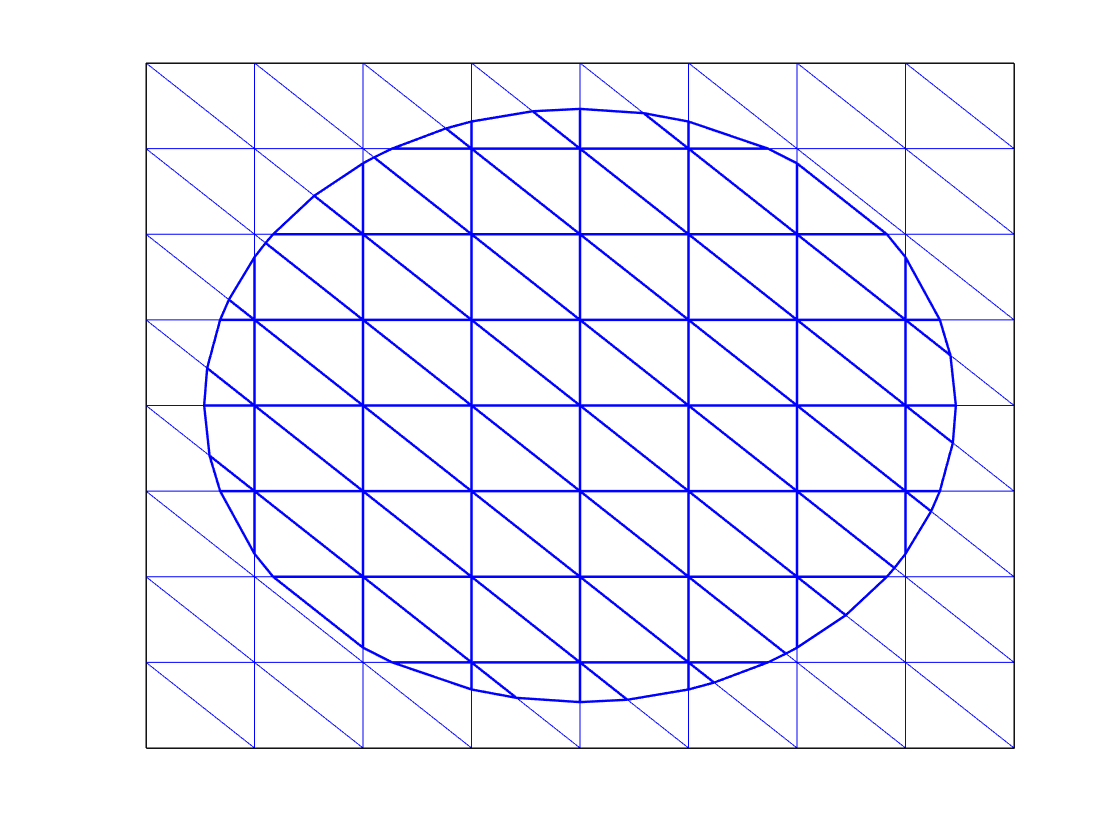} 
\includegraphics[height = 4.7 cm,width=5 cm]{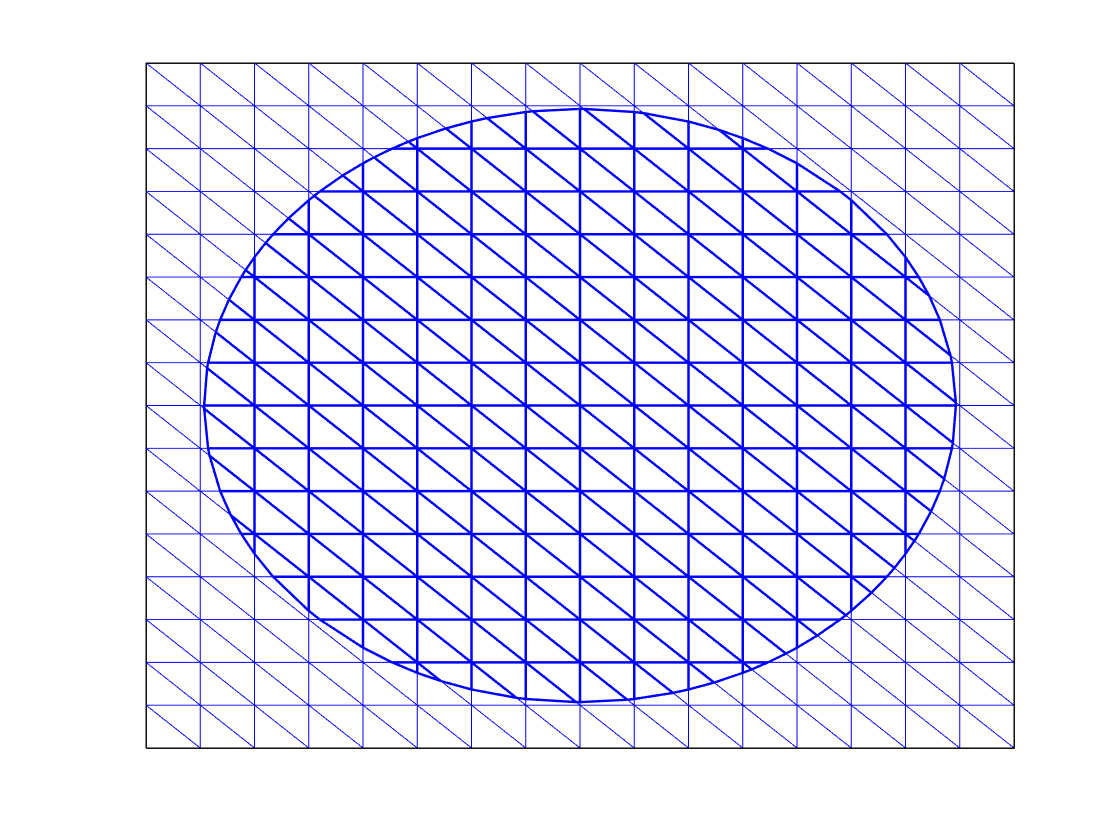} 
\caption{The disc domain at $8\times 8$ and $16\times 16$ meshes: Example  \ref{circle1} .
}\label{domain_1}
\end{figure}

\begin{table}[H]
\normalsize
\caption{History of convergence for the X-HDG scheme \eqref{xhdgscheme_1}: Example \ref{circle1} }\label{circletablek_4}
\centering
\footnotesize
{
\begin{tabular}{p{1cm}<{\centering}|p{1cm}<{\centering}|p{1.4cm}<{\centering}|p{1.45cm}<{\centering}|
p{0.45cm}<{\centering}|p{1.45cm}<{\centering}|p{0.45cm}<{\centering}|p{1.45cm}<{\centering}|p{0.45cm}<{\centering}|
p{1.45cm}<{\centering}|p{0.45cm}<{\centering}}
\hline 
\multirow{2}{*}{$\alpha $}&\multirow{2}{*}{$\nu $}& \multirow{2}{*}{mesh}& 
\multicolumn{2}{c|}{$\frac{\lVert u-u_{h}\rVert_0}{\lVert u\rVert_0}$ }&\multicolumn{2}{c|}{$\frac{\lVert L-L_h \rVert_0}{\lVert L\rVert_0}$}&\multicolumn{2}{c|}{$\frac{\lVert \nabla u-\nabla_hu_h \rVert_0}{\lVert \nabla 
u\rVert_0}$}&\multicolumn{2}{c}{$\frac{\lVert p-p_h \rVert_0}{\lVert p\rVert_0}$}\cr\cline{4-11} 
&&&error&order&error&order&error&order&error&order\cr 
\cline{1-11}
\multirow{5}{*}{0}&\multirow{5}{*}{1}
&$8\times8$ &2.0363E-02 &-- &7.5156E-02 &-- &1.1122E-01 &-- &8.1187E-01 &-- \\
&&$16\times16$ &5.1184E-03 &1.99 &3.8783E-02
 &0.95 &5.5770E-02 &1.00 &2.5358E-01 &1.68 \\
&&$32\times32$ &1.2773E-03 &2.00 &1.9648E-02 &0.98 &2.8352E-02 &0.98 &8.7613E-02 
&1.53 \\
&&$64\times64$ &3.1972E-04 &2.00 &9.8908E-03
 &0.99 &1.4277E-02 &0.99 &3.5517E-02 &1.30 \\
&&$128\times128$ &8.0032E-05 &2.00 &4.9607E-03
 &1.00 &7.1467E-03 &1.00 &1.6197E-02 &1.13 \\
\hline
\multirow{5}{*}{1}&\multirow{5}{*}{1}
&$8\times8$ &2.2913E-02 &-- &8.3725E-02 &-- &1.1101E-01 &-- &1.6159E-00 &-- \\
&&$16\times16$ &6.0720E-03 &1.92 &4.3187E-02
 &0.96 &5.5742E-02 &0.99 &4.5306E-01 &1.83 \\
&&$32\times32$ &1.5612E-03 &1.96 &2.1956E-02 &0.98 &2.8348E-02 &0.98 &1.4085E-01 
&1.69 \\
&&$64\times64$ &3.9616E-04 &1.98 &1.1074E-02
 &0.99 &1.4276E-02 &0.99 &4.9710E-02 &1.50 \\
&&$128\times128$ &9.9862E-05 &1.99 &5.5587E-03
 &0.99 &7.1467E-03 &1.00 &2.1012E-02 &1.24 \\
\hline
\multirow{5}{*}{1}&\multirow{5}{*}{0.01}
&$8\times8$ &2.0363E-02 &-- &2.5740E-01 &-- &2.1070E-01 &-- &3.6030E-01 &-- \\
&&$16\times16$ &5.1184E-02 &1.38 &1.7239E-01
 &0.58 &1.5956E-01 &0.40 &1.4529E-01 &1.31 \\
&&$32\times32$ &1.2773E-03 &1.80 &9.5893E-02 &0.85 &9.3376E-02 &0.77 &5.6350E-02 
&1.37 \\
&&$64\times64$ &3.1972E-03 &1.93 &4.9693E-02
 &0.95 &4.9081E-02 &0.93 &2.4771E-02 &1.19 \\
&&$128\times128$ &8.0032E-04 &1.98 &2.5158E-02
 &0.98 &2.4953E-02 &0.98 &1.1893E-02 &1.06 \\
\hline
\end{tabular}
}
\end{table}
\begin{figure}[H]
\centering
\begin{minipage}[t]{0.3\textwidth}
\includegraphics[height = 4.7 cm,width=5  cm]{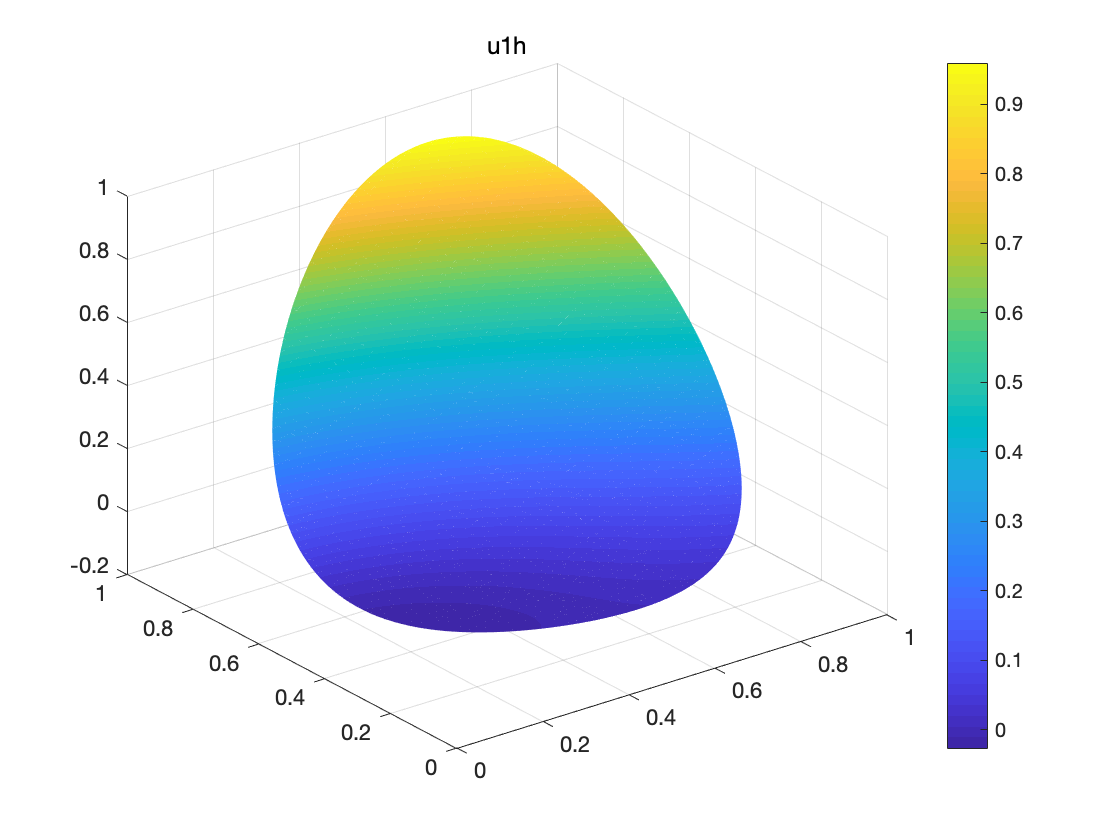} 
\end{minipage}
\begin{minipage}[t]{0.3\textwidth}
\includegraphics[height =4.5 cm,width=5  cm]{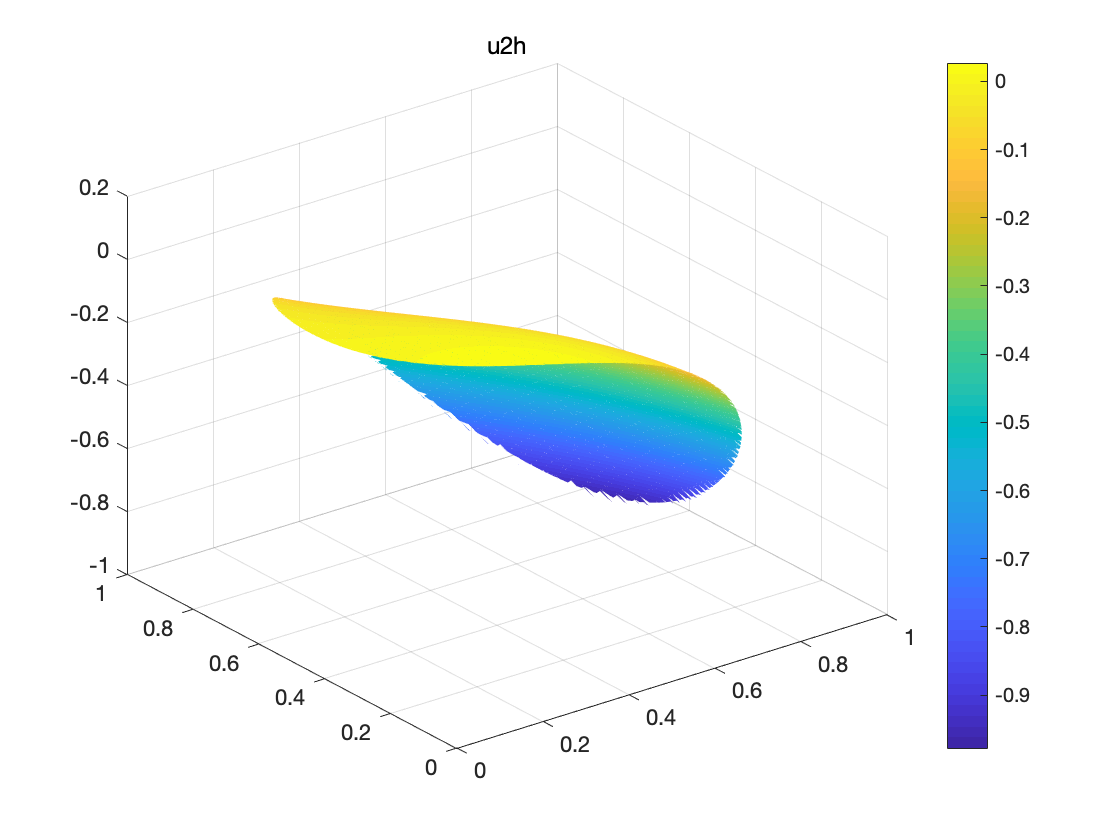}
\end{minipage}
\begin{minipage}[t]{0.3\textwidth}
\includegraphics[height = 4.5 cm,width=5  cm]{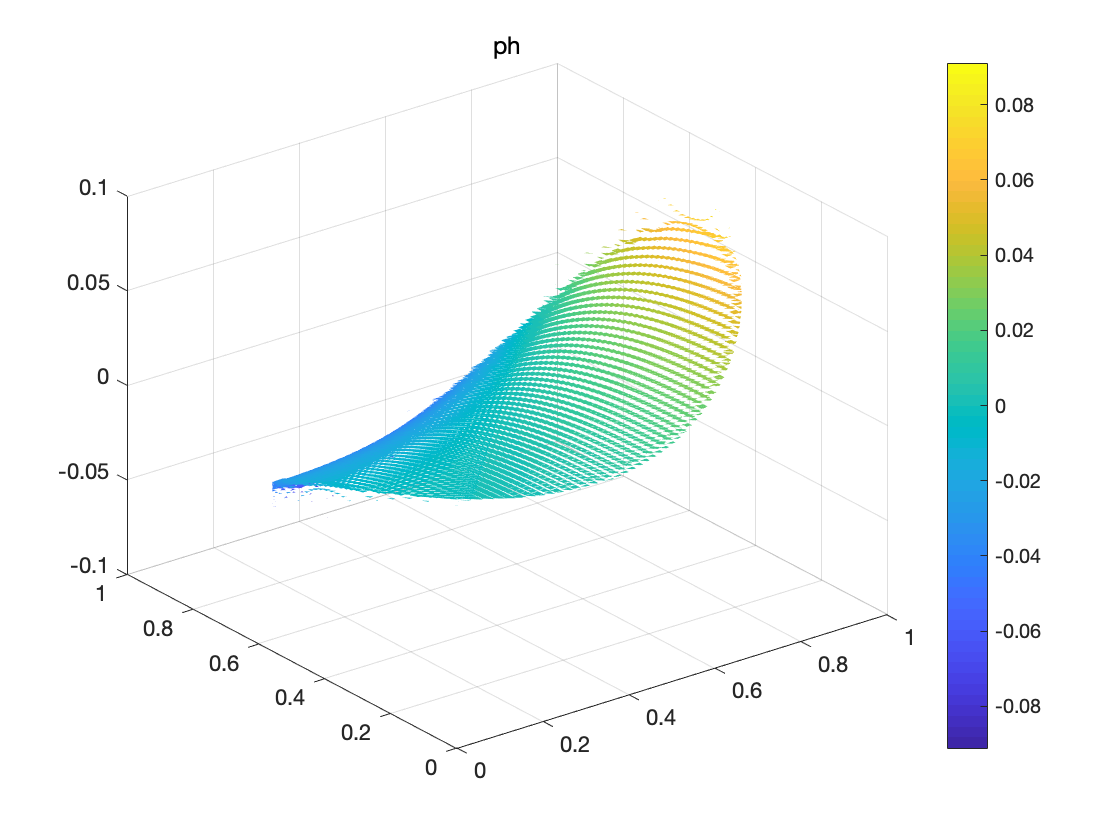}
\end{minipage}
\tiny\caption{The X-HDG solutions $u_{1h}$(left), $u_{2h}$(middle) and $p_h$(right) at $\nu = 
1, \alpha = 0$ and $128\times 128$ triangular mesh: Example \ref{circle1}. } \label{circlefigdomain}
\end{figure}

\begin{exmp}Curved domain test 2:  five-star shaped boundary \label{fivestar1}
\end{exmp}

Let $\Omega$ be a  five-star shaped domain with boundary $ \Gamma =\{(r,\theta): \rho(r,\theta) = 0,\ 0\leq \theta < 2\pi\}$ in 
\eqref{pb2}, where $\rho(r,\theta) = r-\frac{\sqrt{3}}{4}-\frac{1}{10}\sin(5\theta+\frac{\pi}{2})$, $r = \sqrt{x^2+y^2}$.
%
The exact solution of  \eqref{pb2} is given by
\begin{align*}
u_1(x,y) = x^2y, \quad 
u_2(x,y) = -xy^2, \quad
p(x,y) = \frac{1}{3}(x^3-y^3),
\end{align*}
%
and the coefficients $(\nu,\alpha)=(1,0), (1,1), (0.01,1).$  

We take $\mathbb{B} = [-1,1]^2$ in the X-HDG scheme  \eqref{xhdgscheme_1}. 
Table \ref{fivestar1_tab} shows that the boundary-unfitted X-HDG method is of optimal convergence rates for the numerical solutions.  Figure \ref{fivestarfigdomain} plots  the numerical solutions $\bm{u}_h$ and $ p_h$ at $128\times 
128$ triangular mesh.

\begin{figure}[htp]
\centering
\includegraphics[height = 4.7 cm,width=5  cm]{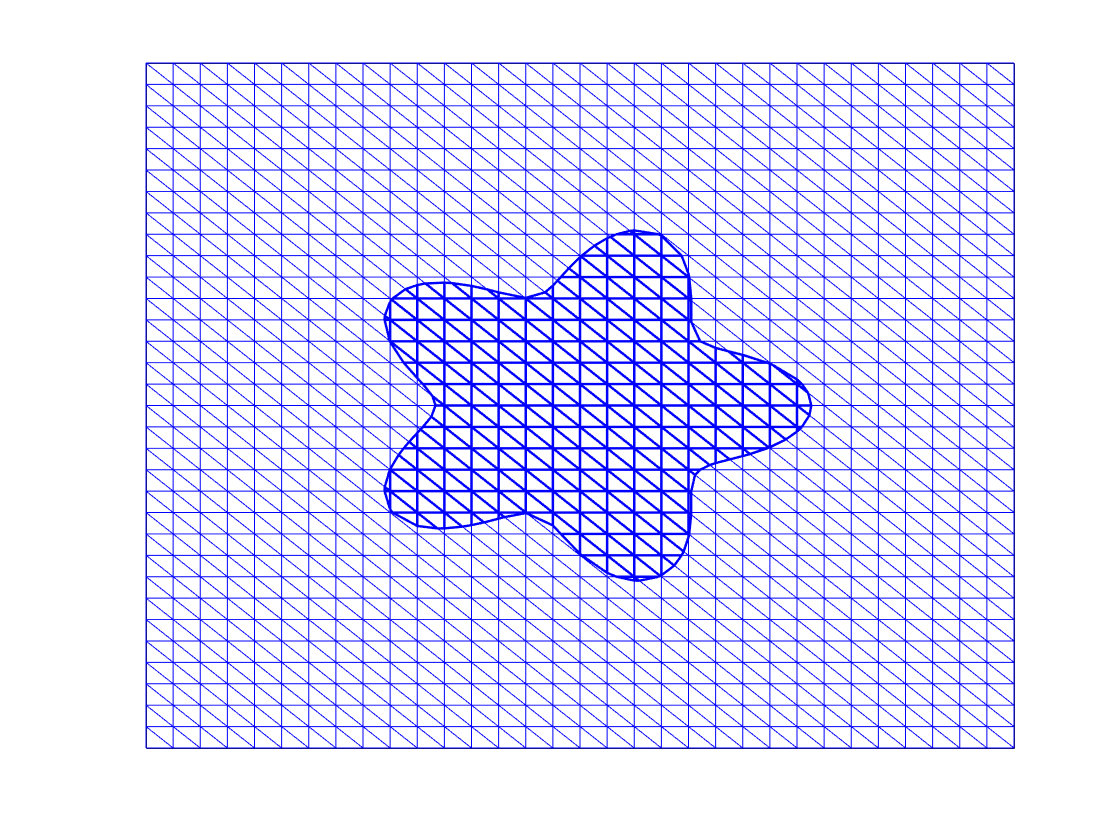} 
\caption{The  five-star shaped domain  with $32\times 32$ mesh: Example \ref{fivestar1}. 
}
\label{domain_6}
\end{figure}

\begin{table}[H]
\normalsize
\caption{History of convergence for the X-HDG scheme \eqref{xhdgscheme_1}: Example \ref{fivestar1} }\label{fivestar1_tab}
\centering
\footnotesize
{
\begin{tabular}{p{1cm}<{\centering}|p{1cm}<{\centering}|p{1.4cm}<{\centering}|p{1.45cm}<{\centering}|
p{0.45cm}<{\centering}|p{1.45cm}<{\centering}|p{0.45cm}<{\centering}|p{1.45cm}<{\centering}|p{0.45cm}<{\centering}|
p{1.45cm}<{\centering}|p{0.45cm}<{\centering}}
\hline 
\multirow{2}{*}{$\alpha $}&\multirow{2}{*}{$\nu $}& \multirow{2}{*}{mesh}& 
\multicolumn{2}{c|}{$\frac{\lVert u-u_{h}\rVert_0}{\lVert u\rVert_0}$ }&\multicolumn{2}{c|}{$\frac{\lVert L-L_h \rVert_0}{\lVert L\rVert_0}$}&\multicolumn{2}{c|}{$\frac{\lVert \nabla u-\nabla_hu_h \rVert_0}{\lVert \nabla 
u\rVert_0}$}&\multicolumn{2}{c}{$\frac{\lVert p-p_h \rVert_0}{\lVert p\rVert_0}$}\cr\cline{4-11} 
&&&error&order&error&order&error&order&error&order\cr 
\cline{1-11}
\multirow{4}{*}{0}&\multirow{4}{*}{1}
&${16\times16}$ &2.0455E-02 &-- &1.0942E-01 &-- &1.2095E-01 &-- &5.2046E-01 &-- \\
&&${32\times32}$ &5.9306E-03 &1.79 &5.5978E-02
 &0.97 &6.1824E-02 &0.97 &2.5200E-01 &1.05 \\
&&${64\times64}$ &1.5742E-03 &1.91 &2.8320E-02 &0.98 &3.1320E-02 &0.98 &1.2398E-01 
&1.02 \\
&&${128\times128}$ &3.9680E-04 &1.99 &1.4260E-02
 &0.99 &1.5792E-02 &0.99 &6.1617E-02 &1.01 \\
\hline
\multirow{4}{*}{1}&\multirow{4}{*}{1}
&${16\times16}$ &2.0401E-02 &-- &1.0941E-01 &-- &1.2090E-01 &-- &5.2097E-01 
&-- \\
&&${32\times32}$ &5.9174E-03 &1.79 &5.5976E-02
 &0.97 &6.1817E-02 &0.97 &2.5209E-01 &1.05 \\
&&${64\times64}$ &1.5714E-03 &1.91 &2.8319E-02 &0.98 &3.1319E-02 &0.98 &1.2400E-01 
&1.02 \\
&&${128\times128}$ &3.9615E-04 &1.99 &1.4259E-02
 &0.99 &1.5791E-02 &0.99 &6.1619E-02 &1.01 \\
\hline
\multirow{4}{*}{1}&\multirow{4}{*}{0.01}
&${16\times16}$ &4.3717E-01 &-- &9.8775E-01 &-- &9.8686E-01 & -- &1.9356E-01 &-- 
\\
&&${32\times32}$&1.3610E-01 &1.68 &5.8835E-01
 &0.75 &5.9952E-01 &0.72 &7.5767E-02 &1.35 \\
&&${64\times64}$ &3.6656E-02 &1.89 &3.1398E-01 &0.91 &3.2368E-01 &0.89 &3.2769E-02 
&1.21 \\
&&${128\times128}$ &9.3985E-03 &1.96 &1.6092E-01
 &0.96 &1.6672E-01 &0.96 &1.5592E-02 &1.07 \\
\hline
\end{tabular}
}
\end{table}
\begin{figure}[ht]
\centering
\begin{minipage}[t]{0.3\textwidth}
\includegraphics[height = 4.7 cm,width=5  cm]{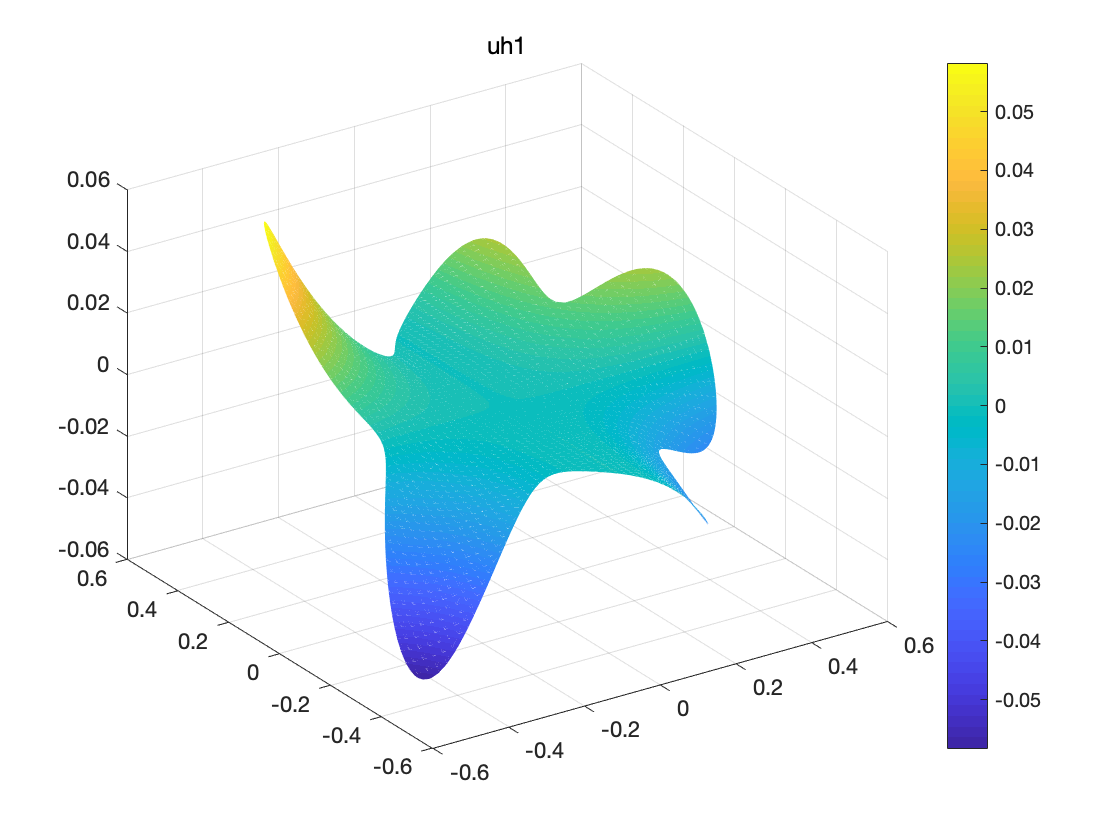} 
\end{minipage}
\begin{minipage}[t]{0.3\textwidth}
\includegraphics[height = 4.7 cm,width=5 cm]{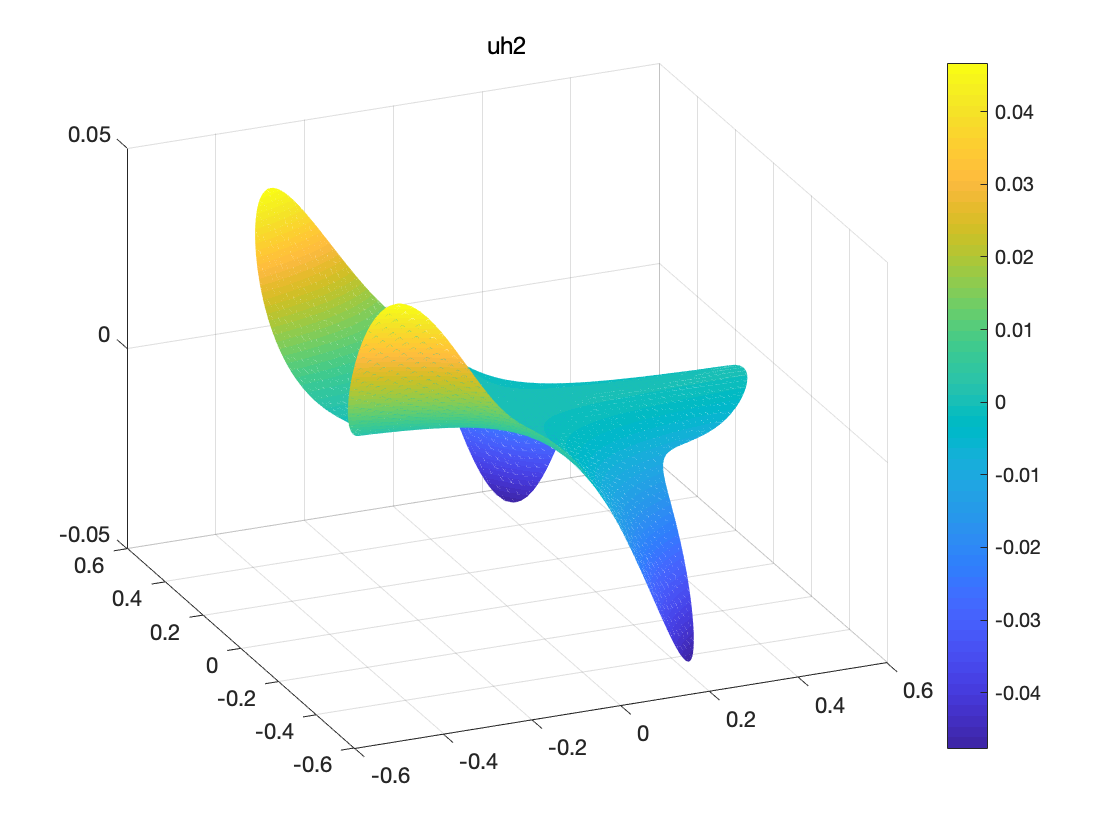}
\end{minipage}
\begin{minipage}[t]{0.3\textwidth}
\includegraphics[height = 4.7 cm,width=5  cm]{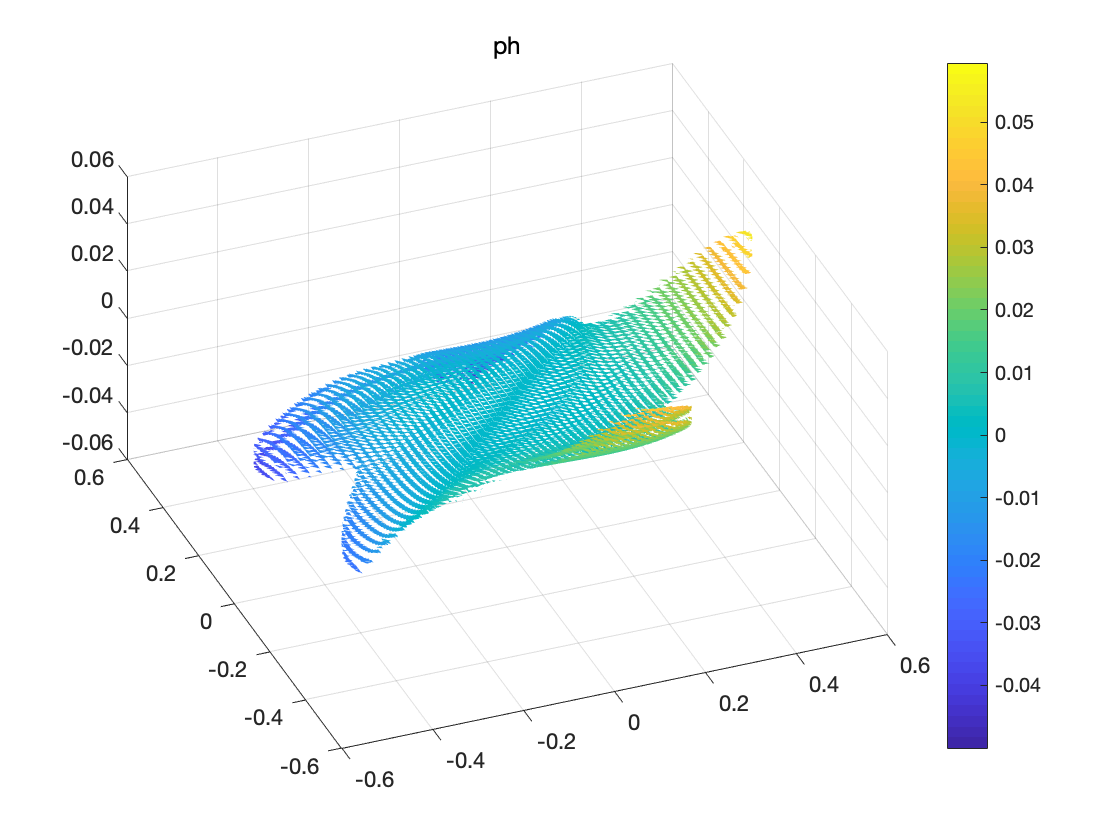}
\end{minipage}
\tiny\caption{The X-HDG solutions $u_{1h}$(left), $u_{2h}$(middle) and $p_h$(right) at $\nu 
= 1, \alpha = 0$ and $128\times 
128$ triangular mesh: Example \ref{fivestar1}.
}\label{fivestarfigdomain}
\end{figure}

	\section{Conclusions}
	 
	 For  the Darcy-Stokes-Brinkman interface problems, the proposed    low order   interface-unfitted X-HDG method is of optimal convergence and 
  applies to curved domains with boundary-unfitted meshes.   Numerical experiments have demonstrated the  performance of the method.	
	\footnotesize
	\bibliographystyle{plain}
	\bibliography{yihui_reference_papers,yihui_reference_books}
	
\end{document}